\documentclass[a4paper, 12pt, reqno]{amsart}

\usepackage{accents}
\usepackage{amsfonts, amsmath, amssymb, amsthm}
\usepackage{cancel}
\usepackage{color, framed}
\usepackage{graphicx}
\usepackage{mathtools}
\usepackage{relsize} 
\usepackage{tikz} 
\usepackage{tikz-cd}
\usepackage{verbatim}
\usepackage[all]{xy}
\usepackage{hyperref}

\allowdisplaybreaks 

\definecolor{shadecolor}{gray}{0.8}

\newcommand{\flplus} 
{
\hspace{0.1cm}
\begin{tikzpicture}[baseline=-0.582ex]
    \draw [line width=0.24pt](-0.1129, 0) -- (0.1129, 0) -- (0, 0) -- (0, 0.1129) -- (0, -0.1129) arc (270:90:0.1129) -- (0, 0);
\end{tikzpicture}
\hspace{0.1cm}
}

\newcommand{\frplus} 
{
\hspace{0.1cm}
\begin{tikzpicture}[baseline=-0.582ex]
    \draw [line width=0.24pt](-0.1129, 0) -- (0.1129, 0) -- (0, 0) -- (0, 0.1129) -- (0, -0.1129) arc (-90:90:0.1129) -- (0, 0);
\end{tikzpicture}
\hspace{0.1cm}
}

\AtEndDocument{\vfill\eject\batchmode} 
\usepackage{amssymb}
\usepackage[T1]{fontenc}
\DeclareSymbolFont{script}{U}{eus}{m}{n}
\DeclareMathSymbol{\Wedge}{0}{script}{"5E}

\usepackage{t1enc}

\def\Bbb{\mathbb}
\def\Cal{\mathcal}

\newcommand{\Rho}{\mathsf{P}}

\DeclareMathOperator{\G}   {G}
\DeclareMathOperator{\GL}  {GL}

\DeclareMathOperator{\Hom} {Hom}
\DeclareMathOperator{\PSL} {PSL}

\DeclareMathOperator{\SL}  {SL}
\DeclareMathOperator{\SO}  {SO}
\DeclareMathOperator{\Sp}  {Sp}
\DeclareMathOperator{\SU}  {SU}

\DeclareMathOperator{\U}   {U}

\newcommand{\wt}[1]{\widetilde{#1}}

\newcommand{\hook}{\, \lrcorner \,}

\newcommand{\bbC}{\mathbb{C}}

\newcommand{\bbJ}{\mathbb{J}}

\newcommand{\bbP}{\mathbb{P}}
\newcommand{\bbR}{\mathbb{R}}

\newcommand{\bbT}{\mathbb{T}}
\newcommand{\bbV}{\mathbb{V}}

\newcommand{\eps}   {\epsilon}

\newcommand{\mbc}{\mathbf{c}}
\newcommand{\mbg}{\mathbf{g}}

\newcommand{\mbp}{\mathbf{p}}
\newcommand{\mbq}{\mathbf{q}}

\newcommand{\mcA}{\mathcal{A}}
\newcommand{\mcD}{\mathcal{D}}
\newcommand{\mcE}{\mathcal{E}}
\newcommand{\mcF}{\mathcal{F}}
\newcommand{\mcG}{\mathcal{G}}
\newcommand{\mcL}{\mathcal{L}}

\newcommand{\mcO}{\mathcal{O}}
\newcommand{\mcS}{\mathcal{S}}
\newcommand{\mcT}{\mathcal{T}}

\newcommand{\mfg}  {\mathfrak{g}}

\newcommand{\mfu}  {\mathfrak{u}}

\newcommand{\wtG}{\wt G}
\newcommand{\wtM}{\wt M}
\newcommand{\wtP}{\wt P}
\newcommand{\wtmcT}{\wt\mcT}
\newcommand{\wtRho}{\wt\Rho}
\newcommand{\wtnabla}{\wt\nabla}

\def\R{\mathbb{R}}
\newcommand{\na}{\nabla}

\newcommand{\ce}{{\Cal E}}

\newcommand{\cT}{{\mathcal T}}

\newcommand{\rpl}                         
{\mbox{$
\begin{picture}(12.7,8)(-.5,-1)
\put(0,0.2){$+$}
\put(4.2,2.8){\oval(8,8)[r]}
\end{picture}$}}

\def\cD{\mathcal{D}}

\def\cV{\mathcal{V}}

\def\cH{\mathcal{H}}

\def\R{\mathbb{R}}

\newcommand{\si}{\sigma}
\newcommand{\id}{\operatorname{id}}

\newcommand{\End}{\operatorname{End}}

\newtheorem{theorem}{Theorem}[section]
\newtheorem{lemma}[theorem]{Lemma}
\newtheorem{proposition}[theorem]{Proposition}
\newtheorem{corollary}[theorem]{Corollary}

\newtheorem*{theoremA}{Theorem A}
\newtheorem*{theoremB}{Theorem B}
\newtheorem*{theoremC}{Theorem C}
\newtheorem*{theoremD}{Theorem D}

\theoremstyle{definition}
\newtheorem{definition}[theorem]{Definition}

\theoremstyle{remark}
\newtheorem{remark}[theorem]{\rm\bf Remark}
\newtheorem*{definition*}{\rm\bf Definition}

\newcommand{\nn}[1]{(\ref{#1})}

\newcommand{\Ric}{\operatorname{Ric}}

\newcommand{\vol}{\displaystyle\mathlarger{\boldsymbol{\epsilon}}}

\def\sideremark#1{\ifvmode\leavevmode\fi\vadjust{\vbox to0pt{\vss
 \hbox to 0pt{\hskip\hsize\hskip1em
 \vbox{\hsize3cm\tiny\raggedright\pretolerance10000
  \noindent #1\hfill}\hss}\vbox to8pt{\vfil}\vss}}}%

\author{A.\ R. Gover, K. Neusser, \& T.\ Willse}

\title{Projective geometry of Sasaki--Einstein structures and their compactification}
\begin{document}

\address{
A.R.G.: Department of Mathematics\\
  The University of Auckland\\
  Private Bag 92019\\
  Auckland 1142\\
  New Zealand\\
K.N.: Department of Mathematics and Statistics\\
Masaryk University\\
Kotl\'a\v rsk\'a 267/2 \\
611 37 Brno\\
Czech Republic\\
T.W.: Fakult\"{a}t f\"{u}r Mathematik\\
	Universit\"{a}t Wien\\
	Oskar-Morgenstern-Platz 1\\
	1090 Wien\\
	Austria\\
}

\email{r.gover@auckland.ac.nz}
\email{neusser@math.muni.cz}
\email{travis.willse@univie.ac.at}

\subjclass[2010]{Primary 53A20, 53B10, 53C25, 53C29; Secondary 35Q76, 53A30, 53A40, 53C55}
\keywords{Projective differential geometry, Sasaki--Einstein manifolds, holonomy reductions of Cartan connection, conformal geometry, special contact geometries, K\"ahler manifolds, CR geometry}
\begin{abstract}
We show that the standard definitions of Sasaki structures have
elegant and simplifying interpretations in terms of projective
differential geometry. For Sasaki--Einstein structures we use
projective geometry to provide a resolution of such structures into
geometrically less rigid components; the latter elemental components
are separately, complex, orthogonal, and symplectic holonomy
reductions of the canonical projective tractor/Cartan connection.
This leads to a characterisation of Sasaki--Einstein structures as
projective structures with certain unitary holonomy reductions.  As an
immediate application, this is used to describe the projective
compactification of indefinite (suitably) complete non-compact
Sasaki--Einstein structures and to prove that the boundary at infinity
is a Fefferman conformal manifold that thus fibres over a
nondegenerate CR manifold (of hypersurface type). We prove that this
CR manifold coincides with the boundary at infinity for the
c-projective compactification of the K\"ahler--Einstein manifold that
arises, in the usual way, as a leaf space for the defining Killing field
of the given Sasaki--Einstein manifold.  A procedure for constructing
examples is given. The discussion of symplectic holonomy reductions of
projective structures leads us moreover to a new and simplifying
approach to contact projective geometry. This is of independent
interest and is treated in some detail.

\end{abstract}

\thanks{
The authors gratefully acknowledge support from several sources:
A.R.G.: the Royal Society of New Zealand, via Marsden Grants 13-UOA-018 and 16-UOA-051;
A.R.G., K.N., and T.W.: the Simons Foundation, via grant 346300 for IMPAN, and the Polish Government, via the matching 2015-2019 Polish MNiSW fund;
K.N. and T.W.: the Australian Research Council;
K.N.: the Czech Grant Agency, via grant P201/12/G028;
T.W.: the Austrian Science Fund (FWF), via project P27072--N25.
}

\maketitle
\pagestyle{myheadings}

\markboth{Gover, Neusser, Willse}{Projective geometry of Sasaki--Einstein structures and their compactification}

\section{Introduction}
A (pseudo-)Riemannian manifold $(M, g)$ is Sasaki if its standard
metric cone is (pseudo-)K\"ahler. It is Sasaki--Einstein if the metric
$g$ is also Einstein, which is well-known to be the case if and only
if the metric cone is (pseudo-)K\"ahler and Ricci-flat.  These structures
provide a natural intersection of Riemannian, contact and CR
structures, and are intimately linked to K\"ahler and quaternionic
K\"ahler geometry.  Being an important meeting place for all these
geometric structures, Sasaki manifolds play a distinguished role in
Riemannian geometry and hence have been intensively studied; see
e.g. \cite{AF, BoyerGalicki, BG-book, BT-F, MR, Sparks} and the references
therein. Moreover, they have turned out to be important in some
constructions in theoretical physics in the context of string theory
and the AdS/CFT correspondence \cite{FS,GMSW,HS,LPS}.

Recall that a projective structure on a manifold $M$ is an equivalence
class $\mbp$ of torsion-free connections on its tangent bundle $TM$
that share the same unparametrised geodesics. Evidently, any
(pseudo-)Riemannian metric $g$ gives rise to a projective structure
via its Levi-Civita connection $\nabla$.  Apart from the definition of
Sasaki structures via the metric cone, there is a well-known more
direct equivalent characterisation of these structures (see Definition
\ref{definition:Sasaki-structure}), which involves however a solution
$k^a\in\Gamma(TM)$ to a system of two overdetermined partial
differential equations, given by the Killing equation
$\nabla_{a}k_b+\nabla_b k_a=0$ and the equation $\nabla_a\nabla_bk^c
=-g_{ab} k^c + \delta^c{}_a k_b$. Suitably interpreted, the Killing
equation of a (pseudo-) Riemannian manifold $(M,g)$ exhibits
projective invariance and we shall see that the above second order
equation is linked to the equation of (infinitesimal) projective
symmetries of $(M,g)$.  In the present article we show that these
observations are part of a picture that reveals projective geometry as
a natural unifying framework for treating many aspects of Sasaki and
Sasaki--Einstein geometry. In particular, we shall see that the
projective viewpoint on Sasaki structures, which we develop and investigate here,
leads to natural geometric compactifications of some complete
non-compact (indefinite) Sasaki--Einstein structures.

In Section \ref{Section: Sasaki} we start by recalling the various
equivalent standard definitions of Sasaki structures before
recasting them, in Theorem \ref{Jpar}, in terms of objects fundamental
to projective geometry and its treatments.  We then prolong one of the
defining equation of a Sasaki structure, namely the Killing
equation, and show this yields a projectively invariant linear
connection $\nabla^{\textrm{prol}}$, on a certain vector bundle over
$(M, g)$, whose parallel sections are in bijection to the solutions of
the Killing equation (see Proposition \ref{prolongprop}). In fact, we
show in Section \ref{subsection:tractor-2-forms} that this connection
results from a simple projectively invariant modification of the
so-called normal tractor connection (which is equivalent to
the normal Cartan connection) of the projective structure induced by
$g$ on $M$.  In its standard presentation, the tractor connection is a
canonical connection $\nabla^{\mcT}$ on a certain vector bundle $\mcT$
of rank $n+2$, over any $n+1$-dimensional projective manifold
$(M,\mbp)$, and this connection  is flat if and only if $(M,\mbp)$ is locally
projectively equivalent to the flat projective structure on
$\bbR^{n+1}$. We recall its definition and basic properties in Section
\ref{tr}. In the case of Sasaki--Einstein structures we show that the
defining Killing vector field $k$ lifts to a parallel section of
$\nabla^{\textrm{prol}}$ that satisfies in addition some algebraic
equations that makes this section already parallel for the tractor
connection.  More specifically, we show in Section \ref{Section: Proof
  of Theorem A} that an Einstein--Sasaki structure induces a
$\nabla^{\mcT}$-parallel Hermitian structure on its projective tractor bundle
$\mcT$, recovering a result by Armstrong \cite{ArmstrongP1, ArmstrongP2}: \begin{theoremA} Suppose  $(M, g, k)$ is a Sasaki--Einstein manifold of signature $(2 p - 1, 2 q)$ with Levi-Civita connection $\nabla$. Then the projective tractor bundle $(\mcT, \nabla^{\mcT})$
of the induced projective structure $(M, [\nabla])$ admits a parallel Hermitian structure given by:
\begin{itemize}
\item a parallel complex
structure $\bbJ \in \Gamma(\End \mcT)$ on $\mcT$
\item a parallel metric $h \in \Gamma(S^2
\mcT^*)$ of signature $(2 p, 2 q)$ on $\mcT$ that is Hermitian with respect to $\bbJ$: $h(\,\cdot\,,\,\cdot\,) = h(\bbJ\,\cdot\,,\bbJ\,\cdot\,).$
\end{itemize}
If $M$ is in addition simply connected, then $(\mcT, \bbJ)$ (viewed as a complex vector bundle) also admits a parallel
tractor complex volume form $\vol_{\bbC}\in\Gamma(\Wedge^{p+q}_{\bbC}\mcT)$, where the subscript $\bbC$ indicates that the wedge product is taken over $\bbC$.
\end{theoremA}
A parallel Hermitian structure $(h, \bbJ)$ on the tractor bundle $(\mcT, \nabla^{\mcT})$ of a projective structure is equivalent to a reduction of the holonomy of
the tractor connection to the unitary group $\U(p, q)$ and the existence of a parallel tractor complex volume form gives rise to a further reduction of its holonomy to $\SU(p, q)$.
So Theorem A can be also rephrased as: The projective tractor connection of a Sasaki--Einstein manifold of signature $(2 p - 1, 2 q)$ has restricted holonomy contained in $\SU(p, q)$.

Theorem A naturally raises the question of the converse: Given a projective manifold $(M,\mbp)$ admitting a holonomy reduction of its tractor connection to $\SU(p, q)$ what is the nature of the underlying geometry on $M$?
Armstrong observed in \cite{ArmstrongP1, ArmstrongP2} that on the connected components of some open dense set of $M$ one obtains Sasaki--Einstein structures. In Theorems B, C, and D below, which will be proved in Section \ref{assembly}, we give a more refined answer, using the theory of holonomy reductions of Cartan connections and curved orbit decompositions from \cite{CGHjlms,CGH}.

In the following theorem and throughout the article, a {\em hypersurface} means
an embedded submanifold of codimension $1$.

\begin{theoremB}\label{B}
 Let $(M, \mbp)$ be a projective manifold of dimension at least $5$ equipped with a parallel tractor Hermitian structure $(h,\bbJ)$ on its tractor bundle $(\mcT, \nabla^{\mcT})$.

Then, $M$ is of odd dimension $2m + 1$ and $h$ has signature $(2p,2q)$ for some $p,q\in \mathbb{Z}_{\geq 0}$. Furthermore, $\bbJ$
 determines a nowhere vanishing vector field $k\in\Gamma(TM)$, and
  $M$ is stratified into a disjoint union of submanifolds
  $$M = M_+ \cup M_0 \cup M_-,
  $$
  according to the strict sign of $\tau := h(X,X)$, where $X$ is the \textit{canonical tractor}
\nn{euler} defined in Section \ref{tr}.
  The components $M_{\pm}$ and $M_0$ are
  each equipped with a geometry canonically determined by $(M, \mbp,
  h, \bbJ)$ as follows.
\begin{enumerate}
	\item The submanifolds $M_{\pm}$ are open and (if nonempty)
          are respectively equipped with Sasaki--Einstein structures
          $(g_{\pm}, k)$ with Ricci curvature $\emph{Ric}_{\pm}= 2 m g_{\pm}$
          and $k$ its defining Killing field; $g_+$ has signature $(2 p - 1, 2 q)$, and
          $g_-$ has signature $(2 q - 1, 2 p)$. The metrics $g_{\pm}$
          are compatible with the projective structure $\mbp$ in the
          sense that their respective Levi-Civita connections
          $\nabla^{\pm}$ satisfy $\nabla^{\pm} \in
          \mbp\vert_{M_{\pm}}$.
	\item The submanifold $M_0$ is (if nonempty) a smooth separating
          hypersurface and is equipped with an oriented (local)
          Fefferman conformal
          structure $\bold{c}$ of signature $(2 p - 1, 2 q - 1)$.

     \end{enumerate}
If $h$ is definite, the stratification is trivial, that is,
if $h$ is $\pm$-definite then, respectively, $M = M_\pm$.
\end{theoremB}

In fact, Theorem $B$ also holds in dimension $3$ with the adjustment
that $M_0$ is canonically equipped with a conformal M\"obius structure
equipped with a parallel complex structure on its tractor bundle. In
this dimension the assumptions of Theorem B however imply that
$(M,\mbp)$ is projectively flat. Note also that Theorems A and B
together immediately give a projective tractor characterisation of
Sasaki--Einstein structures of dimension $2 m + 1$: They are precisely
projective structures equipped with a reduction of the restricted
holonomy group of $\nabla^{\mcT}$ to $\SU(p, q)$, $p + q = m + 1$,
such that the canonical tractor $X$ is nowhere isotropic with respect
to the parallel tractor metric $h$. If $pq = 0$, that latter condition
is automatic and hence the converse of Theorem A holds without
additional assumption.  This characterisation of Sasaki--Einstein
resembles the one via the metric cone mentioned at the beginning and
indeed there is an intimate relation: the tractor connection of a
projective manifold $(M, \mbp)$ can also be alternatively viewed as a
Ricci-flat connection on the tangent bundle of a certain cone
manifold, called the \emph{Thomas cone}, over $(M,\mbp)$; see
e.g.\,\cite{CGHjlms, csbook}.  The tractorial characterisation of
Sasaki--Einstein structures translates then, in this picture, to the
statement that these structures are precisely those projective structures
whose Thomas cone admits a certain reduction of its restricted
holonomy to $\SU(p, q)$, and this  enables the identification of the Thomas cone
with the metric cone of a Einstein--Sasaki structure.

An immediate advantage of the projective characterisation of
Einstein--Sasaki structures is the availability of a resolution of these structures into
elemental geometric components. This resolution is unavailable in
the metric cone characterisation, because in that case the metric is
built into the construction.  These elemental components are (certain)
holonomy reductions of the projective tractor connection to
$\textrm{SO}(2p,2q)$, $\textrm{Sp}(2m+2,\bbR)$ and
$\textrm{GL}(m+1,\bbC)\cap \SL(2m+2,\bbR)$ that are compatible in the
sense that they determine a reduction to the unitary group
\begin{equation}\label{keydis}
\U(p, q) = \SO(2p, 2q) \cap \Sp(2m + 2, \bbR)\cap \GL(m + 1, \bbC).
\end{equation}
Since the intersection of any two groups on the right hand side equals
the unitary group it follows that two compatible reductions,
corresponding to any two of these groups, is sufficient to determine
the third reduction.  Depending on which of the reductions is
emphasised, one gets more emphasis on the metric, respectively
symplectic, point of view on Einstein--Sasaki geometry. Switching
between these can be of great use, in the same way as for the
analogous resolution of K\"ahler manifolds into elemental geometric
components. Note however that, although clearly analogous, the latter is
more immediately obvious because it involves only holonomy reductions
of an affine connection on the tangent bundle.

In Section \ref{hol-sec} we analyse the geometric implications of
these three projective holonomy reductions separately. We start first by discussing, in Section \ref{SOred}, holonomy reductions of projective
structures to $\SO(2p, 2q)$.
We do this rather briefly
since there exist already detailed accounts of this in the literature
\cite{CGcompact, CGHjlms, CGH}.
In Section \ref{Sp-reduction} we then
study projective holonomy reductions to the group $\Sp(2m + 2,
\bbR)$ given by parallel tractor symplectic forms. We first show in
Theorem \ref{projective_structure_with_parallel_symplectic_tractor}
that such a reduction induces on the underlying manifold $M$ an
interesting geometric structure known as a torsion-free contact
projective structure (see Definition \ref{def:proj_cont}).
Conversely, it was shown by Fox in \cite{Fox} that any contact
projective structure gives rise to a distinguished projective
structure on the same base manifold, which in the case of
torsion-freeness leads to a (in our language) projective structure
equipped with a holonomy reduction of its tractor connection to
$\Sp(2m + 2, \bbR)$. We shall give a more direct and for our current
purposes technically less demanding proof of his result in Theorem
\ref{Fefferman_type_construction} (see also Corollary
\ref{corollary:converse-of-symplectic-reduction} which establishes the
precise converse of Theorem
\ref{projective_structure_with_parallel_symplectic_tractor}). Our
considerations leading to the proof of Theorem
\ref{Fefferman_type_construction} will also shed some new light on
some key aspects of contact projective geometry and hence will be of
independent interest. In Section \ref{J-red-sect} we will be concerned
with oriented projective manifolds $(M,\mbp)$ equipped with a holonomy
reduction of its tractor bundle to $\GL(m + 1, \bbC)\cap
\SL(2m+2,\bbR)$ given by a parallel tractor complex structure
$\bbJ$. This will lead us to another important geometry, namely
c-projective geometry,---a complex analogue of projective geometry,
which has in recent years celebrated a revival; see the monograph
\cite{CEMN} for the history of the subject and its current
developments. We first show (see Proposition
\ref{proposition:complex-single-curved-orbit}) that $\bbJ$ determines
a nowhere vanishing vector field $k$ on $M$ (cf.\,also Theorem B),
which hence defines a foliation on $M$. Then we show that, in case $k$
is a projective symmetry of $(M,\mbp)$, any sufficiently small local
leaf space $\widetilde M$ of this foliation inherits from $(\mbp,
\bbJ)$ a c-projective structure, consisting of a complex structure
$\tilde J$ and a c-projective equivalence class $\tilde\mbp$ of
connections (see Definition \ref{Definition: c-projective}), and that
conversely (see Theorem \ref{theorem:c-projective-to-projective}) any
c-projective structure arises in this way. Theorem
\ref{theorem:projective-to-c-projective} and
\ref{theorem:c-projective-to-projective} were first proved on the
level of the Thomas cone, respectively tractors, by Armstrong in the
inspiring works \cite{ArmstrongThesis, ArmstrongP1, ArmstrongP2} (and he discusses there briefly the symplectic reduction).  We give here nevertheless
proofs of these theorems to provide a reasonably complete and, most
importantly, uniform development. Our treatment is also different as
it relates the projective and c-projective structure more directly
(rather than via their respective tractor connections). Apart from the
desire for clarity here, part of our motivation for the detailed treatment
is to lay essential foundations for planned future work on
non-Einstein Sasaki structures. This motivation has also driven our
particular and detailed treatment of the symplectic reduction in
Section \ref{Sp-reduction}. In Section \ref{assembly} we then put
these results together to establish Theorem B in a way that exposes
clearly the geometric dependencies and interactions of the involved
geometric structures and that exploits projective geometry as a natural unifying
framework to study them.

 Theorem B leads moreover to another significant gain from viewing
 Einstein--Sasaki structures as arising via appropriate holonomy
 reductions of projective structures. Namely it yields natural
 geometric compactifications based on projective geometry, as
 introduced in \cite{CGcompact, CG-cb}, for some (suitably) complete
 (indefinite) Sasaki--Einstein structures. Specifically, we obtain:

\begin{theoremC}\label{C} Assume the setting of Theorem B and that
 $M_0\neq \emptyset$ (so $p,q >0$). Then the manifold with boundary
  $(M\setminus M_{\mp},\mbp\vert_{M \setminus M_{\mp}})$ is a
  projective compactification (of order $2$) of the Sasaki--Einstein
  manifold, respectively, $(M_{\pm},g_{\pm}, k\vert_{M_{\pm}})$, with
  projective infinity the conformal Fefferman structure $(M_0, \pm
  \mbc)$.
  \end{theoremC}

 Note that $(M_0, \pm \mbc)$,
 being a conformal Fefferman structure, locally fibers over a CR manifold
 of dimension $2m-1$.  More can be said: suppose we are in the setting of Theorem C and let
 $(\widetilde M, \tilde J, \tilde\mbp)$ be the the (local) leaf space
 defined by $k$ with its canonical c-projective structure. Let us
 write $\pi: M\rightarrow \widetilde M$ for the natural projection
 (where, if necessary, $M$ is replaced by some sufficiently small open
 subset), and set $\widetilde M_{\pm}:=\pi(M_{\pm})$ and
 $\widetilde M_0:=\pi(M_0)$. It is well-known that $(M_{\pm},g_{\pm},
 k)$, by dint of being Sasaki--Einstein, induces respective
 K\"ahler--Einstein structures $(\tilde J, \tilde g_{\pm})$ on
 $\widetilde M_{\pm}$. It turns out that $\widetilde{M}_0$ is a hypersurface in the complex manifold $(\wt M, \tilde J)$ separating these, and as such it is naturally equipped with
 a CR structure (of hypersurface type), which we will show is nondegenerate. Recall that any nondegenerate CR manifold (of hypersurface
 type) determines a conformal structure on a certain circle bundle over it, called its conformal Fefferman structure; see
 \cite{BDS,CGFeffchar,Fefferman,Graham-Sparlings,LeeF,Leit}
 for characterisations of conformal Fefferman structures. Now a nice compatibility arises:

\begin{theoremD}\label{D}
Assume the setting of Theorem B and that $M_0\neq \emptyset$. Then:
\begin{enumerate}
	\item $\wtM_0$ is a smooth hypersurface of the complex manifold $(\wtM, \tilde J)$ (separating $\widetilde M_+$ and $\widetilde M_-$) and
	  as such it inherits a CR structure (of hypersurface type) from $\tilde J$. The latter is nondegenerate and its conformal Fefferman structure equals $(M_0, \pm\mbc)$.  \item The manifold with boundary $(\wtM\setminus \wtM_{\mp}, \tilde J\vert_{\wtM \setminus \wtM_{\mp}},\tilde\mbp\vert_{\widetilde M \setminus \widetilde M_{\mp}} )$ is a c-projective compactification (of order $2$) as defined in \cite{CG-cproj} of the K\"ahler--Einstein manifold, respectively, $(\wtM_{\pm},\tilde{g}_{\pm})$.
\end{enumerate}
  \end{theoremD}

The unit sphere $S^{2m+1}$ in $\mathbb{C}^{m+1}$, were $\mathbb{C}^{m+1}$ is equipped with its
usual definite Hermitian inner product, inherits from this inner product a Riemannian signature
Sasaki--Einstein structure, and indeed this is the standard model of such.
Examples of (pseudo-Riemannian) Sasaki--Einstein manifolds are abundant
in all odd dimensions $2 m + 1 \geq 5$ \cite[Corollaries B, C,
  E]{BoyerGalicki}, \cite{Sparks} and hence Theorem A guarantees that there are
many examples of projective structures equipped with parallel tractor
Hermitian structures. By construction, for the holonomy reductions
that arise from Einstein--Sasaki structures, $M = M_{\pm}$ and in particular $M_0$ is
empty. On the other hand, we devote significant attention (Theorem
B(b), Theorem C, Theorem D, and additional results in Section
\ref{assembly}) to the case when $M_0$ is not empty. If we consider again the sphere $M=S^{2m+1}$,
viewed now as the real ray projectivisation of $\mathbb{C}^{m+1}\setminus \{0\}$, and
equipped with its standard flat projective structure $\mbp$, then for a certain choice
of indefinite Hermitian inner product on $\mathbb C^{m+1}$ (equivalently, indefinite Hermitian metric on its flat tractor bundle)
$M_0$ is not empty.

However, a priori it is not obvious that there are any examples of
projective manifolds with parallel indefinite Hermitian structure on
its tractor bundle and $M_0\neq \emptyset$ that are not projectively
flat, that is, not locally projectively equivalent to $(S^{2m+1},
\mbp)$. In Section
\ref{section:examples} we resolve this by establishing the existence
of a large number of examples with $M_0$ nonempty, justifying the
consideration we give to that case. Any Levi-nondegenerate embedded
hypersurface $\wtM_0 \subset \bbC^m$, $m > 1$, inherits a CR structure
of hypersurface type, and we show that for those for which a certain
invariant weighted function, the \textit{CR obstruction}, vanishes
(again, there are many of these), we can construct examples of
non-flat projective structure with parallel tractor Hermitian
structures for which $M_0 \neq \emptyset$ and in particular for which
the CR structure underlying the conformal Fefferman structure on $M_0$
is the one we started with. This construction recasts Fefferman's
original procedure for constructing approximate solutions to a
particular Dirichlet problem naturally associated to CR geometry
\cite{Fefferman} in projective tractor language using ideas from
\cite{GPW, GrahamWillse}.

\tableofcontents

\section{Background on affine geometry and notation}\label{back}
Throughout manifolds and tensor on manifolds will be assumed smooth, and $M$ specifically will always denote a smooth manifold of
dimension $n+1\geq 2$. If $n$ is even, which will be the case in various parts of this article, we write $n=2m$ for $m\geq 1$. For simplicity of exposition, we also assume $M$ is connected.

When it is convenient, will use standard abstract index notation in
the sense of Penrose. As usual we use the Roman alphabet to index
tensors on $M$: we write $\ce^a$ (respectively $\ce_a$) for the
tangent bundle $TM$ (respectively cotangent bundle $T^*M$) of $M$
and $\xi^a$ (respectively $\omega_a$) for a vector field (respectively
a $1$-form) on $M$. Moreover, we write $\xi^a\omega_a$ for the
canonical pairing between vector fields and $1$-forms, using an
abstract index analogue of the Einstein summation convention, and
denote by the Kronecker delta $\delta^b{}_a$ the identity section of
the bundle $\textrm{End}(TM)$ of endomorphisms of $TM$. Indices
enclosed by round (respectively by square brackets) indicate
symmetrisation (respectively skew-symmetrisation) over the enclosed
indices. Similarly, we  use the Greek alphabet to index sections
of a hyperplane distribution $H\subset TM$ (see Section
\ref{Sp-reduction}) and capital Roman letters to index sections of the
 projective tractor bundle (see Section \ref{tr}).

\subsection{Affine and projective manifolds}\label{affine_manifolds}
A (torsion-free) linear connection $\nabla$ on the tangent bundle $TM$ of a manifold $M$ is called a \emph{(torsion-free) affine connection} on $M$ and we shall refer to such a pair $(M, \nabla)$ as a \emph{(torsion-free) affine manifold}.

Suppose now $(M,\nabla)$ is a torsion-free affine manifold of dimension $n+1\geq 2$.
The \emph{curvature} of $\nabla$ will be denoted by
$$R_{ab}{}^c{}_d\in\Gamma(\Wedge^2 T^*M \otimes \End(TM)),$$
where we adopt the convention $R_{ab}{}^c{}_d\xi^d:=(\nabla_{a}\nabla_b-\nabla_b\nabla_a)\xi^c$ for a vector field $\xi^a$. The \emph{Ricci tensor} of $\nabla$ is defined by $\Ric_{bd}:=R_{cb}{}^c{}_d $.

The totally trace-free part $W_{ab}{}^c{}_d$ of the curvature $R_{ab}{}^c{}_d$ is
called the {\em projective Weyl curvature}: By definition it has the same symmetries as the curvature, including the algebraic Bianchi identity
$W_{[ab}{}^c{}_{d]}=0$, but is, in addition, \emph{totally trace-free}, meaning that
\begin{equation}\label{trace-freeness_Weyl_curv}
W_{ab}{}^a{}_d=0\quad \textrm{ and }\quad W_{ab}{}^d{}_d=0.
\end{equation}
The symmetries of $W_{ab}{}^c{}_d$
immediately imply that it  vanishes identically if
$\dim M =2$.  Specifically, in any dimension $n+1\geq 3$ we have the following decomposition
\begin{equation}\label{decp}
R_{ab}{}^c{}_d= W_{ab}{}^c{}_d +2 \delta^c{}_{[a} \Rho_{b]d}+\beta_{ab}\delta^c{}_d,
\end{equation}
where
\begin{equation}\label{proj_Schouten}
\Rho_{ab}:=\tfrac{1}{n(n+2)}\left[ (n+1) \Ric_{ab}+\Ric_{ba}\right]
\end{equation}
is called the {\em projective Schouten tensor} and $\beta_{ab}=-2\Rho_{[ab]}$.

The differential Bianchi identity of the curvature of $\nabla$ moreover implies
\begin{equation}\label{diff_Bianchi_Weyl}
\nabla_c W_{ab}{}^c{}_{d}=(n-1)C_{abd},
\end{equation}
where
\begin{equation}\label{Cotton}
C_{abc}:= \nabla_a \Rho_{bc}-\nabla_b \Rho_{ac}
\end{equation}
is called the {\em
  projective Cotton tensor}.

We call a torsion-free affine connection $\nabla$ \emph{special}, if it preserves a volume density on $M$. In this case, one has $R_{ab}{}^c{}_c=0$ and hence the Bianchi identity $R_{[ab}{}^c{}_{d]}=0$ implies that the Ricci tensor of $\nabla$ and hence also the projective Schouten tensor are symmetric: $\Ric_{ab}=\Ric_{(ab)}=n\Rho_{(ab)}=n\Rho_{ab}$.

Recall that two affine connections $\nabla$ and $\widehat{\nabla}$ on a manifold are
said to be {\em projectively equivalent} if they have the same geodesics as unparameterised curves. Any two connections that differ
only by torsion are projectively equivalent, which motivates the following definition of a projective structure:

\begin{definition} Let $M$ be a manifold of dimension at least $2$.
\begin{itemize}
\item A \emph{projective structure} on $M$ is an equivalence class $\mbp$ of projectively equivalent torsion-free affine connections on $M$, and we refer to $(M,\mbp)$ as a \emph{projective manifold}.
If $M$ is orientable with fixed orientation we refer to $(M,\mbp)$ as an \emph{oriented projective manifold}.
\item
Two projective manifolds are called \emph{(locally) projectively equivalent}, if there exists a \emph{(local) projective diffeomorphism} between them, that is a (local) diffeomorphism mapping unparametrised geodesics to unparametrised geodesics.
\end{itemize}
\end{definition}

It is well known that for a torsion-free affine manifold $(M,\nabla)$
the projective Weyl curvature is \emph{projectively invariant}: $W_{ab}{}^c{}_d$ only depends on the
projective structure $\mbp:=[\nabla]$ induced by $\nabla$.  Moreover,
in dimension $n+1\geq 3$ its vanishing characterises projectively flat
torsion-free affine manifolds $(M,\nabla)$, where $(M,\nabla)$ is
called \emph{projectively flat} if it is locally projectively
equivalent to $\bbR^{n+1}$ with its canonical projective structure
given by the geodesics of its flat connection, namely the straight
lines. If $\dim M =2$, the projective Cotton tensor is projectively
invariant and takes over the role of the projective Weyl curvature in
that it is zero if and only if the structure is projectively flat.

For later use let us also mention that the standard \emph{compact
homogeneous flat model for oriented projective manifolds} is the
$(n+1)$-dimensional sphere arising as the ray projectivisation
$S^{n+1}:=\mathbb{P}_+(\mathbb{R}^{n+2})$ of $\mathbb{R}^{n+2}$
(i.e.\ the double cover of $\mathbb{RP}^{n+1}$), where the projective
structure on $S^{n+1}$ is given by the unparametrised geodesics of the
round metric on $S^{n+1}$, namely, the embedded great circles. It is projectively flat and the
group of orientation-preserving projective diffeomorphisms of
$S^{n+1}$ can be identified with the special linear group
$\textrm{SL}(n+2,\bbR)$ acting transitively on
$\mathbb{P}_+(\mathbb{R}^{n+2})$ in the standard way.

 \subsection{Density bundles and related objects}\label{notn}
Recall that on any manifold $M$ of dimension $n+1$ the line bundle
$\mathcal{K}:=(\Wedge^{n+1} TM)^{\otimes 2}$ of volume densities is canonically oriented and thus one may
take oriented roots of it:
Given $w\in \mathbb{R}$ we set
\begin{equation} \label{pdensities}
\ce(w):=\mathcal{K}^{\frac{w}{2n+4}} ,
\end{equation}
and
\begin{equation} \label{cdensities}
\ce[w]:=\mathcal{K}^{\frac{w}{2n+2}} .
\end{equation}
The notation $\ce(w)$ is convenient for projective geometry (for
example), while the notation $\ce[w]$ is more convenient and standard
in conformal geometry. Further we shall use the abbreviations
$TM(w):=TM\otimes\ce(w)$, respectively $TM[w]:=TM\otimes\ce[w]$, and a
similar notational shorthand for any tensor product of a vector bundle
with $\ce(w)$ respectively $\ce[w]$.  Note also that any affine
connection $\nabla$ on $M$ determines a connection on $\mathcal{K}$
and thus also on the roots $\ce(w)$ and $\ce[w]$. We use $\nabla$ to
denote any of these.

Observe that, if $M$ is oriented, then its orientation can be encoded by a tautological weighted volume form
$
\eps_{a_0 \cdots
  a_{n}} \in \Gamma(\Wedge^{n + 1} T^*M (n + 2))
$
that defines the identification
$$
\eps_{a_0 \cdots
  a_{n}}: \Wedge^{n + 1} TM\to \ce(n+2),
$$
with its dual
$
\eps^{a_0 \cdots
  a_{n}}\in \Gamma( \Wedge^{n + 1} TM (-n - 2) )
$
defined to satisfy $\eps^{a_0 \cdots
  a_{n}}\eps_{a_0 \cdots
  a_{n}}=(n+1)!$.
Note that  $
\eps_{a_0 \cdots
  a_{n}}$ and its dual are parallel for any affine connection $\nabla$.

\section{Sasaki structures}\label{Section: Sasaki}
Suppose $M$ is a manifold equipped with a (pseudo-)Riemannian metric
$g_{ab}\in\Gamma(S^2T^*M)$. Then $g_{ab}$ and its inverse
$g^{ab}\in\Gamma(S^2TM)$, which is characterised by
$g_{ab}g^{bc}=\delta^{c}{}_a$, can be used to raise and lower indices,
and we will do this without further comment. Moreover, we denote by
$\nabla$ the Levi-Civita connection of $g$ and
by $R_{ab}{}^c{}_d$ the curvature of $\nabla$ (as for any affine manifold).

 Recall that $k^a\in\Gamma(TM)$ is called a
\emph{Killing field} of $(M,g)$, if $k$ is a symmetry
of $g$, that is, if the Lie derivative of $g$ along $k$
vanishes. Expressed in terms of the Levi-Civita connection,
$k^a\in\Gamma(TM)$ is Killing if and only if
\begin{equation}\label{Killing_equation_1}
\nabla_{(a}k_{b)}=0.
\end{equation}

\begin{definition}\label{definition:Sasaki-structure}
A \emph{Sasaki structure} on a manifold $M$ consists of a (pseudo-)Riemannian metric
$g_{ab}\in\Gamma(S^2T^*M)$ and a Killing field $k^a \in \Gamma(TM)$ of $g$ such that
\begin{enumerate}
	\item $g_{ab} k^a k^b = 1$
	\item $\nabla_a\nabla_bk^c = -g_{ab} k^c + \delta^c{}_a k_b .$
\end{enumerate}
\end{definition}

Note that (a) implies in particular that $k$ vanishes nowhere.

The following proposition gives well-known equivalent
characterisations of Sasaki structures, see e.g.\,\cite{BG-book, Sparks}.

\begin{proposition}\label{proposition:Sasaki-characterizations}
Let $(M, g)$ be a (pseudo-)Riemannian manifold. Then the following are equivalent:
\begin{enumerate}
    \item There exists $k^a \in \Gamma(TM)$ such that $(M, g, k)$ is a Sasaki structure.
    \item There exists a Killing vector field $k^a \in \Gamma(TM)$ satisfying
        \[
			g_{ab} k^a k^b = 1
			\qquad \textrm{and} \qquad
			R_{bc}{}^a{}_d k^d = 2 \delta^a{}_{[b} k_{c]} \textrm{.}
        \]
      \item The metric cone $(C(M), \bar{g}) := (\bbR_+ \times M, dr^2 + r^2 g)$ is K\"{a}hler with respect to some complex structure on $C(M)$.
    \end{enumerate}
\end{proposition}
 \begin{proof}
For later purposes we give here the proof of the equivalence $(a)\Leftrightarrow (b)$; for the equivalence $(b)\Leftrightarrow(c)$ see the above mentioned references.

Suppose $k^a\in\Gamma(TM)$ is a Killing field of $(M, g)$, which by \eqref{Killing_equation_1} is equivalent to $\nabla_ak_b=\nabla_{[a}k_{b]}$.
The $2$-form $\nabla_ak_b$ is evidently exact and so in particular closed, which is equivalent to $\nabla_{[a}\nabla_{b\phantom{]}} k_{c]}=0$. Therefore, we have
 	\begin{equation}\label{prolongation_Killing}
 	\nabla_{a}\nabla_b k_{c}=-\nabla_{b}\nabla_c k_{a}-\nabla_{c}\nabla_a k_{b}=-\nabla_{b}\nabla_c k_{a}+\nabla_{c}\nabla_b k_{a}=R_{bc}{}^d{}_ak_d.
 	\end{equation}
 	Since $R_{abcd}=R_{ab[cd]}$, we deduce that $\nabla_{a}\nabla_b k_{c}=-R_{bcad}k^d$, which establishes the equivalence of $(a)$ and $(b)$.
\end{proof}	
	
Note that by (c) any Sasaki manifold is necessarily odd-dimensional.

\subsection{A projective characterisation of Sasaki geometry}\label{psas}

From Proposition \ref{proposition:Sasaki-characterizations} we deduce
now a characterisation of Sasaki manifolds in terms of a unit length Killing
field and basic objects from projective differential geometry as
introduced in Section \ref{affine_manifolds}:
\begin{theorem}\label{Jpar}
A (pseudo-)Riemannian manifold $(M, g)$ is Sasaki if and only if there exists a Killing vector field $k^a\in\Gamma(TM)$ that satisfies
\begin{enumerate}
    \item $g_{ab} k^a k^b = 1$,
    \item $W_{ab}{}^c{}_d k^d = 0$, and
    \item $\Rho_{ab} k^a k^b = 1$,
\end{enumerate}
where $W_{ab}{}^c{}_d$ is the projective Weyl curvature and $\Rho_{ab}$ the projective Schouten tensor of the Levi-Civita connection $\nabla$ of $g$.
\end{theorem}
\begin{proof}
Suppose first that $(M,g,k)$ is a Sasaki manifold of dimension $2m+1$. By Proposition \ref{proposition:Sasaki-characterizations}(b), $k^a\in\Gamma(TM)$ is a Killing field of unit length
satisfying $R_{bc}{}^a{}_d k^d = 2 \delta^a{}_{[b} k_{c]}$. Hence, tracing over  $a,b$ yields $\Ric_{cd} k^d = 2m k_c$, or equivalently, $\Rho_{cd} k^d = k_c$, so $k^d$ is an eigenvector of $P^c{}_d$ with eigenvalue $1$. In particular,
$$\Rho_{ab} k^a k^b =g_{ab} k^a k^b=1.$$ Then, rearranging \eqref{decp} and substituting gives
\[
    W_{b c}{}^a{}_d k^d
        = R_{bc}{}^a{}_d k^d - 2 \delta^a{}_{[b} \Rho_{c]d} k^d
        = 2 \delta^a{}_{[b} k_{c]} - 2 \delta^a{}_{[b}k_{c]}
        = 0 \textrm{.}
\]
Conversely, suppose now $k^a\in\Gamma(TM)$ is a Killing field satisfying (a)-(c).  Since $R_{abcd}=R_{ab[cd]}$, we have
\begin{equation*}
0=R_{abcd}k^ck^d= W_{abcd}k^ck^d + 2 g_{c[a} \Rho_{b]d}k^ck^d=2k_{[a}\Rho_{b]d}k^d.
\end{equation*}
Hence, there exists a function $B$ on $M$ such that $\Rho_{bd} k^d = B k_b$ and, by $(a)$ and $(b)$, we must in fact have $B$ constant and equal to $1$.
Substituting back into \eqref{decp} therefore shows
\begin{equation*}
R_{bc}{}^a{}_dk^d = 2\delta^a{}_{[b} k_{c]}.
\end{equation*}
Hence, $(M,g,k)$ is Sasaki by Proposition \ref{proposition:Sasaki-characterizations}(b).
\end{proof}

Nondegeneracy of the projective Weyl tensor is thus an obstruction to the existence of the Levi-Civita connection of a Sasaki structure in a projective class:

\begin{corollary}\label{obstruction_existence_of_Sasaki_metric}
  If at some point $x$ on a projective manifold $(M, \mbp)$
    the map $\Psi_x: T_x M \to \Wedge^2 T^*_x M \otimes T_x M$ defined by $\xi^a \mapsto W_{ab}{}^c{}_d \xi^d$ has trivial kernel, then $\mbp$
is not induced by the Levi-Civita connection of a Sasaki structure.
\end{corollary}
\begin{remark} For projective manifolds $(M,\mbp)$ determined by the Levi-Civita connection of a metric, satisfaction of the
  \emph{Weyl nullity condition}, namely the condition that the map $\Psi$ of Corollary \ref{obstruction_existence_of_Sasaki_metric} has nontrivial kernel at all points of $M$
  (cf. also condition (b) of Theorem \ref{Jpar}), implies the existence of strong tools, including subtle invariants and an alternative metric projectively invariant
  tractor connection. These ideas are developed and explored in \cite{CEMN, GMat}.
\end{remark}
\subsection{Prolonging the Killing equation}
The definition of a Sasaki structure includes a Killing
vector field. In this subsection, we consider a manifold equipped with
a special torsion-free affine connection $\nabla$ and prolong
the \emph{Killing-type equation}
\begin{equation}\label{equation:Killing-type}
\nabla_{(a} k_{b)} = 0
\end{equation}
to a closed, first-order system. This prolongation will expose a strong link with projective geometry.

Note that the deduction of \eqref{prolongation_Killing} from
\eqref{equation:Killing-type} in the proof of Proposition
\ref{proposition:Sasaki-characterizations} used that $\nabla$ is
torsion-free but not that it preserves a metric.  Hence, \eqref{equation:Killing-type} implies that
\begin{equation*}
\nabla_{a}\nabla_{b}k_c=R_{bc}{}^d{}_ak_d.
\end{equation*}
Thus, the Killing-type equation \eqref{equation:Killing-type} is  equivalent to the following closed, first-order system:
\begin{align*}
    0 &= \nabla_a k_b - \mu_{ab} \\
    0 &= \nabla_a \mu_{bc} - R_{bc}{}^d{}_a k_d \textrm{,}
\end{align*}
where $\mu_{ab}:=\nabla_{[a}k_{b]}\in\Gamma(\Wedge^2 T^*M)$.
Using the projective decomposition \eqref{decp} of the curvature tensor $R_{ab}{}^c{}_d$ of  an affine manifold $(M,\nabla)$ this observation can rephrased (cf. also\ \cite{BCEG}) as follows:

\begin{proposition}\label{prolongprop} On a manifold $M$ of dimension at least $2$ equipped with a torsion-free special affine connection $\nabla$,
solutions of the Killing-type equation
\[
    \nabla_{(a} k_{b)} = 0
\]
on $1$-forms $k_b$ are in bijective correspondence with sections $(k, \mu)$ of
$T^* M \oplus \Wedge^2 T^*M$ which are parallel with respect to the connection \begin{equation}\label{prol_step2}
\nabla^{\emph{prol}}_a\left(\begin{array}{c}
k_b\\
\mu_{bc}
\end{array}\right):=
\left(\begin{array}{c}
\nabla_a k_b - \mu_{bc}\\
\nabla_a \mu_{bc}+2 \Rho_{a[b}k_{c]}
\end{array}\right)
- \left(\begin{array}{c}
0\\
W_{bc}{}^d{}_a k_d
\end{array}\right).
\end{equation}
\end{proposition}

We will see that the first term on the right-hand side is a canonical
connection, available on any projective manifold, namely the projective
tractor connection.

\section{Projective geometry}\label{proj-sec}
It is well known that two torsion-free affine connections $\nabla$ and $\widehat{\nabla}$
are projectively equivalent if and only if there exists a $1$-form $\Upsilon_a\in\Gamma(T^*M)$ such that
\begin{equation}\label{projective_change}
	\widehat{\nabla}_a\xi^b=\nabla_a\xi^b+\Upsilon_a\xi^b+\delta^b{}_a\Upsilon_c\xi^c
\end{equation}
for any vector field $\xi^a\in\Gamma(TM)$.
For two projectively equivalent affine connections $\widehat{\nabla}$ and $\nabla$ related as in \eqref{projective_change} the connections they induce on the projective density bundles $\mcE(w)$ are related by
\begin{equation}\label{projective_change_on_densities}
\widehat{\nabla}_a\sigma=\nabla_a\sigma+w\Upsilon_a.
\end{equation}
This shows that, given a projective manifold $(M,\mbp)$, for any fixed $0\neq w\in\bbR$ mapping a connection in $\mbp$ to its induced connection on $\mcE(w)$ induces an (affine)
bijection between connections in $\mbp$ and connections on $\mcE(w)$. In particular, this implies that any trivialisation of $\mcE(w)$, viewed as a  nowhere-vanishing section
$\sigma\in\mcE(w)$, gives rise to a connection $\nabla$ in $\mbp$ characterised by $\nabla_a\sigma=0$. For $0\neq w \in\bbR$ we call a nowhere vanishing section of $\mcE(w)$ as well as its corresponding connection in $\mbp$ a \emph{scale}. Note $\nabla\in\mbp$ is a scale if and only if it is special in the sense of Section \ref{affine_manifolds}.

\subsection{The projective tractor bundle and its connection}\label{tr}

On a general projective manifold $(M,\mbp)$ of dimension $n+1\geq2$ there is no distinguished
connection on $TM$. There is however, as already advertised, a projectively invariant connection (that is, one intrinsic to $\mbp$) on a related vector bundle $\cT$
of rank $n+2$, called the \textit{(normal) projective tractor connection}, which we now describe.

Consider the first jet bundle
$J^1\ce(1)\to M$ of the density bundle $\ce(1)$ (see, for example,
\cite{palais} for a general development of jet bundles).
There is a canonical bundle map called the {\em first jet projection map}
$J^1\ce(1)\to\ce(1)$, which at each point is given by mapping
a $1$-jet of a density to its evaluation at that point, and
we may identify the kernel of this map with $\ce_a(1)=T^*M (1)$.  We write $\cT^*$, or, in an
abstract index notation, $\ce_A$, for $J^1\ce(1)$, and $\cT$ or $\ce^A$
for the dual vector bundle. Then we can view the jet projection as a
canonical section $X^A \in \Gamma(\ce^A(1))$ called the \textit{canonical tractor}. Likewise,
the inclusion of the kernel of this projection can be viewed as a
canonical bundle map $\ce_a(1)\to\ce_A$, which we denote by
$Z_A{}^a$. In this notation the jet exact sequence (at $1$-jets) is
\begin{equation}\label{euler}
0\to \ce_a(1)\stackrel{Z_A{}^a}{\to} \ce_A \stackrel{X^A}{\to}\ce(1)\to 0.
\end{equation}

As mentioned above,
 any connection $\nabla \in \mbp$
determines a connection on $\ce(1)$ and vice versa. On the other hand,  by definition,
a connection on $\ce(1)$ is precisely a splitting
of the $1$-jet sequence \nn{euler}. Hence,
given a choice of connection $\nabla\in\mbp$, we get an identification
$\smash{\ce_A \cong \ce(1)\oplus \ce_a(1)}$,
and we write
 \begin{equation}\label{split}
 Y_A:\ce(1) \to \ce_A \qquad \mbox{and} \qquad W^A{}_a: \ce_A\to \ce_a(1),
 \end{equation}
for the corresponding bundle maps giving the splitting of \nn{euler}; thus,
$$
 X^AY_A=1, \qquad  Z_A{}^b W^A{}_a=\delta^b{}_a, \qquad \mbox{and} \qquad Y_AW^A{}_a=0.
$$

With
respect to a connection $\nabla\in\mbp$, equivalently a splitting \nn{split},
we define a connection on $\mathcal T^*$ by
\begin{equation}\label{pconn}
\nabla^{\mathcal{T}^*}_a \binom{\si}{\mu_b}
:= \binom{ \nabla_a \si -\mu_a}{\nabla_a \mu_b + \Rho_{ab} \si},
\end{equation}
where $\Rho_{ab}$ is the projective Schouten tensor of $\nabla\in \mbp$ as introduced in Section \ref{affine_manifolds}.

It turns out that \nn{pconn} is
independent of the choice $\nabla \in \mbp$, and so
$\nabla^{\mathcal{T}^*}$ is determined canonically by the projective
structure $\mbp$. We have followed the construction of \cite{BEG}, but
this {\em cotractor connection} was first introduced by Thomas in \cite{Thomas}.  We term $\cT^*=\ce_A$ the {\em cotractor bundle}, and we note that the dual
{\em tractor bundle} $\cT=\ce^A$
canonically carries the dual {\em tractor connection}: In terms of a
splitting \nn{split} this is given by
\begin{equation}\label{tconn}
\nabla^\cT_a \left( \begin{array}{c} \nu^b\\
\rho
\end{array}\right) =
\left( \begin{array}{c} \nabla_a\nu^b + \rho \delta^b{}_a\\
\nabla_a \rho - \Rho_{ab}\nu^b
\end{array}\right).
\end{equation}
Note that given a choice $\nabla\in \mbp$, by coupling with the
tractor connection we can differentiate tensors taking values in
tractor bundles as well as weighted tractors. In particular, $\nabla_aX^A=W^A{}_a$, so this gives an explicit computational handle
on the splitting determined by the choice $\nabla$.

The curvature $R=R^{\mcT}$ of $\nabla^{\mcT}$ is a $2$-form on $M$ with values in the bundle $\mathfrak{sl}(\mcT)$ of trace-free endomorphisms of $\mcT$. With respect to a choice $\nabla\in\mbp$ it may be written
as
\begin{equation}\label{tractor_curvature}
R_{ab}{}{}^C{}_{D}=W_{ab}{}^c{}_d W^C{}_c Z_{D}{}^d-C_{abd}Z_{D}{}^d X^C,
\end{equation}
where $W_{ab}{}^c{}_d$ is the projective Weyl curvature and $C_{abd}$ the projective Cotton tensor as introduced in Section \ref{affine_manifolds}.

From \eqref{diff_Bianchi_Weyl} and formula \nn{tractor_curvature}, it follows that the tractor curvature $R^{\mcT}$ vanishes identically
if and only if $W_{ab}{}^c{}_d$ does for $n+1\geq3$ or $C_{abc}$ does
for $n+1=2$. Thus $R^{\mcT}$ vanishes
identically if and only if $(M,\mbp)$ is projectively flat.

Let us also mention that for an oriented projective manifold $(M, \mbp)$ the tractor bundle $\mcT$ always admits a parallel \emph{tractor volume form}, that is,
a nowhere vanishing section $\vol_{A_1 \cdots A_{n + 2}}$ of $\Wedge^{n+2} \mcT^*$ that is parallel for the tractor connection. It is (with respect to any splitting of \eqref{euler}) given by
\begin{equation}\label{tractor_volume_form}
\vol_{A_0 \cdots A_{n + 1}}
	:=
		Y_{[A_0} Z_{A_1}{}^{a_1}... Z_{A_{n + 1}]}{}^{a_{n + 1}}
			\epsilon_{a_1 \cdots a_{n + 1}}
				\in\Gamma(\Wedge^{n+2} \mcT^*),
\end{equation}
where $\epsilon_{a_1 \cdots a_{n + 1}}\in\Wedge^{n+1}T^*M(n+2)$ is the weighted $(n+1)$-form encoding the orientation on $M$ as introduced in Section \ref{notn}.

Recall that any projective manifold $(M,\mbp)$ admits a \emph{canonical} (usually called its \emph{normal)} \emph{Cartan connection} (see \cite{csbook, Cartan}),
and this is equivalent to
its normal tractor connection \cite{CapGoTAMS}. For later purposes let us sketch this correspondence for oriented projective manifolds:
If $(M,\mbp)$ is oriented and of dimension $n+1$ its normal Cartan connection $\omega$ is of type $(\mathfrak {sl}(n+2,\bbR), P)$, where $P$ is the stabiliser in
$\textrm{SL}(n+2,\bbR)$ of a ray in $\bbR^{n+2}$. Thus $\omega$ is a certain $\mathfrak {sl}(n+2,\bbR)$-valued $1$-form on the total space
of a $P$-principal bundle $ \mathcal G\rightarrow M$ that defines a trivialisation $T\mcG\cong \mcG\times \mathfrak {sl}(n+2,\bbR)$ and is suitably
compatible with the principal right action of $P$ on $\mcG$; see \cite{csbook} for details.
The defining properties of a Cartan connection imply that $\omega$ extends to a
$G$-principal connection on the $G$-principal bundle $\hat\mcG:=\mcG\times_P G$, where $G:=\textrm{SL}(n+2,\bbR)$. Hence, for any $G$-module $\mathbb V$
the Cartan connection $\omega$ induces a linear connection $\nabla^{\cV}$
on any associated vector bundle
$$\cV:=\mcG\times_P\bbV \cong \hat\mcG\times_G\mathbb V.$$
For the standard representation $\mathbb V=\bbR^{n+2}$ of $G$ we have that $\cV=\mcT$ and that $\nabla^{\cV}$ coincides precisely with the canonical (normal) tractor connection $\nabla^{\mcT}$.
We refer to any vector bundle $\cV$ as above as a \textit{tractor bundle} of $(M,\mbp)$ and to $\nabla^{\cV}$ as the tractor connection on $\cV$. Since the tractor connection on any tractor bundle is induced by that on $\mcT$,
we also often just write $\nabla^\mcT$ (instead of $\nabla^\cV$) for the tractor connection on any tractor bundle $\cV$.
Conversely, the Cartan bundle $\mcG\rightarrow M$
can be recovered from $\mcT$ by forming an adapted frame bundle of $\mcT$, and $\omega\in\Omega^1(\mcG, \mathfrak{sl}(n+2,\bbR))$ can be reconstructed from
$\nabla^{\mcT}$; see \cite{CapGoTAMS, csbook}.

Finally, for the flat homogeneous model of oriented projective geometry, namely, $G / P \cong S^{n+1}$ equipped with its canonical ($G$-invariant) projective structure, the normal Cartan connection is the
Maurer--Cartan form on $G$, where $G$ is viewed as the total space of the $P$-principal bundle $G\rightarrow G/P$. For any $G$-module $\mathbb V$ the tractor bundle $\cV=G\times _P\mathbb V$
over $G/P$ can be trivialised via
\begin{align}\label{tractor_bundle_homog_model}
G\times _P\mathbb V \stackrel{\cong}{\to} G/P\times\mathbb V, \qquad [g, v]\mapsto (gP, gv) ,
\end{align}
and the tractor connection $\nabla^{\cV}$ is precisely the flat connection defined by this trivialisation.

\subsection{Parallel tractors and normal solutions} \label{parBGG-sec}

\newcommand{\cG}{\mathcal{G}}

Suppose $(M,\mbp)$ is an oriented projective manifold of dimension $n+1\geq 2$ and let $P\subset G$ be the groups as defined in Section \ref{tr}.
From \nn{euler} we conclude that there is a (projectively invariant) injective bundle map
\begin{equation}\label{injection_cotangent}
T^* M\to \End (\cT), \qquad u_b\mapsto X^A Z_{B}{}^b u_b .
\end{equation}
Sections of $\End (\cT)$ act on any tractor bundle $\mathcal V$ (in the obvious tensorial way), and hence, via the injection \eqref{injection_cotangent}, so does $T^*M$; this action in turn determines a sequence of bundle maps
\begin{equation}\label{Kostant}
\partial^* : \Wedge^k T^*M\otimes \cV \to  \Wedge^{k-1} T^*M\otimes \cV,\qquad k=1,\ldots,n+1,
\end{equation}
which together satisfy $\partial^*\circ \partial^*=0$. Hence they form a complex
and determine subquotient
bundles
$\cH_k(M,\cV) := (\operatorname{ker} \partial^*) / (\operatorname{im} \partial^*)
$ of the bundles $\Wedge^k T^*M\otimes \cV
$ of $\cV$-valued $k$-forms.

\begin{remark}
Via the canonical Cartan connection the bundles $\Wedge^k T^*M\otimes \cV$ can be identified with vector bundles associated to the Cartan bundle $\mcG\rightarrow M$
with standard fiber the $P$-module $\Wedge^k(\mathfrak g/\mathfrak p)^*\otimes \mathbb V\cong\Wedge^k\mathfrak p_+\otimes \mathbb V$, where $\mathfrak p_+$
is the nilradical of the Lie algebra $\mathfrak p$ of $P$ (for our $P$ the nilradical is abelian). The maps
$\partial^*$ can be identified with bundle maps induced from the chain complex computing the Lie algebra homology of $\mathfrak p_+$ with values in $\mathbb V$.
By a Theorem of Kostant \cite{Kostant}, the vector bundles $\cH_k(M, \mathcal V)$ are all weighted tensors bundles on $M$.
\end{remark}

In \cite{CD,CSS} it was shown that each irreducible $G$-module $\mathbb{V}$ determines a so-called \textit{BGG sequence}
$$
  \Gamma(\cH_0)\stackrel{\cD^{\cV}_0}{\rightarrow}\Gamma(\cH_1)\overset{\cD^{\cV}_1}{\rightarrow}\cdots
  \overset{\cD^{\cV}_{n}}{\rightarrow}\Gamma(\cH_{n+1}),
  $$
  where  $\cH_k := \cH_k(M,\cV)$ and
  each $\cD^{\cV}_i$ is a projectively invariant
  linear differential operator (that is, one intrinsic to $\mbp$).  For projectively flat structures
  $(M,\mbp)$ the BGG sequence corresponding to $\cV$ is a complex,
  which is dual, in a suitable sense, to the generalised
  Berstein--Gelfand--Gelfand resolution of $\mathbb V$ by generalised
  Verma modules.

  In this article we will only be interested in the \textit{first BGG operator} $\cD^\cV=\cD^V_0$,
  which defines an overdetermined system of PDEs and is closely related to the
  tractor connection $\na^{\cV}$ on $\cV$. Recall that the parabolic subgroup $P\subset G$ determines a filtration on $\bbV$ by $P$-invariant subspaces. Denoting the largest non-trivial filtration component by $\bbV^0$, which is given by
  $\bbV^0=\mathfrak p_+\bbV$, and setting $\cV^0:=\mcG\times_P\mathbb V^0$, then $\cH_0$ just equals the quotient  $\cV/\cV^0$.
  We denote by $\Pi^{\cV}:\Gamma(\cV)\to\Gamma(\cH_0)$ the natural projection.
  The key ingredient in the construction of
  $\cD^{\cV} : \Gamma(\cH_0)\to\Gamma(\cH_1)$ is a projectively invariant differential operator
  $$L^{\cV}:\Gamma (\cH_0)\to \Gamma (\cV),$$
  called the \emph{(first) BGG splitting operator}:
  It is uniquely characterised by the conditions that (1) $\Pi^{\cV} \circ L^{\cV}$ is the identity map on $\Gamma(\cH_0)$ and (2) the composition
   $\nabla^\cV \circ L^{\cV}$ takes values in $\ker \partial^*\subset \Gamma (T^*M\otimes \cV)$.
  For $\si\in \Gamma(\cH_0)$, $\mcD^{\cV}(\sigma)$ is given by projecting $\nabla^{\cV} (L^{\cV}(\si))$ to $\Gamma (\cH_1)$.

  It follows straightforwardly from the construction of $\cD^\cV$ and the characterising properties of $L^{\cV}$  that the bundle map $\Pi^{\cV}$ can be used to identify
  parallel sections of $\cV$ with special solutions of the first
  BGG operator $\cD^\cV$.
   More precisely, one has:

  \begin{proposition}[\cite{CGH}] \label{normp} Let $\mathbb{V}$ be an irreducible $G$-module and set $\cV: =\cG\times_P \mathbb{V}$.
  The bundle map
  $\Pi^{\cV}$ induces an injection from the space of $\nabla^{\cV}$-parallel sections of
  $\mathcal{V}$ to a subspace of $\ker \mathcal D^{\mathcal V} \subset \Gamma(\mathcal{H}_0)$.
\end{proposition}
We will call sections in the kernel of $\cD^{\cV}$ that arise, as in Proposition \ref{normp}, from applying the projection $\Pi^{\cV}$ to $\nabla^{\cV}$-parallel sections
\emph{normal solutions} of $\cD^\cV$ (cf.\ \cite{CGHpoly, Leitner}).

\subsection{Tractor $2$-forms}
\label{subsection:tractor-2-forms}
For an oriented projective manifold $(M,\mbp)$ consider the tractor
bundle $\Wedge^2 \mcT^*$ of tractor $2$-forms.
The splitting $\mcT^* \cong \mcE(1) \oplus T^*M(1)$ determined by a choice of connection $\nabla \in \mbp$ induces a splitting $\Wedge^2 \mcT^* \cong T^*M(2) \oplus \Wedge^2 T^*M(2)$: any section $\bbT_{AB} \in \Gamma(\Wedge^2 \mcT^*)$ decomposes as
\[
	\bbT_{AB} = 2 k_b Y_{[A} Z_{B]}^b + \mu_{ab} Z_A{}^a Z_B{}^b
\]
for some $k_a \in \Gamma(T^*M(2))$ and $\mu_{ab} \in \Gamma(\Wedge^2 T^*M(2))$; more compactly, we write
\[
	\bbT_{AB} = \begin{pmatrix}k_b\\ \mu_{bc}\end{pmatrix} .
\]
With respect to $\nabla\in\mbp$ the tractor connection on $\Wedge^2 \mcT^*$ is given by
\begin{equation}\label{equation:tractor-2-form-connection}
	\nabla^{\Wedge^2\mcT^*}_a
		\left(
			\begin{array}{c}
				  k_b    \\
				\mu_{bc}
			\end{array}
		\right)
			:=
		\left(
			\begin{array}{c}
				\nabla_a k_b - \mu_{bc}\\
				\nabla_a \mu_{bc}+2 \Rho_{a[b}k_{c]}
			\end{array}
		\right) .
\end{equation}
Comparing this equation with \eqref{prol_step2} shows a close relationship between $\nabla^{\Wedge^2\mcT^*}$ and the prolongation connection
defined there, which suitably interpreted can be seen as a connection on $\Wedge^2 \mcT^*$: For a choice $\nabla\in\mbp$ one may define a connection on $\Wedge^2 \mcT^*$ by
\begin{equation}\label{equation:prolongation-vs-tractor-connection}
	\nabla^{\textrm{prol}}_a
		\left(
			\begin{array}{c}
				k_b     \\
				\mu_{bc}
			\end{array}
		\right)
			:=
	\nabla^{\Wedge^2\mcT^*}_a
		\left(
			\begin{array}{c}
				k_b     \\
				\mu_{bc}
			\end{array}
		\right)
			-
		\left(
			\begin{array}{c}
				0\\
				W_{bc}{}^d{}_a k_d
			\end{array}
		\right) .
\end{equation}
Since  both terms on the right are projectively invariant, so is $\nabla^{\textrm{prol}}$. Moreover, note that if $\nabla\in\mbp$ is a scale, then we get an identifications of the bundles $T^*M$ and $\Wedge^2 T^*M$ respectively with the weighted bundles $T^*M(2)$ and $\Wedge^2 T^*M(2)$, and $\nabla^{\textrm{prol}}$ precisely coincides with the connection defined in \eqref{prol_step2} on the special affine manifold $(M,\nabla)$.

The projection $\Pi^{\Wedge^2 \mcT^*} : \Gamma(\Wedge^2 \mcT^*) \to \Gamma(T^*M(2))$ is $\bbT_{AB} \mapsto X^A W^B{}_b \bbT_{AB}$, or
\[
	\Pi^{\Wedge^2 \mcT^*}:
	\left(
		\begin{array}{c}
			k_b     \\
			\mu_{ab}
		\end{array}
	\right)
		\mapsto
	k_b ,
\]
and the splitting operator $L^{\Wedge^2 \mcT^*} : \Gamma(T^*M(2)) \to \Gamma(\Wedge^2 \mcT^*)$ is
\begin{equation}\label{equation:splitting-operator-2-form}
	L^{\Wedge^2 \mcT^*}:
	k_b
		\mapsto
	\left(
		\begin{array}{c}
			k_b     \\
			\nabla_{[a} k_{b]}
		\end{array}
	\right) .
\end{equation}
Together \eqref{equation:tractor-2-form-connection} and \eqref{equation:splitting-operator-2-form} give that the (first) BGG operator $\mcD^{\Wedge^2 \mcT^*} : \Gamma(T^*M(2)) \to \Gamma(S^2 T^*M(2))$ corresponding to $\Wedge^2\mcT^*$ is the Killing-type operator
\begin{equation}\label{equation:BGG-operator-2-form}
	\mcD^{\Wedge^2 \mcT^*} : k_a \mapsto \nabla_{(a} k_{b)} ;
\end{equation}
in particular $\mcD^{\Wedge^2 \mcT^*}(k) = 0$ is the (now appropriately projectively weighted) Killing-type equation \eqref{equation:Killing-type}.

\subsection{Adjoint tractors} \label{adj}
For an oriented projective manifold $(M,\mbp)$ another important
tractor bundle is the bundle  $\mcA:=\mathfrak{sl}(\mcT)$
of trace-free bundle endomorphisms of $\mcT$. It is called the \emph{adjoint tractor bundle},
since it is the bundle associated to the Cartan bundle via the adjoint representation of $\textrm{SL}(n+2,\bbR)$.

With respect to a choice $\nabla\in\mbp$, any section  $\bbT^A {}_B\in\Gamma(\mcA)$ decomposes as
\begin{equation}\label{equation:adjoint-decomposition}
    \bbT^A{}_B = \xi^a W^A{}_a Y_B + \alpha X^A Y_B + \phi^a{}_b W^A{}_a Z_B{}^b + \nu_b X^A Z_B{}^b,
\end{equation}
for some $\xi^a \in \Gamma(TM)$, $\alpha \in \Gamma(\mcE)$, $\phi^a{}_b \in \Gamma(\End TM)$, and $\nu_b \in \Gamma(T^* M)$, and trace-freeness of $\bbT^A{}_{B}$ forces $\alpha = -\phi^c {}_c$. We can write such a section more compactly in block matrix notation:
\[
\bbT^A{}_B
	=
		\begin{pmatrix}
			\phi^a{}_b &  \xi ^a     \\
			\nu     _b & -\phi^d{}_d \\
		\end{pmatrix} .
\]

Moreover, from \eqref{pconn} and \eqref{tconn} one deduces that the tractor connection on $\mcA$ with respect to $\nabla$ may be written as
\begin{equation}\label{adjoint_tractor_connection}
	\nabla_c^{\mcA} \bbT^A{}_B
		=
	\nabla^\mcA_c
	\begin{pmatrix}
		\phi^a{}_b &  \xi ^a     \\
		\nu     _b & -\phi^d{}_d
	\end{pmatrix} \\
		=
	\begin{pmatrix}
		   \nabla_c \phi^a{}_b + \Rho_{cb} \xi^a + \delta^a{}_c \nu_b
		&  \nabla_c \xi^a - \phi^a{}_c - \phi^d{}_d \delta^a{}_c \\
		   \nabla_c \nu_b - \Rho_{cd} \phi^d{}_b - \phi^d{}_d \Rho_{bc}
		& -\nabla_c \phi^d{}_d - \Rho_{cd} \xi^d - \nu_c
	\end{pmatrix} .
\end{equation}
For later, let us also record that, given another connection
$\widehat\nabla \in \mbp$ related to $\nabla$ by $\Upsilon$ as in
\eqref{projective_change}, the components $\widehat\xi^a,
\widehat\phi^a{}_b, \widehat\nu_b$ of $\bbT$, determined thereby,
satisfy
\begin{equation}\label{equation:adjoint-component-transformation}
	\widehat\xi ^a     = \xi^a \textrm{,}
		\qquad
	\widehat\phi^a{}_b = \phi^a{}_b + k^a \Upsilon_b - \Upsilon_c k^c \delta^a{}_b \textrm{,}
		\qquad
	\widehat\nu     _b = \nu_b - \Upsilon_c \phi^c{}_b - \phi^c{}_c \Upsilon_b + \Upsilon_c k^c \Upsilon_b .
\end{equation}

The projection $\Pi^{\mcA}: \Gamma(\mcA) \rightarrow \Gamma(TM)$ is $\bbT^A{}_B \mapsto Z_A{}^a X^B \bbT^A{}_B$, or
\[
	\Pi^{\mcA} :
	\begin{pmatrix}
		\phi^a{}_b &  \xi ^a \\
		\nu     _b & -\phi^d{}_d
	\end{pmatrix}
		\mapsto
	\xi^a ,
\]
and the splitting operator $L^{\mcA} : \Gamma(TM) \to \Gamma(\mcA)$ is
\begin{equation}\label{equation:splitting-operator-adjoint}
	L^{\mcA} : \xi^a \mapsto
		\begin{pmatrix}
			  \nabla_b \xi^a - \frac{1}{n +2} \delta^a{}_b \nabla_c \xi^c
			& \xi^a \\
			  -\frac{1}{n + 2} \nabla_b \nabla_c \xi^c{} - \Rho_{bc} \xi^c
			& -\frac{1}{n + 2} \nabla_c \xi^c
		\end{pmatrix} \textrm{.}
\end{equation}

From \eqref{adjoint_tractor_connection} and \eqref{equation:splitting-operator-adjoint} we conclude that the (first) BGG operator corresponding to $\mathcal A$ is the map $\mathcal D^{\mcA} : \Gamma(TM) \to \Gamma((S^2 T^*M \otimes TM)_{\circ})$ given by
\begin{multline}\label{equation:BGG-operator}
	\mathcal D^{\mcA} : \xi^a
		\mapsto
	\nabla_{(b} \nabla_{c)} \xi^a+ \Rho_{(bc)} \xi^a \\
		- \tfrac{1}{n + 2}\left[(\nabla_{(b} \nabla_{d)} \xi^d+ \Rho_{(bd)} \xi^d)\delta^a{}_c + (\nabla_{(c} \nabla_{d)} \xi^d+ \Rho_{(cd)} \xi^d)\delta^a{}_b\right]
\end{multline}
and that a solution $\xi^d$ of $\mathcal D^{\mcA}$ is normal if and only if it satisfies
\begin{equation}\label{equation:adjoint-normality-conditions}
\left\{
\begin{array}{rcl}
	W_{ab}{}^c{}_d \xi^d &\!\!\!=\!\!\!& 0 \\
	C_{abd} \xi^d        &\!\!\!=\!\!\!& 0
\end{array}
\right. .
\end{equation}
Here, $(S^2 T^*M \otimes TM)_{\circ}$ indicates the bundle of totally trace-free elements of $S^2 T^*M \otimes TM$.

For later purposes we also note:

\begin{lemma}\label{parallel_adjoint_tractor_and_curvature}
Let $(M,\mbp)$ be a projective manifold equipped with a $\nabla^{\mathcal A}$-parallel section $\mathbb{T}^{A}{}_{B}\in\Gamma(\mathcal {A})$.
Then the curvature $R\in\Gamma(\Wedge^2 T^*M\otimes \mathcal A)$ of the tractor connection satisfies
\begin{equation}
R_{ab}{}^{C}{}_D\,\mathbb T{}^D{}_C=0.
\end{equation}
\end{lemma}
In particular, if $\mathbb J^{A}{}_{B}:=\mathbb{T}^{A}{}_{B}$ is a parallel complex structure on $\mathcal T$, then the tractor curvature has values in
the bundle $\mathfrak{sl}(\mathcal T, \mathbb J)\subset \mathcal A$ of trace-and-$\mathbb J$-trace free endomorphisms of $\mathcal T$.
\begin{proof}
 Suppose $\dim(M)=n+1$. We use again the notation of \nn{equation:adjoint-decomposition}.
Contracting the expression \eqref{equation:adjoint-decomposition} for $\bbT$ with the formula \eqref{tractor_curvature} for the tractor connection gives
\[
	\bbT^D{}_C R_{ab}{}^C{}_D
		= W_{ab}{}^c{}_d \nabla_c \xi^d - C_{abd} \xi^d.
\]
where we have used that $\phi^{a}{}_b=\nabla_b \xi^a - \frac{1}{n +2} \delta^a{}_b \nabla_c \xi^c$, by the parallelicity of $\mathbb T^{A}{}_B$, and that $W_{ab}{}^c{}_c=0$.
Rewriting the first term and using the second projective Bianchi identity \eqref{diff_Bianchi_Weyl} gives that the contraction equals
\[
	\bbT^D{}_C R_{ab}{}^C{}_D = \nabla_c (W_{ab}{}^c{}_d \xi^d) - n C_{abd} \xi^d,
\]
which vanishes according to \eqref{equation:adjoint-normality-conditions}, since $\Pi^{\mathcal A}(\mathbb J)=\xi$ is a normal solution of $\mathcal D^{\mathcal A}$.
\end{proof}

\subsection{Affine and projective symmetries}\label{proj_symm_section}
Suppose that $\nabla$ is a torsion-free affine connection on a manifold $M$.
Recall that the \emph{Lie derivative} of $\nabla$ along a vector field $\xi\in\Gamma(TM)$ is
the tensor field
$$
\mcL_\xi\nabla\colon TM\to T^*M\otimes TM
$$
characterised by
$$(\mcL_\xi\nabla)(\eta):=\mcL_\xi(\nabla \eta)-\nabla \mcL_\xi \eta$$ for any vector field $\eta\in\Gamma(TM)$.
Note that torsion-freeness of $\nabla$ implies that $\mcL_\xi\nabla$ defines a section
of $\textstyle{S^2 T^*M \otimes TM}$.
A straightforward computation shows that
\begin{equation}\label{Lie_deriv}
(\mcL_\xi\nabla)_{ab}{}^c=\nabla_a\nabla_b \xi^c+R_{da}{}^c{}_b \xi^d=\nabla_{(a}\nabla_{b)} \xi^c+R_{d(a}{}^c{}_{b)} \xi^d,
\end{equation}
where the second identity is a manifestation of the Bianchi identity $R_{[ab}{}^c{}_{d]}=0$.

A vector field $\xi$ on $M$ is called an \emph{affine symmetry} (respectively \emph{projective symmetry}) of $\nabla$,
if its (local) flow preserves $\nabla$ (respectively the projective class $[\nabla]$), which is the case if and only if $\mcL_\xi\nabla=0$ (respectively the trace-free part
of $\mcL_\xi\nabla$, equivalently the trace-free part $\mcL_\xi\widehat\nabla$ for any $\widehat\nabla\in[\nabla]$, equals zero).

Invoking the curvature decomposition \eqref{decp}, we deduce from \eqref{Lie_deriv}:

\begin{theorem}\label{theorem:projective-symmetry}
Suppose $(M,\bold p)$ is a projective manifold and let $\nabla\in\bold p$ be a connection in the projective class.
\begin{enumerate}
\item The differential operator $\cD^{\textrm{sym}} : \Gamma(TM) \to \Gamma((S^2 T^*M \otimes TM)_{\circ})$,
\begin{equation*}
\mathcal D^{\textrm{sym}}: \xi^c\mapsto (\nabla_{(a} \nabla_{b)} \xi^c + \Rho_{(ab)} \xi^c + W_{d(a}{}^c{}_{b)} \xi^d)_{\circ}
		=\mathcal D^{\mcA}(\xi)_{a}{}^c{}_{b} + W_{d(a}{}^c{}_{b)} \xi^d,
\end{equation*}
where $\cdot_{\circ}$ denotes taking the trace-free part, is projectively invariant, that is, independent of the choice of connection in $\bold p$.
\item A vector field $\xi \in \Gamma(TM)$ is a projective symmetry of $\bold p$ if and only if $\xi$ lies in the kernel of $\mathcal D^{\textrm{sym}}$.
\end{enumerate}
In particular, a (normal) solution $\xi$ of $\mathcal D^{\mcA}(\xi)=0$ is a projective symmetry if and only if $W_{d(a}{}^c{}_{b)} \xi^d=0$ (respectively $W_{da}{}^c{}_{b} \xi^d=0$).
\end{theorem}
\begin{proof}
The first statement just follows from the total trace-freeness of $W_{ab}{}^c{}_d$ and from the projective invariance of $\mathcal D^{\mcA}$ and  $W_{ab}{}^c{}_d$. The second statement follows from inserting
\eqref{decp} into formula \eqref{Lie_deriv}. Moreover,
the characterisation of solutions $\xi$ of $\mathcal D^{\mcA}(\xi)=0$ that are projective symmetries is evident and the one for normal solutions follows directly from the normality conditions \eqref{equation:adjoint-normality-conditions} and the Bianchi identity
$W_{ab}{}^c{}_d=-2W_{d[a}{}^c{}_{b]}$ of the projective Weyl curvature.
\end{proof}

\subsection{Proof of Theorem A}\label{Section: Proof of Theorem A}

\begin{proof}[Proof of Theorem A]
Suppose $(M, g, k)$ is a Sasaki--Einstein manifold, $g$ of signature $(2 p - 1, 2 q)$ and denote by $\nabla$ the Levi-Civita connection of $g$.

Now define the tractor metric $$h_{AB} := Y_A Y_B + Z_A{}^a Z_B{}^b g_{ab}\in\Gamma(S^2\mcT^*),$$
which has signature $(2 p, 2 q)$.
Equation \eqref{pconn} is equivalent to the formulas $\nabla_c Y_A = \Rho_{ca} Z_A{}^a$ and $\nabla_c Z_A{}^a = -\delta^a{}_c Y_A$, which together give that
$\nabla_c^{\mcT} h_{AB} = 0$, since $\Rho_{ab} = g_{ab}$ by dint of being Sasaki--Einstein (cf.\, also\ Theorem \ref{Jpar}).

Next, define the tractor $2$-form $$\Omega_{CB}:= 2 k_b Y_{[C} Z_{B]}{}^b + (\nabla_c k_b) Z_C{}^c Z_B{}^b \in \Gamma(\Wedge^2 \mcT^*).$$
Since $k$ is Killing, Proposition \ref{prolongprop} implies that $\nabla^{\textrm{prol}}_c \Omega_{AB} = 0$. On the other hand, since $g_{ab}=\Rho_{ab}$, we deduce from \eqref{decp}
that
$$W_{abcd}=R_{abcd}-g_{ac}g_{bd}+g_{ad}g_{bc},$$
where $R_{abcd}=R_{ab[cd]}$ is the curvature of $\nabla$. Since the right-hand side is skew in $c$ and $d$, so is $W_{abcd}$ (in fact this is true for the projective Weyl tensor of any Einstein metric).
Hence, condition (b) in the characterisation of Sasaki structures in Theorem \ref{Jpar} gives $W_{bc}{}^d{}_a k_d = 0$; substituting in \eqref{equation:prolongation-vs-tractor-connection} then gives $\nabla^{\mcT}_c \Omega_{AB} = 0$.

Now, define the $\nabla^\mcT$-parallel tractor endomorphism $\bbJ \in \Gamma(\End \mcT)$ by raising an index of $\Omega$ with $h$:
\[
	\bbJ^A{}_B
		:= h^{AC} \Omega_{CB}
		 =
		 	\begin{pmatrix}
		 		\nabla_b k^a & k^a \\
		 		        -k_b & 0
			\end{pmatrix} .
\]
To establish that $(h, \Omega, \bbJ)$ defines a $\nabla^{\mcT}$-parallel Hermitian structure on $\mcT$, it remains to show that $\bbJ$ is a complex structure: Squaring and using condition (a) of
the Definition \ref{definition:Sasaki-structure} of a Sasaki structure gives
\begin{equation}\label{tractor_J_squared}
	\bbJ^A{}_C \bbJ^C{}_B
		=
			\begin{pmatrix}
				(\nabla_c k^a) \nabla_b k^c - k^a k_b & k^c \nabla_c k^a \\
				k_c \nabla_b k^c & -1
			\end{pmatrix}.
\end{equation}

Differentiating condition (a) of Definition \ref{definition:Sasaki-structure} and using that $k$ is Killing yields $$0=(\nabla_ak_c)k^c=-(\nabla_{c}k_a)k^c,$$ which shows that the off-diagonal entries of \eqref{tractor_J_squared} vanish.
Differentiating this equation once more gives
\[
	0 = \nabla_b(k^c \nabla_c k^a) = (\nabla_b k^c) \nabla_c k^a + k^c \nabla_b \nabla_c k^a.
\]
Hence, (a) and (b) of Definition \ref{definition:Sasaki-structure}
imply that the top-left entry of \eqref{tractor_J_squared} equals
$-\delta^{a}{}_b$, which shows that $\bbJ$ is a complex structure on
$\mcT$.

To complete the proof of Theorem A it remains to show that if $M$ is simply connected, it also admits a parallel tractor complex volume form. For such $M$, the existence of such forms is equivalent to the vanishing of the trace $R_{ab}{}^C{}_C$ and the $\bbJ$-trace $R_{ab}{}^C{}_D \bbJ^D{}_C$, which follows from Lemma \ref{parallel_adjoint_tractor_and_curvature}.
\end{proof}

\section{Elemental geometries and projective holonomy}\label{hol-sec}
In this section we shall discuss in detail various holonomy reductions of oriented projective structures. We start by a brief review on the general theory of
such holonomy reductions, for details see \cite{CGH}, where such a theory was developed for holonomy reductions of general Cartan geometries; specifically
for the projective case see also \cite{CGHjlms}.

\subsection{Holonomy reductions of projective structures}
\label{subsection:holonomy-reductions-projective}
Suppose $(M,\mbp)$ is a connected oriented projective manifold of dimension $n+1\geq 2$ and let $P\subset G:=\textrm{SL}(n+1,\bbR)$ be the stabiliser of a ray in $\bbR^{n+1}$.
The \emph{(projective) holonomy} of $(M, \mbp)$ is defined to be the holonomy of the standard
tractor connection $\nabla^{\mcT}$ of $(M,\mbp)$, which is a conjugacy class of subgroups in $G$, since $\nabla^{\mcT}$ is induced from a $G$-principal connection as explained in Section \ref{tr}.
In the following sections we shall investigate the geometric implications of oriented projective structures admitting holonomy reductions that arise from certain parallel tractor fields,
equivalently from normal solutions of certain first BGG operators.

Suppose $s\in\Gamma(\cV)$ is a $\nabla^\cV$-parallel section of a
tractor bundle $\cV$ corresponding to a $G$-module $\mathbb V$.
First, since $s$ is a section of the associated bundle
$\hat{\mathcal{G}}\times_G \mathbb{V}$ it corresponds to a
$G$-equivariant map
$$
\underline{s}:\hat{\mathcal{G}}\to \bbV.
$$
Next, the parallelicity of $s$ implies that the $G$-orbit $\mathcal
O:=G\underline{s}(u)\subset \mathbb V$ for $u\in \mathbb{\hat{\mathcal{G}}}$ does
not depend on $u\in \hat{\mathcal{G}}$ (see e.g.\ \cite[Lemma 2.3]{CGH}); it is
referred to as the \emph{$G$-type} of $s$. Since the tractor
connection is not just induced by any $G$-principal connection but
rather a $G$-principal connection arising from a Cartan connection of
type $(G,P)$ by extension (see Section \ref{tr}), any point $x\in M$
has a well-defined \emph{$P$-type} with respect to $s$, given by the
$P$-orbit $P\underline{s}(u)\subset \mathcal O\subset\mathbb V$, where
$u\in \mathcal{G}\subset \hat{\mathcal{G}}$ is in the fibre $\mathcal{G}_x$ over $x$.  In
\cite{CGH} it was shown that this yields a decomposition according to
$P$-type of $M$ into a disjoint union of initial submanifolds---
\begin{equation}\label{curved_orbit_decomp}
M=\bigsqcup_{i\in P \backslash \mathcal O} M_i
\end{equation}
---where  $P \backslash \mathcal O$ denotes the set of $P$-orbits in $\mathcal O$ and $M_i$ denotes the set of all points in $M$ of $P$-type $i$. The initial submanifolds $M_i$ are called \emph{curved orbits} and
the decomposition \eqref{curved_orbit_decomp} the \emph{curved orbit decomposition} of $M$ determined by $s$, since in the case of the flat homogeneous model $\Sigma:=G/P=\mathbb P_+(\mathbb R^{n+2}) \cong S^{n + 1}$
the curved orbit decomposition can be identified with an actual decomposition of $G/P$ into orbits of a subgroup $H$ of $G$. Let us explain this a bit more: Fix an element $\alpha\in\mathcal O$ and let $H\subset G$ be the isotropy group of $\alpha\in\mathcal O\subset \mathbb V$ in $G$ and identify $\mathcal O=G/H$. In view of \eqref{tractor_bundle_homog_model}, it is easy to see (cf. \cite[Section 2.5]{CGH}) that associating to each point $x=gP\in G/P$ its $P$-type induces a bijection
$$H \backslash G / P \cong P \backslash \mathcal O = P \backslash G / H .$$
Hence the initial submanifolds $\Sigma_i\subset \Sigma=G/P$ of \eqref{curved_orbit_decomp} can be identified with the $H$-orbits in $G/P$. Moreover, note that any $\Sigma_i$ can be identified with a homogeneous space
$\Sigma_i=H/(H\cap P_i) $, where $P_i$  is conjugate to $P$. Hence, $\Sigma_i$ can be viewed as naturally equipped with a homogeneous geometric structure whose automorphisms are exactly given by left multiplication by $H$.
One says in this case that $\Sigma_i$ is a \emph{Klein geometry} (or \emph{flat Cartan geometry}) of the form $H/(H\cap P_i)$.

Suppose now again that $(M,\mbp)$ is a general projective manifold equipped with $\nabla^\cV$-parallel section $s$ of a tractor bundle $\cV$ and consider a parallel section of the (flat) tractor connection $\nabla^\cV$
on $\Sigma=G/P$ of the same $G$-type as $s$ (since $\nabla^\cV$ is flat on $\Sigma$ this always exists). In \cite[Theorem 2.6]{CGH} it is shown that the curved orbit decomposition \eqref{curved_orbit_decomp}
determined by $s$ is in a certain way locally diffeomorphic to the orbit decomposition of the model and that any initial submanifold $M_i$ comes naturally equipped with a (non-flat in general) Cartan geometry
of type $(H, H\cap P_i)$---the type of the Klein geometry $\Sigma_i=H/(H\cap P_i)$. In the following sections familiarity with Cartan connections will not be needed, all the projective holonomy reductions
will be discussed in terms of tractors and will see how the induced geometries on the curved orbits arise without explicit reference to the projective Cartan connection.

Finally, let us remark that the zero set of a normal BGG-solution (corresponding to a parallel tractor) is easily seen (from its definition) to always consist of a union of $P$-types (see \cite[Section 2.7]{CGH}) and
from the local identification of curved orbits with actual orbits in the model one sees that the forms of these zero sets can be read off from the model \cite{CGHjlms, CGH}.

\subsection{Reduction to \texorpdfstring{$\SO(p', q')$}{SO(p', q')}: Parallel tractor metric}
\label{SOred}

A simple example of holonomy reduction and curved orbits, which is
important for our discussion below, arises when a projective manifold
$(M,\mbp)$ is equipped with a parallel tractor metric of some
signature. The geometry of oriented projective structures equipped
with a parallel metric on its standard tractor bundle has been already
studied in detail in the literature \cite{CGcompact, CGHjlms, CGH}.
In this section we will therefore be brief, minimally illustrating the
above general theory and recording the facts that we will need later.

First we consider the homogeneous model for projective geometry $\Sigma:= S^{n+1}=\mathbb{P}_+(\mathbb{R}^{n+2})$.
Fix a nondegenerate symmetric bilinear form $h$ on $\Bbb R^{n + 2}$, say, of signature $(p', q')$, and consider its stabiliser $H = \SO(h) \cong \SO(p', q')$ in $\SL(n + 2, \bbR)$.
Then, the $H$-orbits are determined by the strict sign of the restrictions of $h$ to the rays $[X]_+ \in \mathbb{P}_+(\mathbb{R}^{n+2})$. In particular, if $h$ is definite
there is just one $H$-orbit; otherwise (that is, if $p', q' > 0$) then
there are three $H$-orbits:
$$
\Sigma = \Sigma_{+} \cup \Sigma_0 \cup \Sigma_-,
$$
where $\Sigma_\pm$ are open submanifolds and $\Sigma_0$ an embedded hypersurface of $\Sigma$. Moreover, $\Sigma_0$ separates $\Sigma_+$ and $\Sigma_-$, that is, $\Sigma\setminus \Sigma_0$ is disconnected and
$\Sigma_+$ and $\Sigma_-$ are both unions of connected components thereof.
\medskip

Now suppose that we have a projective manifold $(M,\mbp)$ of dimension $n+1$
and equipped with a parallel tractor metric $h$ of signature
$(p',q')$, and let $$\tau:= \Pi^{S^2 \mcT^*}(h) = h_{AB}X^AX^B\in\Gamma(\mcE(2))$$
denote its projecting part (where now $X$ is as defined in \nn{euler}).  If $h$ is definite (so that $p' = 0$ or $q' = 0$) then $\tau$ vanishes nowhere. If not, then the curved orbit decomposition of $M$  determined by $h_{AB}$ is given according to the strict sign of $\tau$ (which is well defined, as $\tau$ is section of an oriented line bundle):
\begin{equation}\label{decom_h}
M = M_+ \cup M_0 \cup M_-,
\end{equation}
where $M_\pm$ are open submanifolds (if not empty) and $M_0$ (if not empty) is an embedded hypersurface of $M$, as this is the case in the
model. In either case, on the open set where $\tau$ is nowhere zero,
$\tau$ determines a scale connection $\nabla\in\mbp$ such that
$\nabla\tau=0$, as explained in Section \ref{proj-sec}. Since $h$ is parallel, it
is the image $L^{S^2 \mcT^*}(\tau)$ of the BGG splitting operator $L^{S^2 \mcT^*} : \Gamma(\ce(2))\to
\Gamma(S^2\mcT^*) $ and, with respect to the splitting of $S^2\mcT^*$
determined by $\nabla$, it takes the form
\begin{equation}\label{tractor_metric}
	h_{AB} = \tau Y_A Y_B + \tau \Rho_{ab} Z_A{}^a Z_B{}^b,
		\qquad
	\textrm{or}
		\qquad
	\begin{pmatrix}
		\tau & 0 \\
		   0 & \tau\Rho_{ab},
	\end{pmatrix} ,
\end{equation}
where the Schouten tensor $\Rho_{ab}\in \Gamma(S^2T^*M)$ has evidently to be
nondegenerate. That $h$ is parallel then further implies that
\begin{equation}
\tau\nabla_{a}\Rho_{bc}=\tfrac{1}{n}\tau\nabla_{a}\Ric_{bc}=0,
\end{equation}
and so $\nabla^{\pm} := \nabla\vert_{M_{\pm}}$ is the Levi-Civita connection of the metric $g_{bc} := \Rho_{bc}\vert_{M_{\pm}}$.

If $M_0$ is not empty and hence an embedded hypersurface of $M$, then, as in the model, it also separates $M_+$ and $M_-$. This can for instance directly be seen as follows: With respect to any scale connection $\nabla\in\bold{p}$, the
condition that $h$ is parallel implies that $h_{AB} X^B = \tau Y_A +
\frac{1}{2} \nabla_a \tau Z_A{}^a$, and since $h$ is nondegenerate,
this quantity vanishes nowhere. In particular at any point $x_0$ of
$M_0 = \{x\in M~:~\tau(x) = 0\}$ we must have $\nabla_a \tau (x_0)
\neq 0$, which together with the fact that $\tau$ has different signs on $M_+$ and $M_-$ implies the claim.

The following theorem is the main result we need from this section. For part (a)
see also \cite[\S5]{ArmstrongP1}.
\begin{theorem}[\cite{CGcompact, CGHjlms, CGH}]\label{theorem:orthogonal-reduction}
Suppose $(M, \mbp)$ is a  projective manifold of
dimension $n+1 \geq 3$ equipped with a parallel tractor metric
$h_{AB}\in \Gamma(S^2\mcT^*)$ of signature $(p',q')$.  Set
$\tau:=h_{AB}X^AX^B\in\mcE(2)$. Then $M$ admits a stratification
according to the strict sign of
$\tau$ as in \eqref{decom_h}; if $p' = 0$ or $q' = 0$ it is trivial (i.e. $M = M_-$ or $M = M_+$, respectively). Moreover:
\begin{enumerate}
\item	If $M_+$ is nonempty, the metric
        $g^+_{ab}:=\Rho_{ab}\vert_{M_+} \in\Gamma(S^2 T^*M_+)$ on
        $M_+$ has signature $(p' - 1, q')$.  If $M_-$ is nonempty, the
        metric $g^-_{ab} := \Rho_{ab}\vert_{M_-} \in \Gamma(S^2
        T^*M_-)$ has signature $(q' - 1, p')$.  In both cases the
        metric $g^{\pm}$ is Einstein of positive scalar curvature, and
        the Levi-Civita connection $\nabla^{\pm}$ of $g^{\pm}$
        satisfies $\nabla^{\pm} \tau\vert_{M_{\pm}} = 0$ and $\nabla^{\pm}\in\bold{p}\vert_{M_\pm}$.

\item If $M_0$ is nonempty, it is a separating, embedded hypersurface and, in a scale $\nabla$
  along $M_0$, $\frac{1}{2}\nabla_a\nabla_b\tau $ induces a
  nowhere-degenerate section of $S^2T^*M(2)|_{TM_0}$ that defines a
  conformal structure $\mbc$ of signature $(p'-1,q'-1)$ on $M_0$.

\item The pseudo-Riemannian manifolds $(M_\pm, g^{\pm}_{ab})$ are
  projectively compactified respectively by $(M\setminus M_{\mp}, \mbp \vert_{M\setminus M_{\mp}})$. In
  both cases the projective compactification is of order $2$, and the
  projective infinity is $(M_0, \pm\mbc)$.
\end{enumerate}
\end{theorem}

Here, a connection $\nabla$ on $M_{\pm}$ is \textit{projectively
  compact} (with boundary $M_0$) of order $\alpha > 0$ if and only if
for all $x \in M_0$ there is a local defining function $\rho$ in a
neighborhood $U$ of $x$ in $M_{\pm} \cup M_0$ so that the modified
(projectively equivalent) connection $\hat\nabla = \nabla +
\frac{1}{\alpha \rho} d\rho$ (where the notation means that $\nabla$
is projectively modified according to \eqref{projective_change} via
the $1$-form $\frac{1}{\alpha \rho} d\rho$) admits a smooth extension
from $U \cap M_{\pm}$ to all of $U$; this condition is independent of
the choice of defining function. A metric is projectively compact if
and only its Levi-Civita connection is projectively compact. In the
metric setting, and when $\alpha = 2$, for any defining function
$\rho$, $(\hat\nabla d\rho) \vert_{TM_0}$ is a nondegenerate metric on
$M_0$. The conformal class it defines is independent of the choice of
defining function and coincides with $\pm\mbc$.

Given an oriented projective manifold $(M, \mbp)$ equipped with a parallel tractor metric $h_{AB}$, by rescaling that tractor metric by a constant we may (and henceforth do) assume that the projective tractor volume form $\vol$ is a volume form for $H$, that is, that $\vol_x$ is a volume form on $\mcT_x$ for $H_x$ for all $x \in M$.

\subsection{Reduction to \texorpdfstring{$\Sp(2 m + 2, \bbR)$}{Sp(2m+2, R)}: Parallel tractor symplectic form}
\label{Sp-reduction}
In this section we discuss holonomy reductions of projective manifolds of dimension $2m+1$ induced by parallel tractor symplectic forms, that are, parallel sections
\[
	\Omega_{AB} \in \Gamma(\Wedge^2 \mcT^*)
\]
nondegenerate in the sense that $\Omega^{\wedge (m + 1)} \in \Gamma(\Wedge^{2 m + 2} \mcT^*)$ vanishes nowhere. We shall see that such reductions give rise to contact structures equipped with some additional geometric structure, making them contact projective structures in the sense of \cite{Fox}. We therefore start by reviewing some basic facts about contact and contact projective structures.

\subsubsection{Contact structures}
\label{subsec_contact_structures}
Suppose $M$ is a manifold of dimension $n+1\geq 2$ equipped with a hyperplane distribution $H\subset TM$, that is, a vector subbundle $H \subset TM$ of corank $1$. Let us write $Q:=TM/H$ for the quotient line bundle. It will be useful to introduce abstract indices for sections of $H$ respectively $H^*$; we shall use the Greek alphabet to index such sections, which means we write $X^\alpha$, respectively $\omega_{\alpha}$, for a section of $H$, respectively $H^*$, and similarly for sections of any tensor product of these bundles. Note that we have the following short exact sequences
\begin{equation}\label{short_exact_sequence}
0\rightarrow H\rightarrow TM\rightarrow Q\rightarrow 0\qquad\textrm{ and }\qquad
0\rightarrow Q^*\rightarrow T^*M\rightarrow H^*\rightarrow 0,\nonumber
\end{equation}
and let us write $\Pi_{\alpha}^a$ for the inclusion of $H$ into $TM$ respectively the projection of $T^*M$ to $H^*$. The projection $TM\rightarrow Q$ can be viewed as a $Q$-valued $1$-form,
which we denote by $k$ and call it the \emph{defining conformal $1$-form} of $(M,H)$. Moreover, for sections $\xi,\eta\in\Gamma(H)$ set
\begin{equation}\label{Levi_bracket}
\mathcal L(\xi,\eta):=-k([\xi,\eta]),
\end{equation}
where $[\cdot,\cdot]$ denotes the Lie bracket of vector fields.
Note that \eqref{Levi_bracket} defines a vector bundle map $\mathcal L: \Wedge^2 H\rightarrow Q$, which is called the \emph{Levi bracket} of $(M,H)$.
If $\mathcal L$ is nondegenerate at any point $x\in M$, then $H$ is called a \emph{contact distribution};
in this case $n=2m$ is forced to be even, and we will refer to $k$ also as the \emph{conformal contact form}. For a contact manifold $(M,H)$ non-degeneracy of the
Levi bracket $\mathcal L_{\alpha\beta}\in\Gamma(\Wedge^2H^*\otimes Q)$ implies that it has an inverse $\mathcal L^{\alpha\beta}\in\Gamma(\Wedge^2H\otimes Q^*)$ such that
\begin{equation*}
\mathcal L_{\alpha\beta}\mathcal L^{\beta\gamma}=-\delta^{\gamma}{}_{\alpha} .
\end{equation*}
Note that, via the inclusion $H\rightarrow TM$, this inverse $\mathcal L^{\alpha\beta}$ has a canonical extension to a section $\mathcal L^{ab}:=\Pi_{\alpha}^a\Pi_{\beta}^b\mathcal L^{\alpha\beta}\in\Gamma(\Wedge^2TM\otimes Q)$ with the property that $\mathcal L^{ab}k_a=-\mathcal L^{ba}k_a=0$.

Suppose now that $(M, H)$ is a contact manifold. Then one can identify $Q$ with a root of $\mathcal K=(\Wedge^{2m+1} TM)^{\otimes 2}$ provided that $(M,H)$ is \emph{co-orientable}, that is, that the line bundle $Q$ is orientable
(equivalently, trivialisable). To see that, note first that any (local) trivialisation of $Q$ can be viewed as a (local) $1$-form
$\theta\in\Gamma(T^*M)$ with annihilator $H$. Such forms are called \emph{(local) contact forms}. Having chosen a (local) contact form $\theta$, we have $(\theta\circ\mathcal L)|_{\Wedge^2 H}=d\theta|_{\Wedge^2 H}$ and
non-degeneracy of $\mathcal L$ is obviously equivalent to
\begin{equation}\label{volume_form_ass_to_contact_form}
\epsilon^{(\theta)}:=\theta\wedge \underbrace{d\theta\wedge\cdots  \wedge d\theta}_m\in\Gamma(\Wedge^{2m+1}T^*M)
\end{equation}
being nowhere vanishing (on its domain). In particular, the square of $\epsilon^{(\theta)}$  defines a (local) trivialisation of $\mathcal K$. Since any other (local) contact form can be written as $\tilde\theta=f\theta$
for some nowhere vanishing smooth function $f$, we see that $\epsilon^{(\tilde\theta)}=f^{m+1}\epsilon^{(\theta)}$ and hence $\mathcal K=Q^{2m+2}$. Therefore, if we assume that $Q$ is orientable and fix an orientation on $Q$,
in which case $(M,H)$ is called \emph{co-oriented}, we obtain an identification $Q \cong \mathcal E (2)$, and hence the conformal contact form $k$ can be identified with a section of $T^*M(2)$.

For any co-oriented contact manifold we call a contact form \emph{positive or compatible}, if it is (viewed as a section of $Q^*$) positive with respect to the orientation on $Q=\mathcal E(2)$.
Let us write $\mathcal E_+(2)$ for the positive elements of $\mathcal E(2)$ and call its nowhere vanishing sections \emph{positive scales of weight $2$}. Then mapping $\tau\in\Gamma(\mathcal E_+(2))$ to
$k\tau^{-1}\in \Gamma(T^*M)$ induces a bijection between positive scales of weight $2$ and positive contact forms. Moreover, note that any co-oriented contact manifold determines an orientation on $M$ induced by the volume form $\epsilon^{(\theta)}$  \eqref{volume_form_ass_to_contact_form} for any choice $\theta$ of positive contact form. As explained in Section \ref{notn}, such an orientation on $M$ can be encoded
by a nowhere-vanishing tautological weighted $(2m+1)$-form $\epsilon_{a_1 \cdots a_{2 m + 1}}\in\Gamma(\Wedge^{2m+1} T^*M(2m+2))$,
which in this case is characterised by
\begin{equation}\label{orientation_induced_by_contact_structure}
\epsilon_{a_1 \cdots a_{2 m + 1}}=\tau^{m+1}\epsilon^{(\theta)}_{a_1 \cdots a_{2 m + 1}},
\end{equation}
for any positive scale $\tau\in\mathcal E_+(2)$ with corresponding contact form $\theta=\tau^{-1}k$. Note that this shows in particular
 that co-orientable contact structures can only exist on orientable manifolds. In the sequel we shall always assume
that contact manifolds are co-oriented.

Suppose now that $(M,H)$ is a co-oriented contact manifold.
For later purposes let us remark that any choice  of positive scale $\tau\in\mathcal E_+(2)$, equivalently any choice of a positive contact form $\theta=\tau^{-1}k$,
determines a vector field $t\in\Gamma(TM)$, called  the \emph{Reeb vector field associated to $\theta$}, uniquely characterised by the properties
\begin{equation}\label{Reeb}
\theta(t)=1 \qquad\qquad\textrm{and}\qquad\qquad t \hook d\theta = 0.
\end{equation}
In particular any choice of scale $\tau\in\mathcal E_+(2)$ induces a splitting $\Pi_{a}^{\alpha}: TM\rightarrow H$ of \eqref{short_exact_sequence} characterised by
$$\Pi_{\alpha}^b\Pi_{a}^{\alpha}=\delta^b{}_a-\theta_{a}t^{b}.$$
Now set $\omega_{ab}:=(d\theta)_{ab}\in\Gamma(\Wedge^2 T^*M)$ and
$\omega_{\alpha\beta}:=\Pi_{\alpha}^a\Pi_{\beta}^b\omega_{ab}\in\Gamma(\Wedge^2H^*)$.
Then non-degeneracy of $\omega_{\alpha\beta}$ implies the existence of a unique section $\omega^{\alpha\beta}\in\Gamma(\Wedge^2H)$ such that
\begin{equation}\label{inverse_omega1}
\omega^{\alpha\beta}\omega_{\beta\gamma}=-\delta^{\alpha}{}_{\gamma},\end{equation}
equivalently
of a unique section $\omega^{ab}\in\Gamma(\Wedge^2TM)$ such that
\begin{equation}\label{inverse_omega2}
\omega^{ab}\omega_{bc}=-\delta^a{}_b+\theta_ct^a,\end{equation}
where $\omega^{ab}=\Pi_{\alpha}^a\Pi_{\beta}^{b}\omega^{\alpha\beta}$. Evidently, we have $\mathcal L_{\alpha\beta}=\tau\omega_{\alpha\beta}$ and
$\mathcal L^{\alpha\beta}=\tau^{-1}\omega^{\alpha\beta}$ respectively $\mathcal L^{ab}=\tau^{-1}\omega^{ab}$. The section
\begin{equation}\mathcal L_{ab}:=\tau\omega_{ab}\in\Gamma(\Wedge^2T^*M(2)),\end{equation}
is an extension of $\mathcal L_{\alpha\beta}\in\Gamma(\Wedge^2H^*(2))$ to a section of $\Wedge^2T^*M(2)$, which, in contrast to $\mathcal L^{ab}$, depends on the choice of scale $\tau$.
Explicitly, any other scale $\tilde\tau\in\mathcal E_+(2)$ can be written as $\tilde\tau=e^{-2u}\tau$ for a smooth function $u$ implying $\tilde\theta=e^{2u}\theta$ and hence we have
\begin{alignat}{3}
	\widetilde\omega_{ab}
		&= e^{2u}\omega_{ab}+4e^{2u}\Upsilon_{[a}\theta_{b]} \qquad
		&&\textrm{respectively}\qquad&
	\widetilde{\mathcal L}_{ab}
		&= \mathcal L_{ab}+4\Upsilon_{[a} k_{b]} , \label{change_mcL} \\
	\widetilde\omega^{ab}
		&= e^{2u}\omega^{ab} \qquad
		&&\textrm{respectively}\qquad&
	\widetilde{\mathcal L}^{ab}
		&= \mathcal L^{ab} , \\
	\tilde t^a
		&= e^{-2u}t^a+2e^{-2u}\Upsilon_b\omega^{ab} \qquad
		&&\textrm{respectively}\qquad&
	\tilde r^a
		&= r^a+2\Upsilon_b\mathcal L^{ab} , \label{change_Reeb}
\end{alignat}
where $\Upsilon=du$ and $r^a=\tau^{-1}t^a\in\Gamma(TM(-2))$.

\subsubsection{Contact projective structures}
Recall that a distribution $H\subset TM$ on a manifold $M$ is \emph{totally geodesic} for an affine connection $\nabla$,
if any geodesic of $\nabla$ is either tangent to $H$ at no times or tangent to $H$ at all times; we call geodesics of an affine connection $\nabla$ with the latter property
\emph{$H$-geodesics} of $\nabla$. Given a distribution $H$, an affine connection $\nabla$ for which $H$ is totally geodesic is said to be \emph{compatible}
with $H$.

Suppose now that $(M,H)$ is a contact manifold.
Then any compatible affine connection $\nabla$ on $M$ determines a so-called \emph{contact projective equivalence class} $\langle \nabla \rangle$ of $H$-compatible affine connections on $M$, where elements
in $\langle \nabla \rangle$  are characterised by having the same unparameterised $H$-geodesics as $\nabla$. Removing the torsion of an affine connection does not change its geodesics, which motivates the following
definition of a contact projective structure:

\begin{definition}\label{def:proj_cont}
A \emph{contact projective manifold} is a triple $(M,H, \mbq)$,
where $M$ is a manifold of odd dimension at least $3$, $H \subset TM$ a
contact distribution, and $\mbq$ a contact projective equivalence
class of $H$-compatible torsion-free affine connections.
\end{definition}

\begin{remark}
In \cite{Fox}  Fox defines a contact projective structure as an equivalence class of $H$-compatible affine connections. Thus, for any $H$-compatible torsion-free affine connection $\nabla$, Fox' equivalence class containing $\nabla$ contains all of the connections in our equivalence class $\langle\nabla\rangle$ as well as all modifications of these connections by arbitrary torsion.

In Theorem B of \cite{Fox} Fox associates to any contact form $\theta$ a connection in his equivalence class that preserves $\theta$ and $\omega$, and all results in \cite{Fox} are expressed in terms
of data associated to such a connection. Since they preserve $H$, connections so adapted necessarily have torsion, but they are convenient for the computations carried out there. For our purposes, however, it is more convenient to restrict to working with torsion-free connections (and these do not preserve $H$).
\end{remark}

Given a contact projective manifold $(M,H, \mbq)$, for any connection $\nabla\in\mbq$ the projective class $[\nabla]$ of $\nabla$ is evidently contained in $\mbq$.
For a projective structure $\mbp$ whose elements are compatible with a given contact distribution $H$, via a mild abuse of notation we shall also denote the contact projective equivalence class $\langle \nabla \rangle$ determined by some (equivalently, every) $\nabla \in \mbp$ by $\langle \mbp \rangle$.
Next we observe a step towards an alternative characterisation of
contact projective equivalence of connections.
\begin{lemma}\label{lemma:totally-geodesic-contact-Killing}\cite[Lemma 2.1.]{Fox}
Given an affine manifold $(M, \nabla)$, a contact distribution $H \subset TM$ is totally geodesic for $\nabla$ if and only if for some (equivalently, every) contact form $\theta_a \in \Gamma(T^*M)$ of $H$ the projection of $\nabla_{(a} \theta_{b)}$ to $S^2H^*$ vanishes identically (i.e. $\Pi_{\alpha}^a\Pi_{\beta}^b\nabla_{(a} \theta_{b)} = 0$).
 \end{lemma}
\begin{proof} The claim follows from the fact that for any geodesic $\gamma$ of $\nabla$ we have
\begin{equation*}\label{nabla theta}
	\frac{d}{dt} \theta(\gamma'(t))=(\nabla_{\gamma'(t)}\theta)(\gamma'(t))+\theta(\nabla_{\gamma'(t)}\gamma'(t))=(\nabla_{\gamma'(t)}\theta)(\gamma'(t)) . \qedhere
\end{equation*}
\end{proof}

\begin{remark}
Note that the proof of Lemma \ref{lemma:totally-geodesic-contact-Killing} uses only that $H$ is a hyperplane distribution, and not that it is a contact distribution.
\end{remark}

Via the conformal contact form $k$, Lemma \ref{lemma:totally-geodesic-contact-Killing} admits also the following contact-invariant formulation.

\begin{lemma}\label{H-comp_proj_structure_1}
Suppose $(M,H)$ is a contact manifold and let $k_a\in\Gamma(T^*M(2))$ be the conformal contact form. Then an affine
connection $\nabla$ is $H$-compatible if and only if
\begin{equation}\label{nabla-k}
\Pi_{\alpha}^a\Pi_{\beta}^b\nabla_{(a}k_{b)} = 0.
\end{equation}
In particular, any representative of a contact projective equivalence class of $H$-compatible connections satisfies \eqref{nabla-k}.
 \end{lemma}
 \begin{proof}
We have seen in Section \ref{subsec_contact_structures} that for any contact form $\theta$ for $H$ there exists $\tau\in\Gamma(\mathcal E_+(2))$
such that $\tau\theta=k$, which implies $\nabla_ak_b=(\nabla_a\tau)\theta_b+\tau\nabla_a\theta_b$ for any affine connection $\nabla$ on $M$. Hence,
Lemma \ref{H-comp_proj_structure_1} is an immediate consequence of Lemma \ref{lemma:totally-geodesic-contact-Killing}.
 \end{proof}

Note that for a connection $\nabla$ the identity \eqref{nabla-k} is equivalent to the existence of a $1$-form $\eta\in\Gamma(T^*M)$ such that
\begin{equation}\label{eta_eq}
\nabla_{(a}k_{b)}=k_{(a}\eta_{b)}.
\end{equation}

Since $\nabla_{(a}k_{b)}$ is projectively invariant (see \eqref{equation:BGG-operator-2-form}), the $1$-form $\eta$ associated to an $H$-compatible connection $\nabla$ by \eqref{eta_eq}
depends only on the projective class of $\nabla$. Hence, Lemma \ref{H-comp_proj_structure_1} implies:

\begin{corollary}\label{H_comp_proj_structure_2} Suppose $(M,H)$ is a contact manifold and let $k\in\Gamma(T^*M(2))$ be its conformal contact form. Then
 a projective structure $\mbp$ is compatible with $H$ if and only if for some (equivalently, every) $\nabla\in\mbp$ the identity \eqref{nabla-k} is satisfied, which is the case
 if and only if there exists a $1$-form $\eta\in\Gamma(T^*M)$ such that for every connection $\nabla\in\mbp$ we have
\begin{equation*}
\nabla_{(a}k_{b)}=k_{(a}\eta_{b)}
\end{equation*}
\end{corollary}

The following proposition shows that a contact projective structure on a manifold $M$ is a section of an affine bundle over $M$
modelled on a certain vector subbundle of $TM\otimes S^2T^*M$.

\begin{proposition}\label{proposition:contact-projective-difference-tensor} Suppose $(M,H)$ is a contact manifold.
Then two H-compatible affine connections $\widehat\nabla$ and $\nabla$ with the same torsion are contact-projectively equivalent if and only if
\begin{equation}\label{difference_tensor}
    \widehat\nabla_a u^b = \nabla_a u^b + \Lambda^b {}_{ca} u^c,
\end{equation}
where
\begin{equation}\label{equation:contact-projective-transformation}
	\Lambda^c{}_{ba} = 2 k_{(a} \phi^c{}_{b)} + 2 \delta^c {}_{(a} \Upsilon_{b)}
\end{equation}
for some tensor $\phi^a{}_b \in \Gamma(\End_{\circ}(TM)(-2))$ and some $1$-form $\Upsilon_a\in\Gamma(T^*M)$.
Here $\End_{\circ}(TM)$ denotes the bundle of trace-free endomorphisms of $TM$.
\end{proposition}

\begin{proof}
Recall that the trace of a parameterised curve $\gamma(t)$ is an (unparameterised) geodesic for the connection $\nabla$ if and only if $\gamma'\wedge\nabla_{\gamma'}\gamma'=0$.
Hence, two $H$-compatible connections $\widehat\nabla $ and $\nabla$ that differ by a tensor of the form \eqref{equation:contact-projective-transformation} obviously have the same $H$-geodesics and hence by definition are contact-projectively equivalent. Conversely, assume $\widehat\nabla $ and $\nabla$ are two $H$-compatible connections with the same torsion and consider their difference tensor $\Lambda^b {}_{ca} := \widehat\nabla_a u^b-\nabla_a u^b \in\Gamma(TM\otimes T^*M\otimes T^*M)$. Since the connections have the same torsion, we must have $ \Lambda^b {}_{[ca]}=0$. The fact that $\widehat\nabla $ and $\nabla$ moreover have the same $H$-geodesics implies
\begin{equation}\label{equation:H-geodesic-condition}
    2 u^a u^d \Lambda^{[b}{}_{d a} u^{c]} = 0\qquad \textrm{ for all } u\in\Gamma(H),
\end{equation}
since any element in $H$ can be realised as the derivative of an $H$-geodesics. The identity \eqref{equation:H-geodesic-condition} is easily seen to be equivalent to \eqref{equation:contact-projective-transformation}, which completes the proof.
\end{proof}

\begin{remark}
In \eqref{equation:contact-projective-transformation}, the second term on the right-hand side controls projective equivalence, and the first term the structure of geodesics transverse to the contact distribution. Moreover, note that, for two $H$-compatible contact-projectively equivalent connections which not necessarily have the same torsion, the symmetric part of their difference tensor $\Lambda^c{}_{ba}$ is still of the form \eqref{equation:contact-projective-transformation}. The only difference is that $\Lambda^c{}_{ba}$ may then also a have a skew part, which precisely controls the change in torsion; one has $\widehat T_{ab}{}^c-T_{ab}{}^c=2\Lambda^c{}_{[ab]}$, where $\widehat T$ and $T$ are the respective torsion tensors of $\widehat\nabla$ and $\nabla$.
\end{remark}

Given a contact projective manifold $(M,H,\mbq)$, Corollary \ref{H_comp_proj_structure_2} shows that for any $\nabla\in\mbq$ there exists $\eta\in\Gamma(T^*M)$ such that $\nabla_{(a}k_{b)}=k_{(a}\eta_{b)}$, and that $\eta$ is a projective invariant. The following Proposition shows that $\mbq$ always contains connections with $\eta=0$.

\begin{proposition}\label{eta_Proposition} Suppose $(M,H)$ is a contact manifold. Then for any
$H$-compatible connection $\nabla$, there is a contact-projectively equivalent connection $\widehat\nabla$ for which $\hat\eta = 0$, that is, for which
$$\widehat\nabla_{(a} k_{b)} = 0.$$
Given an $H$-compatible connection $\nabla$ for which $\nabla_{(a} k_{b)} = 0$, the class of contact-projectively equivalent connections $\widehat\nabla$ with the same torsion for which $\widehat\nabla_{(a} k_{b)} = 0$ are exactly those that differ from $\nabla$ by difference tensors of the form
\begin{equation}\label{freedom_eta=0}
	\Lambda^c{}_{ba} = 2 k_{(a} \phi^c{}_{b)} + 2 \delta^c {}_{(a} \Upsilon_{b)},
\end{equation}
where $\phi^c{}_b\in \Gamma(\End_{\circ}(TM)(-2))$ satisfies $\phi^c{}_b k_c = 0$.
\end{proposition}
\begin{proof} Suppose $\nabla$ is an $H$-compatible connection and let $\eta\in\Gamma(T^*M)$  be such that $\nabla_{(a}k_{b)}=k_{(a}\eta_{b)}$.
By Proposition \ref{proposition:contact-projective-difference-tensor}, it suffices to show that we can modify $\nabla$ by a difference tensor of the form
\[
	\Lambda^c{}_{ba} = 2 k_{(a} \phi^c{}_{b)} + 2 \delta^c {}_{(a} \Upsilon_{b)}
\]
such that $\widehat\eta= 0$ for the resulting connection $\widehat\nabla$. Since
the condition is projectively invariant and projective changes are
controlled by the second term on the right-hand side, we may
as well restrict our attention to difference tensors of the form
\begin{equation}\label{contact_proj_change_only}
	\Lambda^c{}_{ba} = 2 k_{(a} \phi^c{}_{b)} \textrm{.}
\end{equation}
If $\dim M = 2 m + 1$, the induced connection $\widehat\nabla$ on $T^*M(2)$ determined by a general
difference tensor $\Lambda^c{}_{ba}$ is
\[
	\widehat\nabla_a \zeta_b = \nabla_a \zeta_b - \Lambda^c{}_{ba} \zeta_c + \tfrac{2}{2m + 2} \Lambda^c{}_{ca} \zeta_b \textrm{.}
\]
For difference tensors of the form \eqref{contact_proj_change_only}, we have (using trace-freeness of $\phi$) that $\Lambda^c{}_{ca} = k_c \phi^c{}_a$, and therefore that
\begin{align}
	\widehat\nabla_{(a} k_{b)}
		&= \nabla_{(a} k_{b)} - 2k_{(a} \phi^c{}_{b)} k_c + \tfrac{1}{m + 1} (k_c \phi^c{}_{(a}) k_{b)}\nonumber \\
		&= \eta_{(a} k_{b)} - \tfrac{2 m + 1}{m + 1} k_{(a} \phi^c{}_{b)} k_c \textrm{.}\label{change_eta}
\end{align}
In particular, we may choose $\phi^c{}_b$ such that $\phi^c{}_bk_c=\frac{m + 1}{2m + 1}\eta_b$  and hence $\Lambda^c{}_{ba}$
so that $\widehat\nabla_{(a} k_{b)} = 0$, which proves the first claim. Note that Proposition \ref{proposition:contact-projective-difference-tensor}
and \eqref{change_eta} also immediately imply the second claim.
\end{proof}

We finish this subsection by introducing a fundamental invariant of a contact projective structure: Suppose $(M, H, \mbq)$ is a contact projective manifold and let
$\tau\in\mathcal E_+(2)$ be a positive scale corresponding to a contact form $\theta=\tau^{-1}k\in\Gamma(T^*M)$. For a connection $\nabla\in\mbq$ consider
$$\nu_{abc}:=\nabla_c\omega_{ab}\in\Gamma(\Wedge^2T^*M\otimes T^*M),$$
where $\omega=d\theta$ as in Section \ref{subsec_contact_structures}. Torsion-freeness of $\nabla$ and closedness of $\omega$ imply that $\nu_{[abc]} = \nabla_{[a}\omega_{bc]}=0$.
Hence, all the possible $1$-forms one can extract from $\nabla_{a}\omega_{bc}$ by contracting with $\omega^{de}$ coincide up to multiplication by a constant.
Thus the totally-trace free part of
\begin{equation}\label{nu_H}
\nu_{\alpha\beta\gamma}:=\Pi_\alpha^a\Pi_\beta^b\Pi_\gamma^c\nu_{abc}\in\Gamma(\Wedge^2H^*\otimes H^*)\end{equation}
with respect to $\omega^{\alpha\beta}\in\Wedge^2H$ is given
by
\begin{equation}\label{trace-free_nu}
\nu_{\alpha\beta\gamma}^{\circ}:=\nu_{\alpha\beta\gamma}+\tfrac{1}{2m+1}(2\lambda_{[\beta}\omega_{\alpha]\gamma}+2\lambda_{\gamma}\omega_{\alpha\beta}),
\end{equation}
where $\lambda_{\beta}=\omega^{\gamma\alpha}\nu_{\alpha\beta\gamma}=\Pi_{\beta}^b\omega^{ca}\nu_{abc}.$

\begin{lemma}\label{torsion_lemma}
Suppose $(M, H, \mbq)$ is a contact projective manifold of dimension $2 m + 1$. For a positive contact form $\theta=\tau^{-1}k\in\Gamma(T^*M)$
and a connection $\nabla\in\mbq$ the quantity
\begin{equation}\label{nu_trace_free}
\omega^{\gamma\epsilon}\nu_{\alpha\beta\epsilon}^{\circ}=
\omega^{\gamma\epsilon}\nu_{\alpha\beta\epsilon}+\tfrac{1}{2m+1}(2\lambda_{[\beta}\delta_{\alpha]}{}^\gamma+2\lambda_{\epsilon}\omega^{\gamma\epsilon}\omega_{\alpha\beta}),
\end{equation}
depends neither on the choice of contact form $\theta$ nor on the choice of connection $\nabla\in\mbq$.
\end{lemma}
\begin{proof} Fix a positive contact form  $\theta\in\Gamma(T^*M)$.
Now let $\nu_{\alpha\beta\gamma}$ and $\hat\nu_{\alpha\beta\gamma}$ be defined as in \eqref{nu_H} with respect to $\omega$
and connections $\nabla\in\mbq$ and $\widehat\nabla\in\mbq$ respectively. Then we conclude from Proposition \ref{proposition:contact-projective-difference-tensor}
that $$\hat\nu_{\alpha\beta\gamma}-\nu_{\alpha\beta\gamma}=-2\Upsilon_{[\beta}\omega_{\alpha]\gamma}-2\Upsilon_{\gamma}\omega_{\alpha\beta},$$
for a section $\Upsilon_{\alpha}\in\Gamma(H^*)$, which shows that $\nu^{\circ}_{\alpha\beta\gamma}=\hat\nu^{\circ}_{\alpha\beta\gamma}$.
Hence, \eqref{nu_trace_free} is independent of the choice of connection $\mbq$.

Suppose now that $\tilde\theta=e^{2u}\theta$ is another positive contact form and fix a connection $\nabla\in\mbq$.
Differentiating \eqref{change_mcL} then shows that
\begin{equation}\label{nu_independent_of_contactform}
\nabla_{c}\widetilde\omega_{ab}=e^{2u}\nabla_c\omega_{ab}+ 2e^{2u}\Upsilon_{c}\omega_{ab}+8e^{2u}\Upsilon_c\Upsilon_{[a}\theta_{b]}+4e^{2u}(\nabla_c\Upsilon_{[a})\theta_{b]}+4e^{2u}(\nabla_c\theta_{[b})\Upsilon_{a]},
\end{equation}
where $\Upsilon=du$. Now let $\nu_{\alpha\beta\gamma}$ and $\tilde\nu_{\alpha\beta\gamma}$ be defined as in \eqref{nu_H} with respect to $\nabla$ and $\theta$ and $\tilde\theta$ respectively.
Torsion-freeness of $\nabla$ implies $\nabla_{[a}\theta_{b]}=\frac{1}{2}\omega_{ab}$ and hence we deduce from \eqref{nu_independent_of_contactform}
and Lemma \ref{lemma:totally-geodesic-contact-Killing} that
\begin{equation*}
\tilde\nu_{\alpha\beta\gamma}=e^{2u}\nu_{\alpha\beta\gamma}+2e^{2u}\Upsilon_{\gamma}\omega_{\alpha\beta}+2e^{2u}\Upsilon_{[\beta}\omega_{\alpha]\gamma},
\end{equation*}
where $\Upsilon_\alpha=\Pi_{\alpha}^a\Upsilon_a$.
Since $\widetilde\omega_{\alpha\beta}=e^{2u}\omega_{\alpha\beta}$ and $\widetilde\omega^{\alpha\beta}=e^{-2u}\omega^{\alpha\beta}$, we therefore get
\begin{equation*}
\widetilde\omega^{\gamma\epsilon}\tilde\nu_{\alpha\beta\epsilon}=\omega^{\gamma\epsilon}\nu_{\alpha\beta\epsilon}+2(\omega^{\gamma\epsilon}\Upsilon_{\epsilon}\omega_{\alpha\beta}+\Upsilon_{[\beta}\delta_{\alpha]}{}^\gamma)=
\omega^{\epsilon\gamma}\nu_{\alpha\beta\epsilon}+(\widetilde\omega^{\epsilon\gamma}\Upsilon_{\epsilon}\widetilde\omega_{\alpha\beta}+\Upsilon_{[\beta}\delta_{\alpha]}{}^\gamma)
\end{equation*}
which shows that \eqref{torsion} is also independent of the choice of contact form.
\end{proof}

\begin{definition} For any contact projective manifold $(M,H,\mbq)$ we call the tensor
\begin{equation}\label{torsion}
T_{\alpha\beta}{}^\gamma
	:= -2\omega^{\gamma\epsilon}\nu_{\alpha\beta\epsilon}^{\circ}\in\Gamma(\Wedge^2H^*\otimes H),
\end{equation}
the \emph{torsion} of $(M, H, \mbq)$.
\end{definition}

For later purposes let us record here that $T_{\alpha\beta\gamma}:=T_{\alpha\beta}{}^\epsilon\mathcal L_{\epsilon\gamma}\in\Gamma(\otimes^3 H^*(2))$ has the properties
\begin{equation}\label{torsion_symmetries}
T_{[\alpha\beta]\gamma}=T_{\alpha\beta\gamma},\qquad T_{[\alpha\beta\gamma]}=0,\qquad \mathcal L^{\alpha\beta}T_{\alpha\beta\gamma}=0, \qquad T_{\alpha\beta\gamma}T^{\alpha\beta\gamma}=0,
\end{equation}
where $T^{\alpha\beta\gamma}=\mathcal L^{\alpha\delta}\mathcal L^{\beta\epsilon} T_{\delta\epsilon}{}^\gamma$. Consequently, we also have
\begin{equation*}
T_{\alpha\beta\gamma}=2T_{\gamma[\beta\alpha]} , \qquad   \mathcal L^{\alpha\gamma}T_{\alpha\beta\gamma}=0 , \qquad   \mathcal L^{\beta\gamma}T_{\alpha\beta\gamma}=0.
\end{equation*}

An immediate consequence of \eqref{torsion_symmetries} is:
\begin{proposition}\emph{(cf.\,also \cite[Proposition 2.1]{Fox})} Any $3$-dimensional contact projective manifold is torsion-free.
\end{proposition}

\begin{remark}
The torsion $T_{\alpha\beta}{}^\gamma\in\Gamma(\Wedge^2H^*\otimes H)$ of a contact projective manifold $(M,H,\mbq)$ coincides with the contact torsion of $(M,H,\mbq)$ as defined by Fox in \cite{Fox}
as a certain invariant part of the torsion of a family of affine connections that are adapted to the contact projective structure and parametrised by contact forms of $H$. This contact torsion in turn also
coincides with the torsion of the canonical Cartan connection of type $(\Sp(2m+2,\bbR), Q)$ of $(M,H,\mbq)$ as constructed in \cite{Fox}, where $Q$ is the stabiliser of a ray in $\bbR^{2m+2}$.
\end{remark}

\subsubsection{Projective structures equipped with a parallel tractor symplectic form}
Suppose $(M,\mbp)$ is a projective manifold of dimension $2m+1$ equipped with a parallel tractor symplectic $2$-form
$\Omega_{AB}\in\Gamma(\Wedge^2\mathcal T^*)$. It follows that the nowhere vanishing section
\begin{equation}\label{tractor_volume_form_associated_to_Omega}
\vol := \tfrac{1}{(m + 1)!} \underbrace{\Omega \wedge \cdots \wedge \Omega}_{m + 1} \in\Gamma(\Wedge^{2m+2}\mcT^*)
\end{equation}
is parallel, which we refer to as the \emph{tractor volume form associated to $\Omega$}. Moreover, the nowhere vanishing section
\begin{equation}\label{orientation_induced_by_Omega }
\epsilon_{a_1...a_{2m+1}}:=\vol_{A_{0}A_1...A_{2m + 1}} X^{A_0} W^{A_1}{}_{a_1} \cdots W^{A_{2m + 1}}{}_{a_{2m + 1}}\in\Gamma(\Wedge^{2m+1}T^*M(2m+2)),
\end{equation}
which is independent of the choice $\nabla\in\mbp$, gives
rise to an isomorphism $$\mathcal E(-(2m+2))\cong \Wedge^{2m+2} T^*M,$$ which in turn determines an orientation on $M$ (see also Section \ref{notn}).
Note that $\epsilon_{a_1...a_{2m+1}}$ is parallel for any connection $\nabla\in\mbp$.

Suppose $(M,\mbp)$ is an oriented projective manifold with parallel tractor volume form $\vol_{A_{1}...A_{2m + 2}}$ (see Section \ref{notn}), equipped with a parallel tractor symplectic form $\Omega_{AB}$. By rescaling $\Omega$ by a constant and possibly reversing the orientation of $M$, we may (and henceforth do) assume that $\frac{1}{(m + 1)!}\Omega^{\wedge (m + 1)} = \vol$. This convention ensures that if $(M, \mbp)$ is also equipped with a parallel tractor metric metric $h_{AB}$ compatible with $\Omega$ in the sense that $h^{AC} \Omega_{CB} \in \Gamma(\End \mcT)$ is a complex structure on $\mcT$, and so that $h$ and $\Omega$ together define a parallel tractor Hermitian structure on $\mcT$, then $\vol$ is a volume form for $h$ (see the discussion after Theorem \ref{theorem:orthogonal-reduction}).

\begin{proposition}
\label{proposition:symplectic-single-curved-orbit}
Let $(M, \mbp)$ be an oriented projective manifold of dimension $n = 2 m + 1 \geq 3$ equipped with a parallel tractor symplectic form $\Omega_{AB} \in \Wedge^2 \mcT^*$. Then, the curved orbit decomposition of $M$, determined by the holonomy reduction of $\nabla^{\mcT}$ to $\Sp(2m + 2, \Bbb R)$ corresponding to $\Omega$, is trivial (that is, consists of a single curved orbit).
\end{proposition}
\begin{proof}
This follows immediately from the transitivity of the standard action of $\Sp(2 m + 2, \bbR)$ on the model $S^{2 m + 1} = \Bbb P_+(\Bbb R^{2 m + 2})$ (see Section \ref{subsection:holonomy-reductions-projective}).
\end{proof}

Since $\Omega_{AB}\in\Gamma(\Wedge^2\mathcal T^*)$ is parallel, its underlying weighted $1$-form
\begin{equation}\label{Killing_eq}
	k_a := \Pi^{\Wedge^2 \mcT^*}(\Omega)_a \in \Gamma(T^*M(2))
\end{equation}
is a normal solution of the corresponding BGG equation
\begin{equation}\label{Killing_eq}
	\nabla_{(a} k_{ b)} = 0,
\end{equation}
where $\nabla$ is any connection in $\mbp$ (cf. Section \ref{subsection:tractor-2-forms});
here normality of a solution $k_a$ is equivalent to the additional algebraic condition 
\begin{equation}\label{normality_Killing}
	W_{ab}{}^d{}_c k_d = 0\textrm{.}
\end{equation}

\begin{remark}
Given a parallel symplectic form $\Omega_{AB} \in \Gamma(\Wedge^2 \mcT^*)$, one can also form the nonzero parallel tractor $(2j)$-forms $\Omega^{\wedge j} \in \Gamma(\Wedge^{2j} \mcT^*)$,
$j = 2, \ldots, m$, whose respective projecting parts
\begin{multline*}
	(\Omega^{\wedge k})_{A_0 \cdots A_{2 j - 1}} X^{A_0} W^{A_1}{}_{a_1} \cdots W^{A_{2 j - 1}}{}_{a_{2j - 1}} \\
		= k_{[a_1} \nabla_{a_2} k_{a_3} \cdots \nabla_{a_{2 j - 2}} k_{a_{2 j - 1}]}
			\in \Gamma(\Wedge^{2j - 1} T^*M (2j))
\end{multline*}
are normal solutions for the respective BGG operators $\mcD^{\Wedge^{\ell} \mcT^*} : \Gamma(\Wedge^{\ell - 1} T^*M (\ell)) \to \Gamma((\Lambda^{\ell - 1} T^*M \otimes T^*M)_{\circ}(\ell))$,
\[
	\mcD^{\Wedge^{\ell} \mcT^*} : \sigma_{a_1 \cdots a_{\ell - 1}} \mapsto \nabla_{(b} \sigma_{a_1) a_2 \cdots a_{\ell - 1}} ,
\]
for $\ell = 2j$. Here, $(\Lambda^{\ell - 1} T^*M \otimes T^*M)_{\circ}$ is the kernel of the alternating map $\Lambda^{\ell - 1} T^*M \otimes T^*M \to \Lambda^{\ell} T^*M$. The  operators $\mcD^{\Wedge^{\ell} \mcT^*}$ are respectively the Killing-Yano operators on weighted $(\ell - 1)$-forms.
\end{remark}

The nondegeneracy of $\Omega_{AB}$ immediately implies (and in fact is equivalent to) the maxmimal nondegeneracy of $k_b$, which together with \eqref{Killing_eq} and \eqref{normality_Killing} shows:
\begin{theorem}
\cite[Theorem 3.5]{ArmstrongP1}
\label{projective_structure_with_parallel_symplectic_tractor}
Suppose $(M,\mbp)$ is a projective (necessarily odd-dimensional) manifold equipped with a parallel symplectic form $\Omega_{AB}\in\Gamma(\Wedge^2\mcT^*)$, and denote by $k := \Pi^{\Wedge^2 \mcT^*}(\Omega)$ the corresponding normal solution of \eqref{Killing_eq}. Then
\begin{equation}\label{k-->H}
H:= \ker k = \{u^a\in TM: k_au^a=0\}\subset TM
\end{equation}
is a contact distribution and $\mbp$ is compatible with $H$, that is, $(M,H,\langle\mbp\rangle)$ is a contact projective manifold. Moreover, the torsion of $(M,H,\langle\mbp\rangle)$  vanishes identically.
\end{theorem}
\begin{proof}
Suppose $M$ has dimension $2 m + 1$. Since $k_a\in\Gamma(T^*M(2))$ is a solution of \eqref{Killing_eq}, its derivative $\nabla_ak_b=:\mu_{ab}$ is a section of $\Wedge^2 T^*M(2)$ for any connection $\nabla\in\mbp$.
Then the nondegeneracy of  $\Omega_{AB}$ is equivalent to the nonvanishing of the section
\begin{align*}
\epsilon_{a_1...a_{2m+1}}
	&= \tfrac{1}{(m + 1)!}(\Omega^{\wedge (m + 1)})_{A_0 A_1 \cdots A_{2 m + 1}} X^{A_0} W^{A_1}{}_{a_1} \cdots W^{A_{2m + 1}}{}_{a_{2 m + 1}} \\
	&= \tfrac{1}{m!} (k \wedge \mu^{\wedge m})_{a_1 \cdots a_{2 m + 1}} \\
	&= \tfrac{1}{m! 2^m} [k \wedge (dk)^{\wedge m}]_{a_1 \cdots a_{2 m + 1}} \textrm{,}
\end{align*}
(where we have employed the notation of Sections \ref{tr} and
\ref{subsection:tractor-2-forms}) which shows that the kernel $H$ of $k$, defined
as in \eqref{k-->H}, is a contact distribution with conformal contact
form $k$. Since $k$ is a solution of \eqref{Killing_eq}, Corollary
\ref{H_comp_proj_structure_2} moreover shows that $\mbp$ is compatible
with $H$ and hence, by definition, $(M,H,\langle\mbp\rangle)$ is a
contact projective manifold.  To see that the torsion of
$(M,H,\langle\mbp\rangle)$ vanishes, fix a positive scale
$\tau\in\mathcal E_+(2)$ with corresponding connection $\nabla\in\mbp$ such that
$\nabla\tau=0$. Then \eqref{prol_step2} and normality
\eqref{normality_Killing} of the solution $k$ imply
\begin{equation*}
0=W_{bc}{}^d{}_a k_d=\nabla_a\mu_{bc}+2\Rho_{a[b}k_{c]}=\tfrac{1}{2}\tau\nabla_{a}\omega_{bc}+2\Rho_{a[b}k_{c]},
\end{equation*}
where $\omega_{ab}=(d\theta)_{ab}=2\nabla_{[a}\theta_{b]}=2\nabla_{a}\theta_{b}=2\tau^{-1}\mu_{ab}$ for $\theta=\tau^{-1}k$.
Hence, projecting to $H$ yields $0=\nu_{\beta\gamma\alpha}=\Pi_{\beta}^b\Pi_{\gamma}^c\Pi_{\alpha}^a\nu_{bca}$, where $\nu_{\alpha\beta\gamma}$
is defined as in \eqref{nu_H} with respect to $\nabla$ and $\theta$. This in turn shows
that the torsion of $(M,H,\langle\mbp\rangle)$, defined as in \eqref{torsion}, vanishes identically.
\end{proof}

\subsubsection{The contact projective-to-projective Fefferman-type construction}

Given a contact projective manifold $(M,H,\mbq)$ of dimension $2 m + 1$, the projective class $[\nabla]$ of any connection $\nabla\in\mbq$ contains among its geodesics the $H$-geodesics of $\mbq$.
In \cite{Fox} Fox shows, however, that there is a distinguished projective structure $\mbp$ on $M$, containing among its geodesics the $H$-geodesics of $\mbq$, with the property that its Weyl curvature $W_{ab}{}^c{}_d$ satisfies
$$W_{\alpha\beta}{}^d{}_{\gamma}k_d:=\Pi_\alpha^a\Pi_\beta^bW_{ab}{}^d{}_{\gamma}k_d=-\tfrac{1}{4}T_{\alpha\beta}{}^{\epsilon}\mathcal L_{\epsilon\gamma}.$$ In Theorem \ref{Fefferman_type_construction}
we will give a construction of this distinguished projective structure that is more direct than the one presented in \cite{Fox} and which will be more suitable for our purposes. We start by recording the following preliminary observation:

\begin{lemma}\label{First_Step}
Suppose $(M,H,\mbq)$ is a contact projective manifold with torsion $T_{\alpha\beta}{}^\gamma\in\Gamma(\Wedge^2 H^*\otimes H)$. For any choice of positive scale $\tau\in\mathcal E_+(2)$ there exists a connection
$\nabla\in\mbq$ such that
\begin{enumerate}
\item $\nabla_{a}\tau=0$
\item $\nabla_{(a}k_{b)}=0$.
\end{enumerate}
For two connections $\widehat\nabla,\nabla\in\mbq$ that satisfies $(a)$ and $(b)$ their difference tensor $\Lambda^c{}_{ba}$ defined as in \eqref{difference_tensor} is of the form
\begin{equation}\label{remaining_freedom_(a-b)}
\Lambda^c{}_{ba} = 2 k_{(a} \phi^c{}_{b)},\end{equation}
where $\phi^c{}_b\in \Gamma(\End_{\circ}(TM)(-2))$ satisfies $\phi^c{}_b k_c = 0$. For any connection $\nabla\in\mbq$ satisfying $(a)$ and $(b)$ we have
\begin{equation}
\omega^{bc}\nabla_a\omega_{bc}=0,
\end{equation}
where $\omega=d\theta$ for $\theta=\tau^{-1}k$ the contact form corresponding to $\tau$, and thus in particular
\begin{equation*}
\omega^{\gamma\epsilon}\nu_{\alpha\beta\epsilon}=\omega^{\gamma\epsilon}\nu_{\alpha\beta\epsilon}^{\circ}=-\tfrac{1}{2}T_{\alpha\beta}{}^\gamma,
\end{equation*}
where $\nu_{\alpha\beta\gamma}$ is defined as in \eqref{nu_H} with respect to $\nabla$ and $\theta$.
\end{lemma}
\begin{proof}
By Proposition \ref{eta_Proposition} we can find a connection $\nabla\in\mbq$ satisfying $(b)$ and the remaining freedom for the choice of such a connection in $\mbq$
is given by \eqref{freedom_eta=0}. Since $(b)$ is a projectively invariant equation, we can use the projective freedom in \eqref{freedom_eta=0} to change $\nabla$ to a
connection that satisfies $(a)$ and $(b)$. Moreover, any other connection $\widehat\nabla$ satisfying $(a)$ and $(b)$ then differs from $\nabla$ by a difference tensor
$\Lambda^c{}_{ba}$ of the form \eqref{freedom_eta=0} that in addition satisfies $\Lambda^a{}_{ba}=0$, which is equivalent to \eqref{remaining_freedom_(a-b)} as claimed.

To prove the last claim, note that $\nabla \tau=0$ is a equivalent to $$\nabla \epsilon^{(\theta)}=\nabla(\theta\wedge \mathcal \omega^m)=\nabla (\tau^{-(m+1)}\epsilon)= \epsilon\nabla \tau^{-(m+1)}=0,$$
where $\dim(M)=2m+1$ and $\epsilon\in \Gamma(\Wedge^{2m+1}T^*M(2m+2))$ is the canonical weighted volume form defined as in \eqref{orientation_induced_by_contact_structure}. For a connection $\nabla\in\mbq$ satisfying $(a)$ and $(b)$, we therefore have
\begin{equation}\label{trace_nabla_omega}
\nabla_a\omega_{[b_1 b_2} \cdots \omega_{b_{2 m - 1} b_{2 m}}\theta_{b_{2 m + 1}]}=0
\end{equation}
and contracting \eqref{trace_nabla_omega} with $\omega^{[b_1 b_2} \cdots \omega^{b_{2 m - 1} b_{2 m}}t^{b_{2 m + 1}]}$ implies $\omega^{bc}\nabla_a\omega_{bc}=0$.
\end{proof}

Suppose now $\nabla\in\mbq$ is a connection satisfying $(a)$ and $(b)$ of Lemma \ref{First_Step} for a choice $\tau\in\mathcal E_+(2)$ of scale. Then $2\nabla_ak_b=\tau\omega_{ab}=\mathcal L_{ab}$ and hence
$\nabla_{[a}\mathcal L_{bc]}=\tau\nabla_{[a}\omega_{bc]}=0$ implies
\begin{equation*}
-R_{ab}{}^d{}_ck_d=\nabla_a\nabla_b k_c -\nabla_b\nabla_ak_c=\tfrac{1}{2}(\nabla_a\mathcal L_{bc}+\nabla_b\mathcal L_{ca})=-\tfrac{1}{2}\nabla_c\mathcal L_{ab}.
\end{equation*}
 It follows that
\begin{equation}\label{crucial_equation}
\nabla_c\mathcal L_{ab}=2R_{ab}{}^d{}_ck_d=2W_{ab}{}^d{}_ck_d+4k_{[a}\Rho_{b]c},
\end{equation}
where $\Rho_{bc}=\Rho_{(bc)}=\frac{1}{2m}\Ric_{(ab)}=\frac{1}{2m}\Ric_{ab}$. Moreover, Lemma \ref{First_Step} and \eqref{crucial_equation} imply that
\begin{equation}\label{properties_Weyl_for_ab_connections}
0=W_{ab}{}^d{}_ck_d\mathcal L^{ab}=2W_{c[b}{}^d{}_{a]}k_d\mathcal L^{ab}=2W_{cb}{}^d{}_{a}k_d\mathcal L^{ab}
	\quad\textrm{and}\quad
W_{\alpha\beta}{}^d{}_{\gamma}k_d=-\tfrac{1}{4}T_{\alpha\beta}{}^{\epsilon}\mathcal L_{\epsilon\gamma}.
\end{equation}

Hence, for any connection $\nabla\in\mbq$ satisfying $(a)$ and $(b)$
of Lemma \ref{First_Step} its Weyl curvature has the property that the
restriction of $W_{ab}{}^d{}_{c}k_d$ to $\Wedge^2 H^*\otimes H^*(2)$
is (up to a constant multiple) essentially the torsion of the contact
projective structure.  Our aim is now, for any scale $\tau\in\mathcal
E_+(2)$, to single out a unique connection $\nabla^{\tau}\in\mbq$
satisfying $(a)$ and $(b)$ by spending the remaining freedom
\eqref{remaining_freedom_(a-b)} to fix $W_{ab}{}^d{}_{c}k_d$ also in
the direction transversal to $H$. We want the projective class
$[\nabla^{\tau}]$ to be independent of the choice of scale $\tau$. Note that
$W_{ab}{}^d{}_{c}k_d$ is projectively invariant. We are seeking
 an extension of
$-4W_{\alpha\beta}{}^d{}_{\gamma}k_d=T_{\alpha\beta}{}^{\epsilon}\mathcal
L_{\epsilon\gamma}$ to a tensor $E_{abc}\in\Gamma(\Wedge^2 T^*M\otimes
T^*M(2))$ which only depends on $[\nabla^{\tau}]$.

For that purpose let us now set
\begin{equation}\label{T^abc}
T^{abc}:=\Pi_{\alpha}^a\Pi_{\beta}^b\Pi_{\gamma}^c\mathcal L^{\alpha\delta}\mathcal L^{\beta\epsilon} T_{\delta\epsilon}{}^\gamma=\mathcal L^{a\delta}\mathcal L^{b\epsilon}T_{\delta\epsilon}{}^c.
\end{equation}
By construction $T^{abc}\in\Gamma(\Wedge^2 TM\otimes TM(-4))$ has the properties
\begin{equation}\label{symmetries_T}
T^{[ab]c}=T^{abc} , \quad T^{[abc]}=0 , \quad T^{abc}\mathcal L_{ab}=0 , \quad T^{abc}k_a=0 , \quad T_{abc}T^{abc}=0,
\end{equation}
where $T_{abc}=T^{efg}\mathcal L_{ea}\mathcal L_{fb}\mathcal L_{gc}$. Consequently we also have
\begin{equation*}
2T^{a[bc]}=T^{cba} , \quad T^{abc}k_b=0 , \quad T^{abc}k_c=0 , \quad T^{abc}\mathcal L_{ac}=0 , \quad T^{abc}\mathcal L_{bc}=0.
\end{equation*}
\begin{lemma}\label{Prop_extension}
Suppose $(M,H,\mbq)$ is a contact projective manifold of dimension $2m+1$ with torsion $T_{\alpha\beta}{}^\gamma\in\Gamma(\Wedge^2 H^*\otimes H)$.
For a choice of positive scale $\tau\in\mathcal E_+(2)$ and connection $\nabla\in\mbq$ set
\begin{align}\label{extension}
E_{abc}:=T^{efg}\mathcal L_{ea}\mathcal L_{fb}\mathcal L_{gc} -\tfrac{2}{2m+1}V^{ef}\mathcal L_{ea}\mathcal L_{fb}k_c+\tfrac{2}{2m+1}V^{ef} \mathcal L_{fc}\mathcal L_{e[b} k_{a]}\\
-\tfrac{4}{2m-1} U^{ef}\mathcal L_{fc}\mathcal L_{e[b} k_{a]}+\tfrac{8}{(2m-1)(2m+1)} W^{d}\mathcal L_{e[b}k_{a]}k_c, \nonumber
\end{align}
where $\mathcal L_{ab}=\tau\omega_{ab}=\tau (d\theta)_{ab} \in\Gamma(\Wedge^2 T^*M\otimes TM(2))$ for $\theta=\tau^{-1}k$ and
\begin{alignat}{2}
	U^{ab} &= \nabla_{c}T^{c(ab)} && \in\Gamma(S^2 TM(-4))\nonumber \\
	V^{ab} &= \nabla_{c}T^{abc}   && \in\Gamma(\Wedge^2 TM(-4))\nonumber \\
	W^a    &=\nabla_{b}\nabla_c T^{bac}+(2m+1)\Rho_{bc}T^{bac} && \in\Gamma(TM(-4)) \nonumber.
\end{alignat}
Then $E_{abc}\in\Gamma(\otimes^3 T^*M(2))$ satisfies the identities
\begin{align}\label{symmetries_E}
E_{\alpha\beta\gamma}=T_{\alpha\beta}{}^\delta\mathcal L_{\delta\gamma} , \qquad E_{[ab]c}=E_{abc} , \qquad E_{[abc]}=0 , \qquad E_{abc}\mathcal L^{ab}=0 .
\end{align}
Moreover, suppose $\tilde\tau=e^{-2u}\tau\in\mathcal E_+(2)$ is another positive scale and $\widetilde\nabla\in[\nabla]$ projective equivalent to $\nabla$ with $\Upsilon_a=\nabla_a u$. Then we have $$\widetilde E_{abc}=E_{abc},$$
where $\widetilde E_{abc}$ is constructed from $\tilde\tau$ and $\widetilde\nabla$ according to \eqref{extension}.
\end{lemma}
\begin{proof}
Fix a scale  $\tau\in\mathcal E_+(2)$ and a connection $\nabla\in\mbq$. Then Corollary \ref{H_comp_proj_structure_2} and the identities \eqref{symmetries_T} imply
\begin{equation*}
V^{ab}k_a=k_a\nabla_c T^{abc}=-T^{abc}\nabla_ck_a=-\tfrac{1}{2}T^{abc}\mathcal L_{ca}=0
\end{equation*}
and similarly $U^{ab}k_a=0.$ Moreover, \eqref{symmetries_T} implies
\begin{equation*}
\mathcal L_{ab}V^{ab}=-T^{abc}\nabla_c\mathcal L_{ab}=-T^{\beta\gamma\alpha}\nabla_\gamma\mathcal L_{\beta\alpha}.
\end{equation*}
Since the contraction of $T^{\beta\gamma\alpha}$ with $\nabla_\gamma\mathcal L_{\beta\alpha}$ only depends on the totally trace-free
part of $\nabla_\gamma\mathcal L_{\beta\alpha}$ (with respect to $\mathcal L^{\alpha\beta}$), we therefore obtain
\begin{equation*}\mathcal L_{ab}V^{ab}=-T^{\beta\gamma\alpha}\nabla_\gamma\mathcal L_{\beta\alpha}=-T^{\beta\gamma\alpha}T_{\beta\alpha\gamma}=0\end{equation*}
and consequently also
\begin{equation*}
W^ak_a=k_a\nabla_{b}\nabla_c T^{bac}=k_a\nabla_b V^{ba}=-\tfrac{1}{2}V^{ba}\mathcal L_{ba}=0.
\end{equation*}
In summary,
\begin{equation}\label{symmetries_UVW}
	U^{ab}k_a =  U^{ba}k_a = 0 , \quad
	V^{ab}k_b = -V^{ba}k_b = 0 , \quad
	\mathcal L_{ab} V^{ab} = 0 , \quad
	W^a k_a = 0 .
\end{equation}
The first three properties in \eqref{symmetries_E} are evident from its definition and the fourth property follows from $T^{abc}\mathcal L_{ab}=0$ and $\mathcal L_{ab}V^{ab}=0$.
Hence, it only remains to prove the second claim.

Therefore, let $\tilde\tau=e^{-2u}\tau\in\mathcal E_+(2)$ be another scale and $\widetilde \nabla$ the connection in the projective class of $\nabla$ that is related to $\nabla$ via $\Upsilon=du$.
To verify that $\widetilde E_{abc}=E_{abc}$ we only need to show that they coincide infinitesimally: Indeed, since the projective class $[\nabla]$ forms an affine
space of over the Fr\'echet space $\Gamma(T^*M)$ of $1$-forms, for any $\widetilde\nabla\in[\nabla]$ and $\Upsilon\in\Gamma(T^*M)$ the curve $\widetilde{\nabla}+t\Upsilon$ lies in $[\nabla]$  (where $\widetilde\nabla+t\Upsilon$ is an abbreviation for the connection in $[\nabla]$ resulting from projectively changing $\widetilde{\nabla}$ by $t\Upsilon$) and, by the Fundamental Theorem of Calculus, any function $F$ on $[\nabla]$ is constant if and only if $\frac{d}{dt}\vert_{t=0} F(\widetilde\nabla+t\Upsilon)=0$ for all $\widetilde\nabla\in[\nabla]$ and $\Upsilon\in\Gamma(T^*M)$ (see also Section 4 of \cite{CDS}).
Hence, we only need to show that  $\widetilde E_{abc}$ and $E_{abc}$ coincide modulo terms that are quadratic in $\Upsilon$ and so we only need to compute the parts of the changes of the relevant
tensors in the expression of $\widetilde E_{abc}$ linear in $\Upsilon$. Straightforward computations show that $U, V, W$ transform according to the formulas
\begin{align*}
&\widetilde U^{ab}=U^{ab}+(2m-1) T^{c(ab)}\Upsilon_c\\
&\widetilde V^{ab}=V^{ab}+(2m+1) T^{abc}\Upsilon_c\\
&\widetilde W^a \equiv W^a+(2m+1)\Upsilon_b U^{ab}+\tfrac{3}{2} (2m-1) \Upsilon_b V^{ba} \pmod {\Upsilon^2} .
\end{align*}
(Here, the notation $A \equiv B \pmod {\Upsilon^2}$ just indicates that the difference $A - B$ consists only of terms at least quadratic in $\Upsilon$.)

Hence we deduce from \eqref{symmetries_T} and \eqref{symmetries_UVW} that the terms in the sum of $\widetilde E_{abc}$ transform as follows:
\begin{align*}
	\widetilde T^{efg}\widetilde{\mathcal L}_{ea}\widetilde{\mathcal L}_{fb}\widetilde{\mathcal L}_{gc}
		&\equiv T^{efg}\mathcal L_{ea}\mathcal L_{fb}\mathcal L_{gc} + 4\Upsilon_e T^{e[fg]}\mathcal L_{gc}\mathcal L_{f[b}k _{a]} \\
		&\phantom{=}\qquad+
4\Upsilon_e T^{e(fg)}\mathcal L_{gc}\mathcal L_{f[b}k _{a]}
			 + 2\Upsilon_g T^{efg} \mathcal L_{ea}\mathcal L_{fb} k_c \pmod {\Upsilon^2} \\
	\widetilde {U}^{ef}\widetilde{\mathcal L}_{fc}\widetilde{\mathcal L}_{e[b}k_{a]}
		&\equiv U^{ef}\mathcal L_{fc}\mathcal L_{e[b}k_{a]}
+2\Upsilon_f U^{ef}\mathcal L_{e[b}k_{a]}k_c \\
			&\phantom{=}\qquad +(2m-1)\Upsilon_d T^{d(ef)}\mathcal L_{fc}\mathcal L_{e[b}k_{a]} \pmod {\Upsilon^2} \\	
	\widetilde{V}^{ef}\widetilde{\mathcal L}_{ea}\widetilde{\mathcal L}_{fb}k_c
		&\equiv V^{ef}\mathcal L_{ea}\mathcal L_{fb}k_c+4\Upsilon_eV^{ef}\mathcal L_{f[b}k_{a]}k_c \\
			&\phantom{=}\qquad + (2m+1)\Upsilon_dT^{efd}\mathcal L_{ea}\mathcal L_{fb} k_c \pmod {\Upsilon^2} \\
	\widetilde{V}^{ef}\widetilde{\mathcal L}_{fc}\widetilde{\mathcal L}_{e[b}k_{a]}
		&\equiv V^{ef}\mathcal L_{fc}\mathcal L_{e[b}k_{a]} +2\Upsilon_f V^{ef}\mathcal L_{e[b} k_{a]} k_c \\
			&\phantom{=}\qquad - 2(2m+1) \Upsilon_dT^{d[ef]} \mathcal L_{fc}\mathcal L_{e[b}k_{a]} \pmod {\Upsilon^2} \\
	\widetilde{W}^{e}\widetilde{\mathcal{L}}_{e[b}k_{a]}k_c
		&\equiv W^{e}\mathcal{L}_{e[b}k_{a]}k_c
+(2m+1)\Upsilon_fU^{ef}\mathcal{L}_{e[b}k_{a]}k_c \\
			&\phantom{=}\qquad +\tfrac{3}{2}(2m-1)\Upsilon_fV^{fe}\mathcal{L}_{e[b}k_{a]}k_c \pmod {\Upsilon^2} .
\end{align*}
This shows that $\widetilde E_{abc}$ coincides with $E_{abc}$ modulo terms at least quadratic in $\Upsilon$ and thus they must coincide exactly.
\end{proof}

\begin{remark}
Note that if $\nabla$ in Lemma \ref{Prop_extension} preserves $\tau$, then $\widetilde\nabla$ is precisely the unique connection in the projective class of $\nabla$ that preserves $\tilde\tau$.
\end{remark}

\begin{theorem}\label{Fefferman_type_construction}
Suppose $(M,H,\mbq)$ is a contact projective manifold with torsion $T_{\alpha\beta}{}^\gamma\in\Gamma(\Wedge^2 H^*\otimes H)$.
For any choice of positive scale $\tau\in\mathcal E_+(2)$ there exists a unique connection
$\nabla^{\tau}\in\mbq$ such that
\begin{enumerate}
\item $\nabla_{a}\tau=0$
\item $\nabla_{(a}k_{b)}=0$
\item $W_{ab}{}^d{}_ck_d=-\frac{1}{4}E_{abc}$,
\end{enumerate}
where $E_{abc}\in\Gamma(\Wedge^2 T^*M\otimes TM(2))$ is the tensor defined as in \eqref{extension} with respect to $\tau$ and $\nabla^{\tau}$. Moreover, the projective class
$\mbp=[\nabla^{\tau}]$ does not depend on the choice of scale $\tau\in\mathcal E_+(2)$ and its Weyl curvature by construction satisfies
\begin{equation*}W_{\alpha\beta}{}^d{}_{\gamma}k_d=-\tfrac{1}{4}T_{\alpha\beta}{}^{\epsilon}\mathcal L_{\epsilon\gamma}.\end{equation*}
 \end{theorem}
 \begin{proof} Fix a positive scale $\tau\in\mathcal E_+(2)$. By Lemma \ref{First_Step} we can find a connection $\nabla\in\mbq$ that satisfies
 $(a)$ and $(b)$, and we have seen that the Weyl curvature of such a connection $\nabla$ then satisfies \eqref{properties_Weyl_for_ab_connections}, which shows that $W_{ab}{}^d{}_ck_d$ and
   $E_{abc}$ have the same symmetries. By Lemma \ref{First_Step} we also know that the difference tensor  $\Lambda^c{}_{ba}$ between $\nabla$ and any other connection $\widehat\nabla\in\mbq$ with properties (a)-(b) is of the form
 \begin{equation}\label{difference_tensor_for_(a,b)}
\Lambda^c{}_{ba} = 2 k_{(a} \phi^c{}_{b)},
\end{equation}
where $\phi^c{}_b\in\Gamma(\textrm{End}_0(TM)(-2))$ with $\phi^c{}_{b}k_c=0$. We now show that the remaining freedom \eqref{difference_tensor_for_(a,b)} is sufficient to arrange $(c)$.
Since
\begin{equation*}
\widehat R_{ab}{}^c{}_d-R_{ab}{}^c{}_d=\nabla_a\Lambda^c{}_{db}-\nabla_b\Lambda^c{}_{da}+\Lambda^c{}_{ea}\Lambda^e{}_{bd}-\Lambda^c{}_{eb}\Lambda^e{}_{da},
\end{equation*}
one computes straightforwardly, using $\nabla_ak_b=\frac{1}{2}\mathcal L_{ab}$ and $\phi^c{}_{b}k_c=0$, that
\begin{equation}\label{change_curv_contracted_with_theta}
\widehat R_{ab}{}^c{}_{d}k_c-R_{ab}{}^c{}_{d}k_c=\tfrac{1}{2}(k_d\phi^c_a\mathcal L_{bc}+k_a\phi^c{}_d\mathcal L_{bc}-k_d\phi^b_c\mathcal L_{ac}-k_b\phi^c_d\mathcal L_{ac})
\end{equation}
and that
\begin{equation}\label{change_Ricci}
\widehat \Ric_{bd}-\Ric_{bd}=\tfrac{1}{2}\phi^a{}_d\mathcal L_{ab}+\tfrac{1}{2}\phi^a{}_b\mathcal L_{ad}+k_d\nabla_a\phi^a{}_b+k_b\nabla_a\phi^a{}_d-k_{b}k_d\phi^a{}_e \phi^e{}_a .
\end{equation}
Thus we conclude from \eqref{change_curv_contracted_with_theta} and \eqref{change_Ricci} that (as usual, taking $\dim M = 2 m + 1$)
\begin{multline}\label{change_Weyl_contracted_with_k}
\begin{aligned}
\widehat W_{ab}{}^c{}_dk_c
	&=W_{ab}{}^c{}_dk_c+\widehat R_{ab}{}^c{}_{d}k_c-R_{ab}{}^c{}_{d}k_c-2k_{[a}\widehat\Rho_{b]d}+2k_{[a}\Rho_{b]d}\\
	&=W_{ab}{}^c{}_dk_c+\tfrac{1}{2}(k_d\phi^c{}_a \mathcal L_{bc}+k_a\phi^c{}_d \mathcal L_{bc}-k_d\phi^c{}_b \mathcal L_{ac}-k_b\phi^c{}_d \mathcal L_{ac})
\end{aligned}
\\ -\tfrac{1}{2m}[\mathcal L_{ed}\phi^e{}_{[b} k_{a]}+\phi^e{}_d \mathcal L_{e[b}k_{a]}+2k_d(\nabla_e\phi^e{}_{[b})k_{a]}] .
\end{multline}
In particular, we have
\begin{equation}\label{c1}
\widehat W_{\alpha\beta}{}^c{}_dk_c r^d=W_{\alpha\beta}{}^c{}_dk_c r^d+\phi^c{}_{[\alpha} \mathcal L_{\beta]c},
\end{equation}
which is equivalent to
$\widehat W_{a[\beta}{}^c{}_{\delta]}k_c r^a=W_{a[\beta}{}^c{}_{\delta]}k_c r^a-\frac{1}{2} \phi^c{}_{[\beta} \mathcal L_{\delta]c}$, and also
\begin{equation}\label{c2}
\widehat W_{a(\beta}{}^c{}_{\delta)}k_c r^a=W_{a(\beta}{}^c{}_{\delta)}k_c r^a+\tfrac{m+1}{2m}\phi^c{}_{(\beta} \mathcal L_{\delta)c}.
\end{equation}
 Note that \eqref{symmetries_T} and \eqref{difference_tensor_for_(a,b)}  imply that
\begin{align*}
\widehat U^{ab} &=\widehat \nabla_cT^{c(ab)}= \nabla_cT^{c(ab)}=U^{ab} \\
\widehat V^{ab} &=\widehat \nabla_cT^{abc}= \nabla_cT^{abc}=V^{ab} .
 \end{align*}
 Hence, by \eqref{properties_Weyl_for_ab_connections} and symmetries \eqref{symmetries_UVW}, we can find $\phi^c{}_b \in\Gamma(\textrm{End}_0(TM)(-2))$ such that
\begin{align} \label{arrangement1}
 \widehat W_{\alpha\beta}{}^c{}_dk_c r^d=\tfrac{1}{2(2m+1)}\mathcal L_{e\alpha}\mathcal L_{f\beta}\widehat V^{ef}\\
\widehat  W_{a(\beta}{}^c{}_{\delta)}k_c r^a=\tfrac{1}{2(2m-1)}\mathcal L_{e\beta}\mathcal L_{f\gamma} \widehat U^{ef}\nonumber.
 \end{align}
 Assume now $\nabla\in\mbq$ is a connection satisfying (a), (b), and \eqref{arrangement1}. Then \eqref{change_Weyl_contracted_with_k} shows that the difference tensor \eqref{difference_tensor_for_(a,b)}
between any other connection $\widehat\nabla\in\mbq$, with these properties, and $\nabla$
is of the form
\begin{equation}\label{last_freedom}
\Lambda^c{}_{ba}=2k_ak_b\xi^c,\end{equation}
for some weighted vector field $\xi^c\in\Gamma(TM(-4))$ such that $k_c\xi^c = 0$. From \eqref{last_freedom} and \eqref{change_Weyl_contracted_with_k} it follows that
\begin{equation*}
\widehat W_{\alpha b}{}^c{}_dk_cr^dr^b=W_{\alpha b}{}^c{}_dk_cr^dr^b+\tfrac{2m+1}{2m}\mathcal L_{c\alpha}\xi^c\qquad \textrm{ and }\qquad \widehat W^a=W^a.
\end{equation*}
Hence, we can find $\xi^c\in\Gamma(TM(-4))$ such that
\begin{equation*}
\widehat W_{\alpha b}{}^c{}_d\theta_cr^dr^b=\tfrac{1}{(2m-1)(2m+1)}W^c\mathcal L_{c\alpha}
\end{equation*}
and therefore a unique connection $\nabla^{\tau}\in\mbq$ which is characterised by properties (a)-(c). Uniqueness of this connection also immediately implies,
using Lemma \ref{Prop_extension} and the projective invariance of equation $(b)$, that $\nabla^{\tilde\tau}$ for another scale $\tilde\tau=e^{-2u}\tau$ is
projectively equivalent to $\nabla^{\tau}$ via $\Upsilon=du$.
\end{proof}

As a corollary of Theorem \ref{Fefferman_type_construction} we obtain the converse of Theorem \ref{projective_structure_with_parallel_symplectic_tractor}:

\begin{corollary}
\label{corollary:converse-of-symplectic-reduction}
Suppose $(M,H,\mbq)$ is a contact projective manifold with vanishing
torsion and denote by $\mbp$ the distinguished compatible projective
structure of Theorem \ref{Fefferman_type_construction}. Then the
conformal contact form $k_a\in\Gamma(T^*M(2))$ of $H$ is a normal
solution of the Killing-type equation \eqref{Killing_eq} of the
projective manifold $(M,\mbp)$.
\end{corollary}

\subsubsection{Remark on parallel symplectic bitractor fields}
We could alternatively have formulated the content of this section dually, in terms of a projective manifold $(M, \mbp)$ of dimension $2 m + 1$ equipped with a parallel symplectic bitractor field, that is, a parallel section $\Omega^{AB} \in \Gamma(\Wedge^2 \mcT)$ nondegenerate in the sense that $\Omega^{\wedge (m + 1)} \in \Gamma(\Wedge^{2 m + 2} \mcT)$ is nonzero. Underlying such a section is a correspondingly nondegenerate normal solution
\[
	\Omega^{AB} Z_A{}^a Z_B{}^b \in \Gamma(\Wedge^2 TM (-2))
\]
of the BGG operator $\mathcal D^{\Wedge^2 \mcT} : \Gamma(\Wedge^2 TM(-2)) \to \Gamma((\Wedge^2 TM \otimes T^*M)_{\circ}(-2))$,
\[
	\mathcal D^{\Wedge^2 \mcT} : \sigma^{ab} \mapsto \nabla_c \sigma^{ab} + \tfrac{1}{m} \nabla_d \sigma^{d[a} \delta^{b]}{}_c \textrm{.}
\]
Here, $(\Wedge^2 TM \otimes T^*M)_{\circ}$ is the kernel of the tensorial contraction map $\Wedge^2 TM \otimes T^*M \to TM$, $\psi^{ab}{}_c \mapsto \psi^{ab}{}_b$.

At the tractor level we can translate readily between these pictures: Given a projective $(2 m + 1)$-manifold equipped with a parallel symplectic tractor $2$-form $\Omega_{AB} \in \Gamma(\Wedge^2 \mcT^*)$ the section
\[
	\Omega^{AB} := \vol^{AB C_1 \cdots C_{2m}} \Omega_{C_1 C_2} \cdots \Omega_{C_{2m - 1} C_{2m}} \in \Gamma(\Wedge^2 \mcT)
\]
is a parallel symplectic bitractor field; the reverse construction is similar. Here, $\vol^{A_1 \cdots A_{2 m + 2}}$ is the nonzero parallel $(2m + 2)$-vector dual to the tractor volume form $\vol_{A_1 \cdots A_{2 m + 2}}$, uniquely characterised by the contraction identity
$$\vol_{A_1 \cdots A_{2 m + 2}} \vol^{A_1 \cdots A_{2 m + 2}} = (2 m + 2)!~.
$$

\subsection{Reduction to \texorpdfstring{$\SL(m + 1, \bbC) \times \U(1)$}{SL(m+1, C) x U(1)}: Parallel tractor complex structure}\label{J-red-sect}

In this subsection we treat the holonomy reduction of an oriented projective manifold $(M, \mbp)$ determined by a parallel tractor complex structure
\[
	\bbJ^A{}_B \in \Gamma(\End \mcT) \textrm{,}
\]
that is, a parallel tractor endomorphism $\bbJ^A{}_B$ satisfying
\[
    \bbJ^A{}_C\bbJ^{C}{}_B = -\delta^A{}_B ;
\]
this condition implies
$\bbJ^A{}_B\in\Gamma(\mathcal A)$ and forces $M$ to have odd
dimension, say, $2 m + 1$. The case $m = 1$ ($\dim M = 3$) is
qualitatively different from the general case and we do not develop it
fully, though some of our results still apply in that setting. As remarked after Theorem B in the introduction, the existence of a \textit{unitary} tractor holonomy reduction in that dimension forces flatness of the projective structure, so that
case is anyway less interesting in our context.  This subsection follows the work of Armstrong \cite[\S3.3]{ArmstrongThesis},
\cite[\S4]{ArmstrongP1}, \cite[\S3.1]{ArmstrongP2}. For consistency with our presentation in Subsection \ref{Sp-reduction}, for self-containment, and for the sake of planned work on general (that is, not necessarily Einstein) Sasaki structures, we record the constructions in detail here. We work at the explicit level of the involved
manifolds, rather than at the level of the tractor bundle.

\begin{proposition}\label{proposition:complex-single-curved-orbit}
Let $(M, \mbp)$ be a connected oriented projective manifold of dimension $2 m + 1 \geq 3$ equipped with a parallel tractor complex structure $\bbJ^A{}_B \in \Gamma(\mcA)$. Then the following holds:
\begin{enumerate}
\item The curved orbit decomposition of $M$ determined by the holonomy reduction of $\nabla^{\mcT}$ to $\SL(2 m + 2, \bbR) \cap \GL(m + 1, \bbC) \cong\SL(m + 1, \bbC) \times \U(1)$ corresponding to $\bbJ$ is trivial.
In particular, the underlying vector field $k^a:= \Pi^{\mcA}(\bbJ)^a = Z_A{}^a \bbJ^A{}_B X^B$ vanishes nowhere.
\item Locally around any point, $(M, \mbp, \mathbb J)$ admits a parallel tractor complex volume form. If $M$ is simply connected, it admits a (globally defined) parallel tractor complex volume form  $\vol_{\bbC}\in\Gamma(\Wedge^{m+1}_{\bbC}\mcT^*)$, equivalently, a holonomy reduction of $\nabla^{\mcT}$ to the proper subgroup $\SL(m + 1, \bbC) \subset \SL(2 m + 2, \bbR) \cap \GL(m + 1, \bbC).$
\end{enumerate}
\end{proposition}
\begin{proof}
\leavevmode
\begin{enumerate}
\item
By the discussion in Section \ref{subsection:holonomy-reductions-projective} (see \cite[Theorem 2.6]{CGH} for details) this is equivalent to showing that $\SL(m + 1, \bbC) \times \U(1)$ acts transitively on the homogeneous model $\mathbb P_+(\bbR^{2m+2})=\textrm{SL}(n+1,\mathbb R)/P$ of oriented projective structures. This follows from the fact that $\SL(m + 1, \bbC) \times \U(1)$ acts transitively on $\bbC^{m+1}\setminus \{0\}$ (in fact, so does its proper subgroup, $\SL(m + 1, \bbC)$). The second statement follows from the observation that the zero locus of a normal solution of a first BGG operator is always a union of curved orbits, as recalled in
Section \ref{subsection:holonomy-reductions-projective}.

Alternatively, $\bbJ^A {}_B X^B = k^a W^A{}_a + \alpha X^A$ for some $\alpha \in C^{\infty}(M)$, so at a point $x \in M$ where $k_x = 0$, $\alpha_x$ is a (real) eigenvalue of $\bbJ_x$. Since $\bbJ_x$ is a complex structure, its only eigenvalues are $\pm i$.

\item By Lemma  \ref{parallel_adjoint_tractor_and_curvature} the tractor curvature has values in $\mathfrak{sl}(\mathcal T, \mathbb J)$, which is equivalent for $(M,\mbp)$
to admitting locally around any point a parallel complex tractor volume form. \qedhere
\end{enumerate}
\end{proof}

\subsubsection{$k$-adapted scales}
\label{subsubsection:k-adapted-scales}
Given a parallel tractor complex structure $\mathbb J$ on a projective manifold $(M,\mbp)$, the underlying vector field $k=\Pi^{\mcA}(\mathbb J)$ determines a subclass of preferred connections in $\mbp$: We say that $\nabla \in \mbp$ is \textit{$k$-adapted} if $k$ is divergence-free with respect to $\nabla$, that is, if $\nabla_c k^c = 0$.
The following proposition shows that at least locally there are always $k$-adapted scales in $\mbp$, cf.\ \cite[Lemma 4.4]{ArmstrongP1}.
\begin{proposition}\label{proposition:k-adapted-connections}
Let $(M, \mbp)$ be an oriented projective manifold of dimension $2m+1 \geq 3$ equipped with a parallel tractor complex structure $\bbJ^A{}_B \in \Gamma(\mcA)$.
\begin{enumerate}
	\item About any point $x \in M$ there is a neighbourhood $U$ and a $k$-adapted scale $\nabla \in \mbp\vert_U$.
	\item Given a $k$-adapted scale $\nabla \in \mbp$, the class of $k$-adapted scales in $\mbp$ consists exactly of the connections
		\begin{equation}\label{equation:k-adapted-class}
			\nabla'_b \xi^a := \nabla_b \xi^a + \xi^a \Upsilon_b + \Upsilon_c \xi^c \delta^a{}_b
		\end{equation}
		where $\Upsilon \in \Gamma(T^*M)$ is exact and satisfies $\Upsilon_c k^c = 0$. In particular, any such $\Upsilon$ is $k$-invariant.
\end{enumerate}
\end{proposition}
\begin{proof}
\leavevmode
\begin{enumerate}
	\item Fix a scale $\nabla \in \mbp$. Taking the trace of both sides of \eqref{projective_change} gives that the divergence $\nabla_c k^c$ transforms under a projective change as
	\begin{equation}\label{div_change}
			\nabla'_c k^c = \nabla_c k^c + (2m + 2) \Upsilon_c k^c \textrm{.}
	\end{equation}
Since $k$ is nowhere vanishing, we can find locally around any point coordinates $(x_1,..., x_{2m+1})$ such that $k=\frac{\partial}{\partial x_1}$.
Changing $\nabla$ projectively via $\Upsilon:=-\frac{1}{2m+2} df$, where $f:=\int^{x_1} \nabla_c k^c\, dt$, identity \eqref{div_change} implies that at least locally there always exist $k$-adapted scales as claimed.
	\item The first part also evidently follows from \eqref{div_change}. The second then follows from the identity $\mcL_k \Upsilon = d\Upsilon(k, \,\cdot\,) + d(\Upsilon(k))$. \qedhere
\end{enumerate}
\end{proof}

\begin{remark}
The proof of Proposition \ref{proposition:k-adapted-connections} uses that $k$ vanishes nowhere but no other properties of $k$.
\end{remark}

Since $\nabla^{\mcT} \bbJ = 0$ we have $\bbJ = L^{\mcA}(k)$ and so specializing \eqref{equation:splitting-operator-adjoint} to a $k$-adapted scale gives that (in such a scale)
\begin{equation}\label{equation:J-k-adapted}
	\bbJ^A {}_B
		=
	\begin{pmatrix}
		\nabla_b k^a & k^a \\
		-\Rho_{bc} k^c & 0
	\end{pmatrix} \textrm{.}
\end{equation}

Substituting \eqref{equation:J-k-adapted} in the complex structure equation $\bbJ^2 = -\id_{\mcT}$ yields (still in a $k$-adapted scale) the equivalent system
\begin{equation}\label{equation:complex-structure-components-k-adapted}
	\begin{array}{rcl}
		k^c \nabla_c k^a
			&\!\!\!=\!\!\!& 0 \\
		\nabla_c k^a \nabla_b k^c - k^a \Rho_{bc} k^c
			&\!\!\!=\!\!\!& -\delta^a{}_b \\
		\Rho_{cd} k^c k^d
			&\!\!\!=\!\!\!& 1 \\
		\Rho_{cd} k^d \nabla_b k^c
			&\!\!\!=\!\!\!& 0
	\end{array}
	 .
\end{equation}
The first equation simply says that the integral curves of $k$ are parameterised geodesics (again, for an arbitrary $k$-adapted scale $\nabla$). The third equation implies that the component $\nu_b = -\Rho_{bc} k^c$ of $\bbJ^A{}_B$ vanishes nowhere and that the tangent bundle splits as
\begin{equation}\label{equation:decomposition-TM}
	TM = \langle k \rangle \oplus H, \qquad \textrm{where} \qquad H := \ker \nu ;
\end{equation}
this decomposition evidently depends on the choice of $\nabla \in \mbp$. Moreover, $\nabla$ canonically defines a connection $\smash{\nabla^H}$ on $H$ by $\smash{\nabla^H_{\eta}} \xi := \pi_H (\nabla_{\eta} \xi)$ (for any $\eta \in \Gamma(TM)$ and $\xi \in \Gamma(H)$), where $\pi_H$ is the bundle projection $TM \to H$ determined by the decomposition \eqref{equation:decomposition-TM}.
The fourth equation in \eqref{equation:complex-structure-components-k-adapted} says that we may regard $\nabla k$ as a bundle map $TM \to H$, and hence the second equation implies that
\[
	J_H := \nabla k\vert_H\in \Gamma(\End H)
\]
is a complex structure on $H$ (cf. \cite[Lemma 4.2]{ArmstrongP1}).

\begin{remark}\label{complex_str_on_TM/k}
It is straightforward to check that the parallel complex structure $\mathbb J$ on $\mcT$ induces in fact a complex structure on $TM/\langle k \rangle$ independent of any choice of scale.
\end{remark}

Likewise, substituting the formula \eqref{equation:J-k-adapted} for $\bbJ$ into the parallelism condition $\nabla \bbJ = 0$ and using the formula \eqref{adjoint_tractor_connection} for the adjoint tractor connection yields (among other differential identities)
\begin{equation}\label{equation:components-J-parallel}
	\nabla_c \nabla_b k^a = -\Rho_{cb} k^a + \Rho_{bd} k^d \delta^a {}_c .
\end{equation}

We are particularly interested in the case that $k$ is a projective symmetry of $\mbp$; in that case, as was shown in \cite{ArmstrongP1}---and as we reprove in a different way in \S\ref{subsubsection:c-projective-leaf-space}---for $m \geq 2$ $(\mbp, \bbJ)$ canonically determines a so-called \textit{c-projective structure} on the (local) leaf space of $k$. For now we remark that if $k$ preserves $\mbp$, it preserves each $k$-adapted scale therein:

\begin{proposition}\label{proposition:k-preserves-k-adapted-scale}
Let $(M, \mbp)$ be an oriented projective structure of dimension $2m+1 \geq 3$ equipped with a parallel tractor complex structure $\bbJ^{A}{}_B\in\Gamma(\mcA)$. If the underlying vector field $k :=\Pi^{\mcA}(\bbJ)$ is a projective symmetry
of $\mbp$, then it is an affine symmetry of every $k$-adapted scale $\nabla \in \mbp$.
\end{proposition}
\begin{proof}
Recall from Section \ref{proj_symm_section} that $k$ is a projective symmetry of $\mbp$ if and only if for some (hence every) connection $\nabla\in\mbp$ the trace-free part of
\begin{align*}
	(\mcL_k\nabla)_{ab}{}^c
		&= \nabla_a \nabla_b k^c{} + R_{da}{}^c{}_b k^d=\nabla_{(a} \nabla_{b)} k^c{} + R_{d(a}{}^c{}_{b)} k^d .
\end{align*}
vanishes. Assuming that $k$ is a projective symmetry and $\nabla$ is a $k$-adapted scale, it hence suffices to show that the trace $(\mcL_k\nabla)_{ab}{}^b$ vanishes. This follows immediately from $\nabla_a k^a=0$,
 the fact the curvature of a scale satisfies $R_{ab}{}^c{}_c=0$, \eqref{decp}, and  \eqref{equation:components-J-parallel}.
\end{proof}

\subsubsection{C-projective geometry: a brief primer}
\label{subsubsection:c-projective-primer}

C-projective geometry is a natural analogue of real projective geometry in the setting of complex manifolds, which subsumes but is distinct from projective geometry in the holomorphic category.
Here we briefly review the notion of a c-projective structure; see the comprehensive article \cite{CEMN} for much more.

Any complex manifold $\wtM$ with complex structure $\tilde J \in \Gamma(\End T\wtM)$ admits a preferred class of affine connections, called \textit{complex connections}, that are, affine connections $\wtnabla$ on $\wtM$ compatible with the complex structure in the sense that $\wtnabla \tilde J = 0$. Any such complex connection in turn determines a distinguished class of curves (which contains but is larger than the class of its geodesics): Given a complex manifold $(\wtM, \tilde J)$ of real dimension $2m\geq 4$, a curve $\gamma$ on $\wtM$ is called
a \textit{$\tilde J$-planar curve}\footnote{In \cite[\S4]{ArmstrongP1} such a curve is called a \textit{generalised complex geodesic}.} with respect to a complex connection $\wtnabla$ if for some (hence every) parameterisation its acceleration $\wtnabla_{\dot\gamma} \dot\gamma$ is in the span $\langle \dot\gamma, \tilde J \dot\gamma \rangle$.
Two complex connections $\wtnabla$ and $\wtnabla'$ on $(\wtM, \tilde J)$  are called \emph{c-projectively equivalent}, if they have the same $\tilde J$-planar curves. If $\wtnabla$ and $\wtnabla'$ have the same torsion, having the same $\tilde J$-planar curves is equivalent to the existence of a (real) $1$-form $\tilde\Upsilon \in \Gamma(T^* \wtM)$ such that\footnote{For later convenience the formula for this transformation rule differs from that in \cite{CEMN} in that $\tilde\Upsilon$ in that reference has been replaced with $2 \tilde\Upsilon$ here.}
\begin{equation}\label{equation:c-projective-transformation}
	\wtnabla'_{\tilde\eta} \tilde\xi
		= \wtnabla_{\tilde\eta} \tilde\xi
			+ \tilde\Upsilon(\tilde\eta) \tilde\xi + \tilde\Upsilon(\tilde\xi) \tilde\eta
			- \tilde\Upsilon(\tilde J\tilde\eta) \tilde J\tilde\xi - \tilde\Upsilon(\tilde J\tilde\xi) \tilde J\tilde\eta
\end{equation}
for all $\tilde\xi, \tilde\eta \in \Gamma(T\wtM)$ \cite{OtsukiTashiro}.

\begin{definition}\label{Definition: c-projective}
A \emph{c-projective structure} on a complex manifold $(\wtM, \tilde J)$ of real dimension $2m\geq 4$ is an equivalence class $\tilde\mbp=\langle \wtnabla\rangle$ of c-projectively equivalent torsion-free complex connections.
\end{definition}

Recall that on a complex manifold $(\wtM, \tilde J)$ the line bundle $\Wedge^{2m} T\wtM$ is canonically oriented and hence we can take its $(m+1)$st positive root. Following the notation of \cite{CEMN},
we denote this root (in the context of c-projective geometry) by $\mcE_\bbR (1,1)$ and its complexification by $\mcE(1,1)$. Moreover, for $k\in\mathbb Z_{\geq0}$ we set $\mcE_\bbR(k,k):=\mcE_\bbR (1,1)^{\otimes k}$
and $\mcE_\bbR (-k,-k):=\mcE_\bbR (k,k)^*$. Given a c-projective manifold $(\wtM, \tilde J, \tilde\mbp)$, note that for any $k\neq 0$, mapping a connection in $\tilde\mbp$ to its induced connection on $\mcE_\bbR(k,k)$ induces an affine bijection  between connections in $\tilde\mbp$ and linear connections on the line bundle $\mcE_\bbR(k,k)$. In particular, any positive section of $\mcE_\bbR(k,k)$, called a \emph{c-projective scale (of weight $(k,k)$)}, determines a connection on $\mcE_\bbR(k,k)$ by decreeing that this positive section is parallel and hence determines a connection in $\tilde\mbp$.

Analogously as projective structures, c-projective structures admit canonical Cartan connections. Specifically, c-projective structures are equivalent to normal, torsion-free (real) parabolic geometries of type $(\PSL(m + 1, \bbC), \wtP')$, where
$\wtP'$ is the stabiliser in $\PSL(m + 1, \bbC)$ of a complex line in $\bbC^{m+1}$.

In the following sections, we will in addition always assume that
c-projective manifolds $(\wtM, \tilde J, \tilde\mbp)$ admit $(m +
1)$st roots of the complex line bundle $\smash{\Wedge^m T^{1, 0}
  \wtM}$, where $T^{1, 0} \wtM$ is the $+i$-eigenspace ($\dim_\bbC
T^{1, 0} \wtM = m$) of the complex-linear extension of $\tilde J$ on
the complexified tangent bundle $T\wtM \otimes \bbC$, and that we have
chosen such a root, which we denote by $\mcE(1, 0)$. Such a choice,
which always exists locally (though it need not exist globally),
corresponds to extending the canonical Cartan geometry of the
c-projective structure to a (again, normal, torsion-free [real])
parabolic geometry of type $(\SL(m + 1, \bbC), \wtP)$, where $\wtP$ is
the stabiliser in $\SL(m + 1, \bbC)$ of a complex line in
$\bbC^{m+1}$. Note that the choice of $\mcE(1,0)$ induces an
identification $\mcE(1,1)=\mcE(1,0)\otimes\mcE(0,1)$, where
$\smash{\mcE(0,1):=\overline{\mcE(1,0)}}$ denotes the complex
conjugate line bundle of $\mcE(1,0)$.

\subsubsection{The c-projective leaf space}\label{subsubsection:c-projective-leaf-space}

Suppose $(M, \mbp)$ is an oriented projective manifold of dimension $2 m + 1 \geq 5$ equipped with a parallel tractor complex structure $\bbJ \in \Gamma(\mcA)$ whose underlying vector field $k = \Pi^{\mcA}(\bbJ)$ is a projective symmetry. By Proposition \ref{proposition:complex-single-curved-orbit}(a) $k$ vanishes nowhere, so it defines a foliation of rank $1$ on $M$, and we may consider a local leaf space thereof, that is, an open subset $U\subseteq M$,
a smooth manifold $\wtM$, and a surjective submersion $\pi: U\rightarrow \wtM$ such that $\ker T_x\pi =\langle k_x\rangle$ for all $x\in U$. Note that for any point $x\in M$ we can find a local leaf space of $k$ defined on a neighbourhood $U$ of $x$.

\begin{theorem}\cite[\S4]{ArmstrongP1}\label{theorem:projective-to-c-projective}
Let $(M, \mbp)$ be an oriented projective manifold of dimension $2m+1 \geq 5$ equipped with a parallel tractor complex structure $\bbJ \in \Gamma(\mcA)$, and suppose that the underlying vector field $k := \Pi^{\mcA}(\bbJ)$ is a projective symmetry. Then any (sufficiently small) local leaf space $\wtM$ of the vector field $k$ inherits a canonical c-projective structure.
\end{theorem}

We now prove this theorem in detail, using language somewhat different from that in \cite[\S4]{ArmstrongP1}, with a view toward proving its converse in an explicit way (Theorem \ref{theorem:c-projective-to-projective} below).
Suppose the assumptions of Theorem \ref{theorem:projective-to-c-projective} are satisfied, and let $\pi: U\rightarrow \wtM$ be a local leaf space of $k$ such that $\mbp|_U$ admits $k$-adapted scales. Recall that a
$k$-adapted scale $\nabla \in \mbp$ (on $U$) determines:

\begin{itemize}
\item a hyperplane distribution $H \subset TU$ transverse to $k$, and hence for each $x \in U$ an isomorphism $T_x \pi \vert_{H_x} : \smash{H_x \rightarrow T_{\pi(x)} \wtM}$ (its inverse gives a preferred lift of each $\tilde\xi \in T_{\pi(x)} \wtM$ to $T_x U$),
\item a complex structure $J_H$ on $H$, and
\item a connection $\nabla^H$ on $H$.
\end{itemize}

For later, we record how $H$ and the corresponding lift vary under a change of $k$-adapted scale: Assume that $\nabla, \nabla' \in \mbp$ are two such scales related by a $1$-form $\Upsilon_a \in \Gamma(T^* M)$ as in Proposition \ref{proposition:k-adapted-connections}. It then follows immediately from \eqref{equation:adjoint-component-transformation} and Proposition \ref{proposition:k-adapted-connections}(b) that the $1$-forms $\nu, \nu'$ defining $H, H'$ as in \eqref{equation:decomposition-TM} are related by
\[
	\nu'_b = \nu_b - \Upsilon_c \nabla_b k^c
\]
and, using the third equation of \eqref{equation:complex-structure-components-k-adapted}, that the corresponding preferred lifts $\xi, \xi' \in T_x U$ of a vector $\tilde\xi \in T_{\pi(x)} \wtM$ to $H_x = \ker \nu_x$ and $H'_x = \ker \nu'_x$
are related by
\begin{equation}\label{equation:lift-tranformation}
	(\xi')^a = \xi^a - (\Upsilon_c \xi^b \nabla_b k^c) k^a .
\end{equation}

By assumption $k$ is a projective symmetry and hence by Proposition \ref{proposition:k-preserves-k-adapted-scale} the flow of $k$ preserves any $k$-adapted scale $\nabla \in \mbp$.
Thus, it preserves invariants of $\nabla$, in particular the projective Schouten tensor $\Rho$ of $\nabla$, the $1$-form $\nu_a = -\Rho_{ab} k^b$, the hyperplane distribution $H= \ker \nu$, the complex structure $J_H \in \Gamma(\End H)$, and the connection $\nabla^H$. Now, $J_H$ descends to an almost complex structure $\tilde J \in \Gamma(\End T\wtM)$ on $\widetilde{M}$, and it follows from the definitions of $J_H$ and $\tilde J$ and \eqref{equation:adjoint-component-transformation} that $\tilde J$ does not depend on $\nabla$ (cf. \cite[Theorem 4.1]{ArmstrongP1}).
Likewise, $\nabla^H$ descends to an affine connection $\wtnabla$ on $\wtM$---which \textit{does} depend on $\nabla$---and which is torsion-free because $\nabla$ is. Moreover, it follows from the definition of $\tilde J$ and equation
\eqref{equation:components-J-parallel} that $\wtnabla$ preserves $\tilde J$. Since $\wtnabla$ is torsion-free, this implies in particular that $\tilde J$ is integrable (cf. \cite[Proposition 4.6]{ArmstrongP1}). In summary:

\begin{proposition}\label{proposition:compatibility-bar-nabla}
Suppose the assumptions of Theorem \ref{theorem:projective-to-c-projective} are satisfied. Then the projective structure $\mbp$ on $M$ induces on any (sufficiently small) local leaf space $\wtM$ of $k$ an (integrable) complex structure $\tilde J$. Moreover, any $k$-adapted scale $\nabla\in\mbp$ gives rise to a torsion-free connection $\wtnabla$ on $\wtM$ which is complex with respect to $\tilde J$.
\end{proposition}

In view of Proposition \ref{proposition:compatibility-bar-nabla}, in order to prove Theorem \ref{theorem:projective-to-c-projective} it remains to show:

\begin{proposition}\label{proposition:induced-c-projective-structure}
Suppose the assumptions of Theorem \ref{theorem:projective-to-c-projective} are satisfied and let $\pi: U\rightarrow\wtM$ be a sufficiently small local leaf space of $k$ (so that $\mbp$ admits $k$-adapted scales). Then for any two $k$-adapted scales  $\nabla, \nabla' \in \mbp|_U$ the induced complex torsion-free connections $\wtnabla, \wtnabla'$ on $(\wtM, \tilde J)$ are c-projectively equivalent.
\end{proposition}

\begin{proof}
In fact we show more: Suppose $\nabla, \nabla' \in\mbp|_U$ are two $k$-adapted scales related by a $1$-form $\Upsilon\in\Gamma(T^*U)$ as in Proposition \ref{proposition:k-adapted-connections}(b). Since $\Upsilon$ is $k$-invariant and annihilates $k$, $\Upsilon = \pi^* \tilde \Upsilon$ for a unique $1$-form $\tilde\Upsilon \in \Gamma(T^* \wtM)$. We show that $\wtnabla$ and $\wtnabla'$ are exactly related via $\tilde \Upsilon$ in the c-projective equivalence equation \eqref{equation:c-projective-transformation}.

Suppose that $\tilde\xi$ and $\tilde\eta$ are vector fields on $\widetilde M$ and denote by $\xi,\eta$ and $\xi',\eta'$ their lifts to vector fields on $U$ determined by $\nabla$ and $\nabla'$ respectively. Recall that at any point $\tilde x\in \widetilde M$ one has by definition $(\wtnabla_{\tilde\eta} \tilde\xi)_{\tilde x} := T_x \pi \cdot \nabla_{\eta} \xi$ and $(\wtnabla'_{\tilde\eta} \tilde\xi)_{\tilde x} := T_x \pi \cdot \nabla'_{\eta'} \xi'$ for any point $x\in\pi^{-1}(\tilde x)$.
Together Proposition \ref{proposition:k-adapted-connections} and \eqref{equation:lift-tranformation} give (using that $\Upsilon_a k^a = 0$) that
\[
	\nabla'_b (\xi')^a=\nabla_b \xi^a - \Upsilon_d (\nabla_e k^d) \xi^e \nabla_b k^a + \xi^a \Upsilon_b + \Upsilon_c \xi^c \delta^a{}_b + k^a \psi_b
\]
for some $1$-form $\psi_a \in \Gamma(T^* M)$.
Contracting both sides of the equation with $(\eta')^b$, again using \eqref{equation:lift-tranformation}, and using that by \eqref{equation:complex-structure-components-k-adapted} $k^b \nabla_b k^a = 0$ moreover give that
\begin{align*}
	\nabla'_{\eta'} \xi'
		= \nabla_{\eta} \xi + \Upsilon(\eta) \xi + \Upsilon(\xi) \eta - \Upsilon(\nabla_{\eta} k) \nabla_{\xi} k - \Upsilon(\nabla_{\xi} k) \nabla_{\eta} k + f k
\end{align*}
for some function $f \in \Gamma(\mcE)$. Since for any $\tilde\zeta \in \Gamma(T\tilde M)$ with lift $\zeta\in\Gamma(TU)$ (determined by $\nabla$), one has by definition that
$T_x \pi \cdot (\nabla_{\zeta} k)_x = \tilde J_{\tilde x}(\tilde{\zeta}_{\tilde x})$ for any $x\in\pi^{-1}(\tilde x)$, we conclude that
\begin{equation*}
	\wtnabla'_{\tilde\eta} \tilde\xi
		= \wtnabla_{\tilde\eta} \tilde\xi
			+ \tilde\Upsilon(\tilde\eta) \tilde\xi + \tilde\Upsilon(\tilde\xi) \tilde\eta
			- \tilde\Upsilon(\tilde J\tilde\eta) \tilde J\tilde\xi - \tilde\Upsilon(\tilde J\tilde\xi) \tilde J\tilde\eta,
\end{equation*}
as claimed.
\end{proof}

\subsubsection{The c-projective--to--projective Fefferman-type construction}

We now describe the inverse of the construction in Section
\ref{subsubsection:c-projective-leaf-space}: We fix a c-projective
structure $(\wtM, \tilde J, \tilde\mbp)$ of dimension $2 m$ and
construct a canonical circle bundle $M \to \wtM$ whose total space is
equipped with a projective structure $\mbp$ and a
$\nabla^{\mcT}$-parallel tractor complex structure $\bbJ$. Our
treatment is formally similar that of the classical Fefferman
construction in \cite[\S2.4]{CapGoverCRTractors}.

As in Section \ref{subsubsection:c-projective-primer}, we fix a choice
of $(m + 1)$st root $\mcE(1, 0)$ of the complex line bundle
$\smash{\Wedge^m T^{1, 0} \wtM}$, and we denote its dual by $\mcE(-1,
0) := \mcE(1, 0)^{\ast}$. Then, we define the \textit{Fefferman space}
$M$ of $\wtM$ to be the fibrewise projectivization $\bbP_+(\mcE(-1,
0))$; with $\bbP_+(\mcE(-1, 0))_x$ the rays in $\mcE(-1, 0)_x$ for all
$x\in \wtM$. Denote by $\mcF$ the bundle obtained by removing from
$\mcE(-1, 0)$ the image of the zero section; then $\bbC^*$ acts freely
on $\mcF$ (on the right) and transitively on each fibre. We can
restrict this action to $\Bbb R^+$ and identify $M$ with $\mcF / \Bbb
R^+$, so the natural projection $\pi: M \to \wtM$ is a principal fibre bundle with structure
group $\bbC^* / \bbR^+ \cong \U(1)$. Denote by $k \in \Gamma(TM)$ the
infinitesimal generator of this action.

\begin{theorem}\label{theorem:c-projective-to-projective}
\leavevmode
\begin{enumerate}
	\item Let $(\wtM, \tilde J, \tilde\mbp, \mcE(1, 0))$ be a c-projective structure of dimension $2 m \geq 4$. Then, the Fefferman space $M$ inherits a canonical projective structure $\mbp$ for which
		\begin{enumerate}
			\item the infinitesimal generator $k$ of the $\U(1)$-action is an infinitesimal projective symmetry, and
			\item the adjoint tractor $\bbJ := L^{\mcA}(k) \in \Gamma(\mcA)$ is a parallel complex structure.
		\end{enumerate}
	\item The construction in (a) and that in Theorem \ref{theorem:projective-to-c-projective} are (locally) inverses of one another.
\end{enumerate}
\end{theorem}

See \cite[Theorem 3.7]{ArmstrongP1} for an approach to Theorem \ref{theorem:c-projective-to-projective} in terms of the Thomas cone.

\begin{remark}
While the definition of the Fefferman space over a c-projective manifold requires a choice of line bundle $\mcE(1, 0)$, this choice is inessential in the sense that different choices yield equivalent projective structures.
\end{remark}

\begin{proof}[Proof of Theorem \ref{theorem:c-projective-to-projective}]
\leavevmode
\begin{enumerate}
	\item \noindent
\textit{Step 1: Construct a connection $\nabla$ on $M$ from a connection $\wtnabla \in \tilde\mbp$.}
	
Fix a (c-projective) scale $\tilde\sigma\in\mcE(1, 1)$, and denote by $\wtnabla \in \tilde\mbp$  the corresponding connection in the c-projective class and by $\wtRho\in\Gamma(T^*M\otimes T^*M)$ its \textit{c-projective Rho tensor}: For any such scale, $\wtRho$ is characterised by \cite[Proposition 2.11]{CEMN}\footnote{The c-projective Rho tensor, $\wtRho$, is the negative of the ``complex rho-tensor'' in \cite[\S3.5.1]{ArmstrongThesis} and the ``complex projective rho tensor'' in \cite[\S4.1]{ArmstrongP1}, and half the ``rho tensor'' in \cite[\S2.4]{CEMN}.}
\begin{equation}\label{equation:Rho-tilde-characterization}
	\wt\Ric(\tilde \xi, \tilde\eta)
		= 2 m \wtRho(\xi, \eta) + 2 \wtRho(\tilde J \tilde\xi, \tilde J \tilde\eta) ,
\end{equation}
where $\wt\Ric$ is the Ricci tensor of $\wtnabla$; in particular $\wtRho$ is symmetric.

The scale $\tilde\sigma$ can be viewed as a Hermitian metric on
$\mcE(-1, 0)$; in particular we may identify $M:= \Bbb P_+(\mcE(-1,
0))$ with the unit circle subbundle of $\mcE(-1, 0)$ that this metric
determines.  Moreover, $\wtnabla\in\tilde\mbp$ induces a linear
connection on $\mcE(-1,0)$, that, by definition, preserves the
metric $\tilde\sigma$, that is, satisfies
$\wtnabla\tilde\sigma=0$. Hence, it determines a $\U(1)$-principal
connection $TM \to \mfu(1) \cong i \Bbb R$ on the unit circle bundle
$M\rightarrow \wtM$ of $\mcE(-1,0)$.

This principal connection can equivalently be encoded as a hyperplane distribution $H \subset TM$ everywhere transverse to the vertical (line) bundle $\ker T\pi = \langle k \rangle \subset TM$ and invariant under the flow of the infinitestimal generator $k$. Let $\Phi : H \to T\wtM$ denote the bundle isomorphism $T \pi\vert_{H}$, and set $J := \Phi^{-1} \circ \tilde J \circ \Phi$. For $\tilde\xi \in \Gamma(T\wtM)$ we write $\xi \in \Gamma(H) \subset \Gamma(TM)$ for its horizontal lift, that is, the unique section of $H$ such that $T_x \pi \cdot \xi_x = \tilde\xi_{\pi(x)}$ for all $x \in M$. Since the flow of $k$ preserves both the fibres of the principal bundle and the distribution $H$, it also preserves any such lift $\xi$: $\mcL_k \xi = 0$.

Using \cite[(45)-(47)]{CEMN} (and taking into account the difference in convention mentioned in the previous footnote), the curvature of the connection determined by $\wtnabla$ on $\mcE(-1, 0)$ is $(\tilde\xi, \tilde\eta) \mapsto -i [\wtRho(\tilde \xi, \tilde J \tilde\eta) - \wtRho(\tilde\eta, \tilde J \tilde\xi)]$, and so we may identify the curvature of the induced principal connection on the unit circle bundle with
\begin{equation}\label{curv_of_H}
	(\xi, \eta)
		\mapsto
	-\pi^*[  \wtRho(\tilde\xi, \tilde J \tilde\eta)
		   - \wtRho(\tilde\eta, \tilde J \tilde\xi)] k .
\end{equation}

Now define a connection $\nabla$ on $M$ by
\begin{equation}\label{equation:nabla-tilde-to-nabla}
	\begin{alignedat}{3}
		\nabla_k        k &= 0 &\qquad
		\nabla_k      \xi &= J \xi \\
		\nabla_{\eta}   k &= J \eta &\qquad
		\nabla_{\eta} \xi &= \Phi^{-1} (\wtnabla_{\tilde\eta} \tilde\xi) + (\pi^* \wtRho)(\eta, J\xi) k
	\end{alignedat}
\end{equation}
for $\xi,\eta\in\Gamma(H)$ lifting $\tilde\xi, \tilde\eta \in \Gamma(T\wtM)$.

Unwinding definitions gives that $\pi^*\mcE(1,1)\cong \mcE(2)$ and that, by construction, $\nabla \pi^*\tilde\sigma=0$.

\medskip
\noindent \textit{Step 2: Compute the affine invariants of $\nabla$.}

Computing directly gives that the torsion $T$ of $\nabla$ satisfies $T(k, \,\cdot\,) = 0$, so that $T$ vanishes if $T\vert_{H}$ does, and that
\[
	T(\xi, \eta) = \Phi^{-1}([\tilde\xi, \tilde\eta]) - [\xi, \eta] + [(\pi^*\wtRho)(\xi, J\eta) - (\pi^* \wtRho)(\eta, J\xi)] k ,
\]
for all $\tilde\xi, \tilde\eta \in \Gamma(T\wtM)$. The sum of the first and second terms on the right-hand side is by definition the curvature of the principal bundle connection determined by $H$. We can rewrite the third term as $\pi^*[\wtRho(\tilde\xi, \tilde J \tilde\eta) - \Rho(\tilde\eta, \tilde J \tilde\xi)] k$, but by \eqref{curv_of_H} this is precisely the negative of that curvature. Thus, $T\vert_{H} = 0$, and hence $\nabla$ is torsion-free.

Computing gives that the curvature of $\nabla$ decomposes with respect to the splitting $TM = \langle k \rangle \oplus H$ as
\begin{align*}
	R(k, \xi) k
		&= -\xi \\
	R(k, \xi) \eta
		&= (\pi^*\wtRho)(\xi, \eta) k \\
	R(\xi, \eta) k
		&= 0 \\
	R(\xi, \eta) \zeta
		&=
			\Phi^{-1} (\tilde R(\tilde\xi, \tilde\eta) \tilde\zeta)
				+ (\pi^* \wtnabla_{\tilde\xi } \wtRho)(\eta, J\zeta) k
				- (\pi^* \wtnabla_{\tilde\eta} \wtRho)(\xi , J\zeta) k \\
				& \qquad\qquad
				+ (\pi^* \wtRho)(\eta, J\zeta) J\xi
				- (\pi^* \wtRho)(\xi , J\zeta) J\eta \\
				& \qquad\qquad\qquad\qquad
				- (\pi^* \wtRho)(\xi , J\eta ) J\zeta
				+ (\pi^* \wtRho)(\eta, J\xi  ) J\zeta,
\end{align*}
and that the projective Schouten tensor $\Rho = \frac{1}{2m} \Ric$ of $\nabla$ decomposes as
\begin{align}
	\Rho(  k,    k) &= 1 \label{equation:Schouten-k-k} \\
	\Rho(  k,  \xi) &= 0 \label{equation:Schouten-k-H} \\
	\Rho(\xi, \eta) &= (\pi^*\wtRho)(\xi, \eta) \label{equation:Schouten-H-H}
\end{align}
where $\xi, \eta, \zeta \in \Gamma(H)$ are lifts of $\tilde\xi, \tilde\eta, \tilde\zeta \in \Gamma(T\wtM)$ respectively. These identities and the projective curvature decomposition \eqref{decp} give that any contraction of $k$ with the projective Weyl curvature $W$ vanishes:
\begin{equation}\label{equation:proof-k-annihilates-W}
	k^a W_{ab}{}^c{}_d = 0, \qquad W_{ab}{}^c{}_d k^d = 0 .
\end{equation}

\medskip
\noindent\textit{Step 3: Verify that the projective structure $\mbp := [\nabla]$ is independent of the choice $\wt\nabla$.}

Any other choice of connection $\wtnabla' \in \tilde\mbp$ can be
written in terms of $\wtnabla$, as in
\eqref{equation:c-projective-transformation}
for some real $1$-form $\tilde\Upsilon \in \Gamma(T\wtM)$, and determines a connection $\nabla'$ on $M$ by appropriately replacing the appropriate unprimed symbols with primed ones in \eqref{equation:nabla-tilde-to-nabla}. We show that $\nabla$ and $\nabla'$ are projectively equivalent, that is, that
\begin{equation}\label{equation:projective-equivalence-proof}
	\nabla'_b \zeta^a = \nabla_b \zeta^a + \Upsilon_b \zeta^a + \Upsilon_c \zeta^c \delta^a{}_b
\end{equation}
for some $\Upsilon \in \Gamma(T^*M)$ and all $\zeta \in \Gamma(TM)$, in fact for $\Upsilon := \pi^* \tilde\Upsilon$.

We again decompose \eqref{equation:projective-equivalence-proof} according to the splitting $\langle k \rangle \oplus H$ and show separately that each equation in the resulting system holds. Using that $\Upsilon(k) = 0$, the system is
\begin{align}	
	\nabla'_k k
		&= 0 \label{equation:invariance-k-k} \\
	\nabla'_k \xi
		&= J \xi + \Upsilon(\xi) k \label{equation:invariance-k-xi} \\
	\nabla'_{\eta} k
		&= J \eta + \Upsilon(\eta) k \label{equation:invariance-eta-k} \\
	\nabla'_{\eta} \xi
		&= \nabla_{\eta} \xi + \Upsilon(\eta) \xi + \Upsilon(\xi) \eta , \label{equation:invariance-eta-xi}
\end{align}
where $\xi, \eta \in \Gamma(H)$ are projectable.

The first equation is just the primed version of the first component of the system \eqref{equation:nabla-tilde-to-nabla}. To derive the others, we use the decomposition of a section $\zeta \in \Gamma(H)$ into the sum of section $\zeta' \in H'$ and a multiple of $k$: Using \eqref{equation:lift-tranformation} and the third equation in \eqref{equation:nabla-tilde-to-nabla} gives
\begin{equation}\label{equation:xi-xi-prime}
	\xi = \xi' + \Upsilon(J \xi) k .
\end{equation}

The left-hand side of \eqref{equation:invariance-k-xi} is
\[
	\nabla'_k \xi
		= \nabla_k' (\xi' + \Upsilon(J \xi) k)
		= \nabla_k' \xi' + (\nabla'_k[\Upsilon(J \xi)]) k + \Upsilon(J \xi) \nabla'_k k
		= J' \xi' .
\]
The second term in the third expression vanishes because the function $\Upsilon(J \xi)$ is $k$-invariant (as $\Upsilon, J, \xi$ all are), and the third term there vanishes because $\nabla'_k k = 0$. As $J \xi \in \Gamma(H)$ and $J' \xi' \in \Gamma(H')$ and are both lifts of $\tilde J\tilde\xi$, \eqref{equation:xi-xi-prime} gives $\nabla'_k \xi = J' \xi' = J \xi - \Upsilon(J(J \xi)) k = J \xi + \Upsilon(\xi) k$, which is the condition \eqref{equation:invariance-k-xi}.

The third condition, \eqref{equation:invariance-eta-k}, then just follows by torsion-freeness of $\nabla'$ from  \eqref{equation:invariance-k-xi}.

Showing the last condition, \eqref{equation:invariance-eta-xi}, is more involved. First, decomposing using \eqref{equation:xi-xi-prime} gives
\begin{equation}\label{equation:decomposition-nabla-prime-xi-eta}
	\begin{aligned}
		\nabla'_{\eta} \xi
			&= \nabla'_{\eta' + \Upsilon(J\eta) k} (\xi' + \Upsilon(J\xi) k) \\
			&= \underbrace{\nabla'_{\eta'} \xi'}_{S_1}
				+ \underbrace{\nabla'_{\eta'} [\Upsilon(J\xi) k]}_{S_2}
				+ \underbrace{\Upsilon(J\eta) \nabla'_k \xi'}_{S_3}
				+ \underbrace{\Upsilon(J\eta) \nabla'_k [\Upsilon(J\xi) k]}_{S_4} .
	\end{aligned}
\end{equation}
The first term, $S_1$, is the most complicated: Rewriting the last equation in the primed incarnation of \eqref{equation:nabla-tilde-to-nabla}, applying \eqref{equation:xi-xi-prime} to $\tilde\zeta = \wtnabla'_{\tilde\eta} \tilde\xi$, invoking \eqref{equation:projective-equivalence-proof}, and simplifying gives
\begin{align*}
	(\Phi')^{-1} (\wtnabla'_{\tilde\eta} \tilde\xi)
		&= \Phi^{-1} (\wtnabla_{\tilde\eta} \tilde\xi) + \Upsilon(\eta) \xi + \Upsilon(\xi) \eta - \Upsilon(J\eta) J\xi - \Upsilon(J\xi) J\eta \\
		&\qquad- \Upsilon(J \Phi^{-1} (\wtnabla_{\tilde\eta} \tilde\xi)) k - 2 \Upsilon(\eta) \Upsilon(J\xi) k - 2 \Upsilon(\xi) \Upsilon(J \eta) k .
\end{align*}
Now, rewriting \cite[Proposition 2.11]{CEMN} (and again accounting for the differences in the conventions for $\wtRho$ and $\tilde\Upsilon$) gives that the c-projective Rho tensors for $\wtnabla$ and $\wtnabla'$ are related by
\[
	\wtRho'(\tilde\eta, \tilde J \tilde\xi)
		= \wtRho(\tilde\eta, \tilde J \tilde\xi) - (\wtnabla_{\tilde\eta} \tilde\Upsilon)(\tilde J \tilde\xi) + \tilde\Upsilon(\tilde\eta) \tilde\Upsilon(\tilde J \tilde\xi) + \tilde\Upsilon(\tilde\xi) \tilde\Upsilon(\tilde J \tilde\eta) .
\]
Using that fact and that
\[
\qquad
\Upsilon(J \Phi^{-1} (\wtnabla_{\tilde\eta} \tilde\xi))
	= (\pi^* \tilde\Upsilon)(\Phi^{-1} (\tilde J \wtnabla_{\tilde\eta} \tilde\xi)) \\
	= \pi^*(\tilde\Upsilon(T\pi \cdot \Phi^{-1} (\tilde J \wtnabla_{\tilde\eta} \tilde\xi)))
	= \pi^*(\tilde\Upsilon(\tilde J \wtnabla_{\tilde\eta} \tilde\xi)),
\]
applying the Leibniz rule to the quantity $\wtnabla_{\tilde\eta} [\tilde\Upsilon(\tilde J \tilde\xi)]$, and invoking the last equation in \eqref{equation:nabla-tilde-to-nabla} gives
\begin{align*}
	S_1
		&= \nabla_{\eta} \xi + \Upsilon(\eta) \xi + \Upsilon(\xi) \eta - \Upsilon(J\eta) J\xi - \Upsilon(J\xi) J\eta \\
		&\qquad\qquad - \pi^* (\wtnabla_{\tilde\eta} [\tilde \Upsilon(\tilde J \tilde\xi)]) k - \Upsilon(\xi) \Upsilon(J\eta) k - \Upsilon(\eta) \Upsilon(J\xi) k
\end{align*}

Applying the Leibniz rule again, using $k$-invariance, and unwinding definitions gives that the remaining terms in \eqref{equation:decomposition-nabla-prime-xi-eta} can be written as
\begin{align*}
	S_2
		&= \pi^*(\wtnabla_{\tilde\eta} [\tilde\Upsilon(\tilde J\tilde\xi)])
			+ \Upsilon(J\xi) J\eta + \Upsilon(J\xi) \Upsilon(\eta) k \\
	S_3 &= \Upsilon(J\eta) J\xi + \Upsilon(\xi) \Upsilon(J\eta) k \\
	S_4 &= 0 .
\end{align*}

Substituting for $S_1, S_2, S_3, S_4$ in \eqref{equation:decomposition-nabla-prime-xi-eta} gives precisely \eqref{equation:invariance-eta-xi}, and so $\nabla$ and $\nabla'$ are related by \eqref{equation:projective-equivalence-proof} and in particular are projectively equivalent.

\medskip
\noindent \textit{Step 4: Verify that the adjoint tractor $\Bbb J := L^\mathcal A(k)$ is a parallel complex structure.}

Note that, by its definition, $\nabla$ is $k$-adapted (that is, that $\nabla_c k^c = 0$), and so \eqref{equation:splitting-operator-adjoint} gives
\begin{align*}
	\Bbb J^A{}_C \Bbb J^C{}_B
		&=
			\begin{pmatrix}
		 		\nabla_c k^a & k^a  \\
		 		-\Rho_{cd} k^d & 0
		 	\end{pmatrix}
		 	\begin{pmatrix}
		 		\nabla_b k^c & k^c \\
		 		-\Rho_{bd} k^d & 0 \\
		 	\end{pmatrix} \\
		&=
			\begin{pmatrix}
		 		\nabla_b k^c \nabla_c k^a - k^a \Rho_{bd} k^d & k^c\nabla_c k^a \\
		 		-\Rho_{cd} k^d \nabla_b k^c & -\Rho_{cd} k^c k^d
		 	\end{pmatrix} .
\end{align*}
The definition of $\nabla$ and \eqref{equation:Schouten-k-k} and \eqref{equation:Schouten-k-H} give that this is $-\id_{\mcT}$.

It remains to show that $\Bbb J$ is $\nabla^{\mathcal V}$-parallel and
that $k$ is a projective symmetry of $\mbp$. In view of Theorem
\ref{theorem:projective-symmetry} and the fact
\eqref{equation:proof-k-annihilates-W} that $W$ vanishes upon
insertion of $k$, it suffices to show that $k$ is a solution of the
adjoint BGG operator $\mathcal D^{\mcA}$
\eqref{equation:BGG-operator}. As usual, we decompose the condition,
in this case $\mathcal D^{\mcA}(k) = 0$, according to the splitting
$TM \cong \langle k \rangle \oplus H$, that is, we show that $\mathcal
D^{\mcA}(k)^a{}_{bc}$ vanishes upon contraction with any pair $k^b
k^c$, $k^b \xi^c$, $\xi^b \eta^c$, where $\xi, \eta$ are again lifts
determined by arbitrary sections $\tilde\xi, \tilde\eta \in
\Gamma(T\wtM)$.

Since $\nabla$ is $k$-adapted, the terms containing derivatives of
$\nabla_d k^d$ vanish. Then, using
\eqref{equation:nabla-tilde-to-nabla}, \eqref{equation:Schouten-k-k},
\eqref{equation:Schouten-H-H}, and the trace-freeness of $J$ gives the
second covariant derivative identities $k^b \nabla_b \nabla_c k^a =
0$, $\xi^b k^c \nabla_b \nabla_c k^a = \xi^a$, and $\xi^b \eta^c
\nabla_b \nabla_c k^a = -\Rho_{bc} \xi^b \eta^c k^a$, and the first
and second identities together give the trace identity $k^c \nabla_d
\nabla_c k^d = 2m$. These identities and \eqref{equation:Schouten-k-H}
are together enough to verify immediately upon substitution into
\eqref{equation:BGG-operator} (after setting $\xi = k$ in that
formula) that all three of the above contractions vanish, and hence
that $\mathcal D^{\mcA}(k) = 0$.
	\item In fact, even more is true: The constructions $\nabla \rightsquigarrow \wtnabla$ in \S\ref{subsubsection:c-projective-leaf-space} and $\wtnabla \rightsquigarrow \nabla$ in the proof of part (a) between individual connections are inverses.
	
Checking this in one direction is almost immediate: Given a c-projective scale $\wtnabla \in \tilde\mbp$ and vector fields $\tilde\xi, \tilde\eta \in \Gamma(T\wtM)$, by \eqref{equation:nabla-tilde-to-nabla} the resulting connection $\nabla$ satisfies $\nabla_{\eta} \xi = \Phi^{-1} (\wtnabla_{\tilde\eta} \tilde\xi) + (\pi^* \wtRho)(\eta, J\xi) k$. But by the characterisation in the proof of Proposition \ref{proposition:induced-c-projective-structure}, the c-projective connection induced by $\nabla$ is $(\tilde\eta, \tilde\xi) \mapsto T\pi \cdot [\Phi^{-1} (\wtnabla_{\tilde\eta} \tilde\xi) + (\pi^* \wtRho)(\eta, J\xi) k] = \wtnabla_{\tilde\eta} \tilde\xi$.

The reverse assertion follows from the claim for the direction just discussed together with the facts  (1) that, (again) from the proof of Proposition \ref{proposition:induced-c-projective-structure},  changing the projective representative (between $k$-adapted scales) via \eqref{projective_change} by the $1$-form $\Upsilon$ induces a change of c-projective representative via \eqref{equation:c-projective-transformation} by $\tilde\Upsilon$ (for $\Upsilon = \pi^* \Upsilon$), and (2) that the reverse holds, as observed in the proof of part (a).\qedhere
\end{enumerate}
\end{proof}

\subsubsection{Relation between projective and c-projective tractors}
\label{section_proj_cproj_tractors}
Suppose $(M, \mbp)$ is a projective manifold of dimension $2m+1\geq 5$ and
assume that its tractor bundle $\mcT$ is
equipped with a parallel complex structure
$\bbJ\in\Gamma(\mcA)$ such that the underlying nowhere
vanishing vector field $k:=\Pi^{\mcA}(\bbJ)$ is a projective
symmetry. We denote the (local) flow of $k$ by $\phi_s$.  As in Theorem
\ref{theorem:projective-to-c-projective}, consider a local leaf space
$\pi: U\rightarrow \wtM$ of $k$ and denote by $(\tilde J, \tilde\mbp)$
the induced c-projective structure on $\wtM$.  Recall that, by
Proposition \ref{proposition:complex-single-curved-orbit}(b), our
assumptions imply at least locally the existence of a parallel tractor complex volume form on $U$. By choosing the local leaf space
sufficiently small we shall assume that it is defined on all of $U$
and denote it by $\smash{\vol_{\bbC}}\in\Gamma(\Lambda_\bbC^{m+1} \mcT\vert_U)$.

As in the case of projective structures, c-projective structures
admit a canonical connection on a certain vector bundle over $\wtM$, called
the \emph{(normal) c-projective tractor connection} (see \cite{CEMN}).
We show in this section that the projective tractor bundle $(\mcT, \nabla^{\mcT})$ of $(M, \mbp)$
naturally descends via $\pi$ to a vector bundle $\wtmcT$ over
$\wtM$ with connection $\nabla^{\wtmcT}$ and that
$(\wtmcT, \nabla^{\wtmcT})$ coincides with the canonical
c-projective tractor bundle and connection of $(\wtM, \tilde J,
\tilde\mbp, \mcE(1, 0))$.

Let $\wtmcT$ be the quotient of $\mcT\vert_U$ by the equivalence
relation $\sim$, where $t_x\sim t_{x'}$ for $t_x, t_{x'}\in\mcT$ if and only if
(1) $x'=\phi_s(x)$ for some $s$ (i.e. $\pi(x)=\pi(x')$) and (2) $t_{x'}$ is
the result of parallel-transporting $t_x$ along $s\mapsto
\phi_s(x)$. We equip $\wtmcT$ with the initial topology of the
natural projection $\tilde{p}:\wtmcT\rightarrow \wtM$. Note also
that, by construction, the fibres of $\tilde p$ are naturally vector
spaces. Choosing local sections of $\pi$, it is easy to see that the
vector bundle charts of $\mcT\vert_U$ give rise to local
trivialisations of $\tilde p: \wtmcT\rightarrow \wtM$ and hence
that the latter can naturally be given the structure of a smooth
vector bundle such that the following diagram commutes:
\begin{center}
\begin{tikzcd}
  \mcT \arrow[r, "\mathrm{proj}"] \arrow[d, "p"]
    & \wtmcT \arrow[d, "\tilde p"] \\
  U \arrow[r, "\pi"]
& \widetilde M \end{tikzcd},
\end{center}
where $\mathrm{proj}:\mcT \rightarrow \wtmcT=\mcT/\sim$ denotes the projection. The universal property of the pullback bundle and the fact that
$\mathrm{proj}\vert_{\mcT_x}: \mcT_x\rightarrow \wtmcT_{\pi(x)}$ is an isomorphism for all $x\in U$ immediately imply that $\pi^*{\wtmcT}\cong \mcT$.

\begin{lemma}\label{c-proj_tractor_bundle_lemma}
\leavevmode
\begin{enumerate}
\item $\Gamma(\wtmcT)$ can be identified with the subsheaf of $\Gamma(\mcT\vert_U)$ consisting precisely of those sections $t\in\Gamma(\mcT\vert_U)$ that satisfy $\nabla_k^{\mcT} t=0$.
\item The restriction of the connection $\nabla^{\mcT}$ to $\Gamma(\wtmcT)\subset \Gamma(\mcT\vert_U)$ descends to a connection $\nabla^{\wtmcT}$ on $\wtmcT$.
\item The complex structure $\bbJ$ and the complex volume form $\vol_{\bbC}$ on $\mcT\vert_U$ descend to a $\nabla^{\wtmcT}$-parallel complex structure $\tilde\bbJ\in\Gamma(\End \wtmcT)$ and
a $\nabla^{\wtmcT}$-parallel complex volume form $\tilde\vol_{\bbC}\in\Gamma(\Wedge^{m+1}_\bbC \wtmcT)$ on $\wtmcT$ respectively. Moreover, $\bbJ$ gives rise to a $\tilde\bbJ$-invariant subbundle $\tilde \mcS\subset \wtmcT$
of (real) rank $2$.
\end{enumerate}
\end{lemma}
\begin{proof}
\leavevmode
\begin{enumerate}
	\item By definition of $\wtmcT$, for any section $\tilde t\in\Gamma(\wtmcT)$ the composition $\tilde t\circ\pi$ can be naturally interpreted as a section of $\mcT$; evidently the sections $t$ of $\mcT$ that arise in this way from sections $\wtmcT$ are precisely those that satisfy $\nabla^{\mcT}_kt=0$.
	\item Fix a vector field $\tilde\xi\in\Gamma(T\wtM)$ and a section $\tilde t\in\Gamma(\wtmcT)$, which we view as a section of $\mcT$ satisfying $\nabla_k^{\mcT} \tilde t=0$.
Let $\xi\in\Gamma(TU)$ be any lift of $\tilde\xi$ and consider $\nabla^{\mcT}_\xi \tilde t$. Note that $\nabla^{\mcT}_k\tilde t=0$ implies that $\nabla^{\mcT}_\xi \tilde t$ is independent of the choice of lift $\xi$.
Moreover, since $\xi$ is projectable as a lift of $\tilde \xi$, we have $T\pi \cdot [k,\xi]=0$ and hence also $\smash{\nabla^{\mcT}_{[k,\xi]}\tilde t=0}$. Moreover, by the formula for the tractor curvature \eqref{tractor_curvature}, Theorem \ref{theorem:projective-symmetry} and \eqref{diff_Bianchi_Weyl} we have $R^{\mcT}(k,\xi)(\tilde t)=0$. Therefore, expanding
the left-hand side of this identity yields $\smash{\nabla^{\mcT}_k\nabla^{\mcT}_\xi \tilde t=0}$. Hence, $\nabla^{\mcT}_\xi \tilde t$ defines a section of $\wtmcT$, which only depends on $\tilde\xi$ and $\tilde t$.
Denoting this section by $\nabla^{\wtmcT}_{\tilde\xi}\tilde t$, one verifies directly that $\nabla^{\wtmcT}_{\tilde\xi}\tilde t$  defines a linear connection $\nabla^{\wtmcT}: \Gamma(\wtmcT)\rightarrow \Gamma(T^*\wtM\otimes \wtmcT)$
on $\wtmcT$.
	\item The first statement follows immediately from the fact that $\bbJ$ and $\vol_{\bbC}$ are parallel and the definitions of $\wtmcT$ and $\nabla^{\wtmcT}$. For the second statement recall that $\mcT$ naturally comes equipped
with a real line subbundle
$\mcT^1:=\langle X \rangle \otimes \mcE(-1)\subset \mcT$.  Moreover, the complex structure $\bbJ$ gives rise to another copy of $\mcE(-1)$ inside $\mcT$ given by the real line bundle 
$\bbJ \mcT^1= \langle \bbJ X \rangle \otimes \mcE(-1)\subset \mcT$.
Now denote by $\mcS\subset \mcT$ the vector subbundle of real rank $2$ given by the direct sum of these two line bundles. By construction, $\mcS$ is preserved under the action of $\bbJ$, making $\mcS$ into a complex
vector subbundle of $\mcT$. We claim that $\mcS$ descends (via $\mathrm{proj}$) to a vector subbundle $\tilde\mcS$ of $\wtmcT$. By definition of $\mcT$, this amounts to showing that parallel transport
along the flow lines preserves $\mcS$, which is equivalent to $\nabla_k^{\mcT}$ preserving $\Gamma(\mcS)$. With respect to the choice of a scale, we have $\nabla_aX^A=W^A{}_{a}$ and
by \eqref{equation:adjoint-decomposition} $K^A:=\bbJ_{B}{}^{A}X^B=k^a W^A{}_{a}+\alpha X^A$. Therefore one gets $\nabla_k X^A=K^A-\alpha X^A$ and $\nabla_k K^A=\bbJ_{B}{}^A\nabla_kX^A=-X^A-\alpha K^A$,
which implies that $\nabla^{\tilde\mcT}_k$ preserves sections of $\mcS$. \qedhere
\end{enumerate}
\end{proof}

\begin{proposition}
The pair $(\wtmcT, \nabla^{\wtmcT})$ constructed in Lemma \ref{c-proj_tractor_bundle_lemma} coincides with the canonical (standard) c-projective tractor bundle with its normal tractor connection.
\end{proposition}
\begin{proof}
Set $\tilde\mfg:=\mathfrak{sl}(m+1,\bbC)$ and let $\wtP$ be the stabiliser in $\wtG:=\SL(m+1,\bbC)$ of the complex line $\ell$ generated by the first basis vector in $\bbC^{m+1}$. Moreover, denote by
$\tilde\mfg_{\bbR}$, $\wtP_\bbR\subset \wtG_\bbR$ the underlying real Lie algebra respectively groups. By the abstract characterisation of tractor bundles and tractor connections in \cite{CapGoTAMS} the c-projective (standard) tractor bundle with its normal tractor connection
is (up to natural equivalence) the unique vector bundle with structure group $\wtP_{\bbR}$ and standard fibre the $(\tilde \mfg_\bbR, \wtP_\bbR)$-module $\mathbb V_{\bbR}:=\bbR^{2m+2}$ (which we view as the underlying real vector space of $\mathbb V:=\bbC^{m+1}$) and
normal nondegenerate $\tilde\mfg_{\bbR}$-connection. Hence, we may prove this proposition by verifying all these properties for our pair $(\wtmcT, \nabla^{\wtmcT})$.

\textbf{Claim 1}: \textit{$\wtmcT$ is a vector bundle with structure group $\wtP_\bbR$ and standard fibre the $(\tilde \mfg_\bbR, \wtP_\bbR)$-module $\mathbb V_\bbR$.}

By construction $\wtmcT$ has (real) rank $2m+2$ and so its standard fibre can be identified with the vector space $\mathbb V_{\bbR}$. The claim is equivalent to the fact that the frame bundle of
$\wtmcT$ admits a reduction to $\wtP_{\bbR}\subset\GL(\mathbb V_{\bbR})$, which follows immediately from the existence of the complex structure $\tilde\bbJ$, the complex volume form $\tilde{\vol}_{\bbC}$,
and the filtration $\tilde\mcS\subset \wtmcT$ of complex vector bundles.

\textbf{Claim 2}: \textit{$\nabla^{\wtmcT}$ is a $\tilde\mfg_{\R}$-connection as defined in \cite{CapGoTAMS}.}

The projective tractor connection $\nabla^{\mcT}$ is the unique normal nondegenerate $\mathfrak{sl}(2m+2,\bbR)$-connection on $\mcT$. Since by assumption
the holonomy of $\nabla^{\mcT}$ reduces to $\tilde G_\bbR$, the structure group of $\mcT$ reduces from $P$ further to $P\cap \wtG_\R$ and $\nabla^{\mcT}$
to a $\tilde\mfg_{\R}$-connection on $\mcT$. Hence, the claim follows from the definition of $\nabla^{\wtmcT}$.

\textbf{Claim 3}: \textit{$\nabla^{\wtmcT}$ is nondegenerate (with respect to the filtration $\tilde\mcS\subset\wtmcT$).}

Write $\Theta^{\mcT}: \mcT\rightarrow \mcT/\mcS$ and $\Pi^{\wtmcT}: \wtmcT\rightarrow \wtmcT/\tilde\mcS$ for the natural projections.
Note that mapping a pair $(\tilde\xi, \tilde s)\in\Gamma(T\wtM)\times \Gamma(\tilde\mcS)$ to $\Pi^{\wtmcT}(\nabla^{\wtmcT}_{\tilde \xi} \tilde s)$ induces a vector bundle map
$$\tilde\Psi: T\wtM\rightarrow \Hom(\tilde\mcS, \wtmcT/\tilde\mcS),$$
which has values in the subspace $\Hom_{\tilde\bbJ}(\tilde\mcS, \wtmcT/\tilde\mcS)$ of $\tilde\bbJ$-linear vector bundle maps from $\tilde\mcS$ to $\wtmcT/\tilde\mcS$,
since $\nabla^{\wtmcT}\tilde \bbJ=0$. Nondegeneracy of $\nabla^{\wtmcT}$ means that $\tilde\Psi$ is injective, in which case rank considerations imply that it induces a vector bundle isomorphism
\begin{equation}\label{filtration_of_c_proj_tractor}
\tilde\Psi:  T\wtM\cong \Hom_{\tilde\bbJ}(\tilde\mcS, \wtmcT/\tilde\mcS).
\end{equation}
Note that $\Theta^{\mcT}(\nabla^{\mcT}_\xi s)$ induces a vector bundle map $\Psi: TM/\langle k \rangle\rightarrow \Hom_{\bbJ}(\mcS, \mcT/\mcS)$ and evidently $\tilde\Psi$ being injective is equivalent
to $\Psi$ being injective (and hence an isomorphism). Fix now $\xi\in\Gamma(TM)$ and a section $s^A=\rho X^A+\lambda K^A$ of $\mcS$, where $\rho,\lambda\in\Gamma(\mcE(-1))$. Then with respect to a scale
we have
\begin{align*}
\nabla_\xi^{\mcT} s^A&=(\nabla_\xi \rho) X^A+\rho \nabla_\xi^{\mcT} X^A+(\nabla_\xi\lambda) K^A +\lambda\nabla_\xi^{\mcT} K^A\\
&\equiv(\rho W^A{}_{a}+ \lambda\phi^b{}_{a} W^A{}_{b})\xi^a\pmod\mcS,
\end{align*}
where $\phi^b{}_a=\nabla_b k^a - \frac{1}{2m +2} \delta^a{}_b \nabla_c k^c$ by \eqref{equation:splitting-operator-adjoint}, which proves that $\Psi$ is injective.

\textbf{Claim 4}: $\nabla^{\wtmcT}$ is normal.

As we already noticed the tractor curvature $R^{\mcT}$ vanishes upon insertion of $k$ and hence it descends to a section $R^{\tilde{\mcT}}\in\Gamma(\Wedge^2 T^*\wtM\otimes \End(\wtmcT))$,
which evidently equals the curvature of $\nabla^{\wtmcT}$. Moreover, normality of $R^{\mcT}$ implies normality of $R^{\tilde{\mcT}}$.
\end{proof}

\begin{remark}
Let us write
$\wtmcT_\bbC:=\wtmcT\otimes_\bbR\bbC=\wtmcT^{1,0}\oplus\wtmcT^{0,1}$
for the complexification of the c-projective tractor bundle $(\wtmcT,
\tilde\bbJ)$ and its decomposition into eigenspaces of $\tilde\bbJ$.
Then $\tilde\Phi$ in \eqref{filtration_of_c_proj_tractor} can be also
seen as an isomorphism $T^{1,0}\wtM\otimes_\bbC \tilde \mcS^{1,0}\cong
\wtmcT^{1,0}/\tilde\mcS^{1,0}$, and one may show that
$\mcS^{1,0}\cong\mcE(-1,0)$. Otherwise put, $\tilde\mcS$ is
$\mcE(-1,0)$ viewed as a real vector bundle equipped with a complex
structure.
\end{remark}

\subsubsection{A brief aside on parallel tractor paracomplex structures}\label{subsubsection:paracomplex}
The results of the previous subsections have suitable analogues in the setting of an oriented projective manifold $(M, \mbp)$ of odd dimension $2 m + 1$ (restricted appropriately to $\dim M \geq 3$ or $\dim M \geq 5$) equipped with a parallel \textit{tractor paracomplex structure}; this is a section $\bbJ \in \Gamma(\mcA)$ that satisfies $\bbJ^2 = \id_{\mcT}$. At each $x \in M$ the eigenvalues of $\bbJ_x$ are $\pm 1$, and trace-freeness forces the $(+1)$- and $(-1)$-eigenbundles to have equal rank $m + 1$. The complex and paracomplex cases differ qualitatively in several ways, and we briefly outline those here; elsewhere, the sign difference between the definitions of complex and paracomplex structures leads to appropriate sign changes in some formulas.

Contracting gives $\bbJ^A{}_B X^B = k^a W^A{}_a - \frac{1}{2 m + 2}
\nabla_c k^c X^A$, where as usual $k := \Pi^{\mcA}(\bbJ)$ is the
underlying vector field. In particular $k$ vanishes at points $x$ where
$X_x$ is an eigenvector of $\bbJ_x$ (in contrast to the complex case, in
which by Proposition \ref{proposition:complex-single-curved-orbit} $k$
vanishes nowhere), and the zero locus $M^0$ of $k$ is partitioned into
two sets $M^0_{\pm}$, according to whether the eigenvalue of $X$ at
each point is $\pm 1$. In particular, if the zero locus is nonempty,
$(M, \mbp)$ admits no $k$-adapted sales.

We determine the curved orbit decomposition of the model: The stabiliser in $\SL(2 m + 2, \bbR)$ of a paracomplex structure on $\Bbb R^{2 m + 2}$ is $\textrm{S}(\GL(m + 1, \bbR) \times \GL(m + 1, \bbR))$, and the action of this group on the model space $\SL(2 m + 2, \Bbb R) / P \cong S^{2 m + 1}$, which we again may identify with the ray projectivisation $\Bbb P_+(\Bbb R^{2 m + 2}) \cong \Bbb P_+(\Bbb R^{m + 1} \times \Bbb R^{m + 1})$ of $\Bbb R^{2 m + 2} \cong \Bbb R^{m + 1} \times \Bbb R^{m + 1}$, has three orbits: The space of rays contained in the $+1$-eigenspace of $\bbJ$, the space of rays contained in the $-1$-eigenspace of $\bbJ$, and the complement of the union of these. This decomposition determines in turn the curved orbit decomposition of $(M, \mbp)$ determined by the parallel tractor complex structure, which thus consists of $M^0_{\pm}$ and the open set $M^* := M - M^0$:
\begin{itemize}
	\item On the open set $M^*$, one can locally choose $k \vert_{M^0}$-adapted scales $\nabla \in \mbp\vert_{M^*}$ and then largely proceed as in the case that $\bbJ$ is complex. Working locally, $\bbJ$ induces a paracomplex structure $\tilde J$ on any local leaf space $\wtM$ of $k\vert_{M^*}$. This structure is again compatible with a class $\tilde\mbp$ of connections induced by $\mbp\vert_{M^*}$, and $\tilde J$ and $\tilde\mbp$ together comprise a paracomplex analogue of a c-projective structure on $\wtM$.
	\item On the closed manifolds $M_0^{\pm}$, which turn out to have dimension $m$, any choice $\nabla \in \mbp$ of scale restricts to a connection on $TM^{\pm}_0$. Checking shows that the restrictions of any two choices of scale are projectively equivalent, and hence $M_0^{\pm}$ both canonically inherit a projective structure.
\end{itemize}

\subsection{Assembling the ingredients: Reduction to \texorpdfstring{$\SU(p, q)$}{SU(p, q)}}
\label{assembly}

In this section, we use the results of the previous three subsections to analyze projective structures equipped with all three holonomy reductions (which are algebraically compatible): Suppose $(M, \mbp)$ is a projective manifold of odd dimension $2 m + 1 \geq 3$ equipped with a parallel tractor Hermitian structure. Such a structure, which is equivalent to a holonomy reduction of the normal tractor $\nabla^{\mcT}$ to $\U(p, q)$, $p + q = 2 m + 2$, is equivalent to a (compatible) triple of reductions of holonomy to $\SO(2 p, 2 q)$, $\Sp(2 m + 2, \bbR)$, $\GL(m + 1, \bbC)$.\footnote{In fact, as in Section \ref{J-red-sect}, to $\SL(m + 1, \bbC) \times \U(1) \subset \GL(m + 1, \bbC)$.} Concretely, these reductions correspond respectively to a parallel tractor metric $h_{AB} \in \Gamma(S^2 \mcT^*)$, a parallel tractor symplectic structure $\Omega_{AB} \in \Gamma(\Wedge^2 \mcT^*)$, and a parallel complex structure $\bbJ^A{}_B \in \Gamma(\mcA)$, where any two of these determine the third, via the identity
\begin{equation}\label{equation:compatibility}
	\Omega_{AB} = h_{CA} \bbJ^C{}_B
\end{equation}

If $M$ is simply connected, Proposition \ref{proposition:complex-single-curved-orbit}(b) immediately gives that such a reduction automatically entails a further reduction to $\SU(p, q)$, equivalently, the existence of a parallel complex tractor volume form $(\vol_{\bbC})_{A_1 \cdots A_{m + 1}} \in \Gamma(\Lambda_\bbC^{m+1} \mcT^*)$.

\subsubsection{Proofs of Theorems B and C}
\label{subsubsection:proof-theorem-b}

\begin{proof}[Proof of Theorem B]
By Section \ref{subsection:holonomy-reductions-projective} the curved orbits determined by the holonomy reduction of $\nabla^{\mcT}$ to $\U(p, q)$ are parameterised by the orbit decomposition of the action of that group on the model space $S^{2 m + 1} = \bbP_+(\bbR^{2 m + 2}) \cong \SL(2 m + 2, \bbR) / P$. By \cite[Theorem 3.1(b)]{Wolf} the orbits of $\SU(p, q)$ coincide with those of the group $\SO(h) \cong \SO(2p, 2q)$ preserving $h$, and hence so do the orbits of the intermediate group $\U(p, q)$: So, by Theorem \ref{theorem:orthogonal-reduction} the curved orbit decomposition determined by $\U(p, q)$ or $\SU(p, q)$ has one curved orbit, $M_+$ if $q = 0$, or $M_-$ if $p = 0$, and otherwise it has three curved orbits $M_+, M_0, M_-$, corresponding to the strict sign of $\tau := h_{AB} X^A X^B$.

\begin{enumerate}
\item On the open orbits $M_{\pm}$, $\tau\vert_{M_{\pm}}$ determines a connection
  $\nabla^{\pm}\in\mbp\vert_{M_{\pm}}$, which, by Theorem \ref{theorem:orthogonal-reduction},
  is the Levi-Civita connection for the metric
  $g^{\pm}_{ab} := \Rho^{\pm}_{ab}$; in particular, $g^{\pm}$ is
  Einstein and the claimed signatures follow from Theorem
  \ref{theorem:orthogonal-reduction}. For readability we henceforth
  sometimes suppress the superscript ${}^{\pm}$ and notation for
  restriction to $M_{\pm}$.
Since $\bbJ^A{}_B$ is parallel, $\bbJ^A{}_B = L^{\mcA}(k)^A{}_B$, where $k^a := \Pi^{\mcA}(\bbJ)^a$. Expanding the compatibility condition $h_{C(A} \bbJ^C{}_{B)} = 0$ using the decomposition \eqref{tractor_metric} of $h_{AB}$ in the scale $\tau$ and formula \eqref{equation:splitting-operator-adjoint} for the splitting operator $L^{\mcA}$ gives
\[
	-\tfrac{1}{2 m + 2} \tau \nabla_e k^e Y_A Y_B
		+ \tau (\Rho_{e(a} \nabla_{b)} k^e - \tfrac{1}{2 m + 2} (\nabla_e k^e) \Rho_{ab}) Z_A{}^a Z_B{}^b
			= 0 .
                        \]
Contracting both sides with $X^A X^B$ gives that $\nabla_c k^c = 0$ (so that, incidentally, $\nabla$ is $k$-adapted). Then instead contracting with $W^A{}_c W^B{}_d$ gives $$0=\Rho_{e(c} \nabla_{d)} k^e =g_{e(c} \nabla_{d)} k^e,$$
which shows that $k$ is a Killing field of $g$.
Moreover, since $(h, \bbJ)$ is a
Hermitian structure, we have $\tau = h_{AB} X^A X^B = h_{AB}\bbJ ^A{}_C
X^C \bbJ^{B}{}_D X^D$. On the other hand, the formula \eqref{equation:splitting-operator-adjoint} for the splitting operator  $L^{\mcA}$ implies $(\bbJ X)^A = \bbJ^A{}_B X^B =
k^a W^A{}_a$ (with respect to the scale $\tau$). Hence, $\tau = \tau g_{ab} k^a k^b$, which implies that $1=g_{ab}k^ak^b=\Rho_{ab}k^ak^b$. By Theorem \ref{Jpar},
in order to show that $(M_\pm, g^\pm, k)$ is Sasaki--Einstein, it therefore remains to verify that $W_{ab}{}^c{}_{d}k^d=0$. The latter however immediately follows from the fact
that $k$ is a normal solution of the adjoint BGG operator and hence satisfies \eqref{equation:adjoint-normality-conditions}.

\item Theorem \ref{theorem:orthogonal-reduction}(b) establishes the existence of a conformal structure $\mbc$ on $M_0$. Like a projective structure, a conformal structure on a manifold $M_0$ of dimension $n\geq 3$
admits a canonical connection, called the \emph{normal conformal tractor connection}, on a certain vector bundle $\mcT_0$ over  $M_0$, which in the conformal case is of rank $\dim(M_0)+2=n+2$, see e.g.\,\cite{BEG}.
By \cite[Theorem 3.2]{CGHjlms}, we may canonically identify the restriction of the (projective) tractor bundle $\mcT$ to $M_0$ with the conformal tractor bundle $\mcT_0$ for $\mbc$, and the normal projective tractor connection $\nabla^{\mcT}$ of $(M, \mbp)$ restricts to a connection on $\mcT_0$, which coincides with $\nabla^{\mcT_0}$. Thus, we may regard the restriction $\bbJ \vert_{M_0}$ as a $\nabla^{\mcT_0}$-parallel complex structure on $\mcT_0$, and by \cite[Theorem 2.5]{CGFeffchar} or \cite{Leitner}, the existence of such a complex structure characterises the conformal structures that arise from the classical Fefferman construction.
\qedhere
\end{enumerate}
\end{proof}

\begin{remark}
In the setting of part (a) of Theorem B (in particular working in the scale $\tau$) contracting both sides of \eqref{equation:compatibility} with $\tau^{-1} X^A W^B{}_b$ gives that the respective projecting parts $\tau$, $\hat k_a \in \Gamma(T^*M(2))$, $k \in \Gamma(TM)$ of $h_{AB}$, $\Omega_{AB}$, $\bbJ^A{}_B$ are related by $\tau^{-1} \hat k_b = \Rho_{bc} k^c$. The left-hand side is the trivialization of $\hat k_b$ with respect to the scale $\tau$, and the right-hand side is just $k_b$---where we have lowered the index with the metric $g_{ab} = \Rho_{ab}$---justifying in our setting the use of the (undecorated) symbol $k$ for both objects.
\end{remark}

\begin{proof}[Proof of Theorem C]
This follows immediately from Theorem B and Theorem \ref{theorem:orthogonal-reduction}(c).
\end{proof}

\subsubsection{The K\"ahler--Einstein structures on the open leaf spaces $\wtM_{\pm}$}
\label{sec-Kaehler-Einstein}
Now fix a projective structure $(M, \mbp)$ of (odd) dimension $2 m + 1 \geq 3$ equipped with a parallel tractor Hermitian structure $(h, \bbJ)$. By \hyperref[b]{Theorem B} the open curved orbits determined by that structure are Sasaki manifolds $(M_{\pm}, g^{\pm}, k\vert_{M_{\pm}})$, where $k = \Pi^{\mcA}(\bbJ)$. In the Sasaki setting the integral curves of $k$ (we again suppress restriction notation) are geodesics (cf. \eqref{equation:complex-structure-components-k-adapted}), and the foliation of $M_{\pm}$ by those geodesics is called the \textit{Reeb foliation}. It is well known \cite[\S1.2]{Sparks} that the Sasaki structure determines a K\"ahler structure on any sufficiently small leaf space $\wtM_{\pm}$ of this foliation (this specialises Theorem \ref{theorem:projective-to-c-projective} to the Sasaki setting), and that if the Sasaki structure is Einstein, so is the K\"ahler structure; we briefly recover these latter facts here.

Let $\pi : M_{\pm} \to \wt M_{\pm}$ denote the canonical projection that maps a point to the integral curve of $k$ through it. Since $k$ is Killing, it preserves $g$ and hence the contact distribution given by the kernel of  $g_{ab}k^b$. Thus, $g$ descends to a smooth metric $\tilde g^{\pm}$ on the local leaf space $\wtM_{\pm}$ characterised by $\tilde g_{\pi(x)}(T_x\pi \cdot \xi, T_x\pi \cdot \eta) = g_x(\xi, \eta)$ for all $x \in M$ and $\xi, \eta \perp k_x$. Since $k$ is Killing, $J^a{}_b := \nabla_b k^a$ is $g$-skew, and hence its descent $\tilde J^{\pm}$ to $M_{\pm}$ defined in Section \ref{subsubsection:c-projective-leaf-space}, is $\tilde g$-skew; thus $(M_{\pm}, \tilde g^{\pm}, \tilde J^{\pm})$ is Hermitian.

Again because $k$ is Killing, $\omega_{ab} := g_{ac} J^c{}_b =
\nabla_b k_a$ is exact and so $d\omega = 0$. Moreover, $\omega =
\pi^*\widetilde\omega$, where $\widetilde\omega = \tilde g(\,\cdot\,,
\tilde J\,\cdot\,)$, so $d\tilde\omega = 0$, too. So, $(\wtM_{\pm},
\tilde g^{\pm}, \tilde J^{\pm})$ is K\"ahler.

Finally, a straightforward computation using the fact that the Ricci tensor of $g$ is $\Ric = 2 m g$, the definition of $\wt g$, the characterization \eqref{equation:Rho-tilde-characterization} of the c-projective Rho tensor $\wt\Rho$, and the identities \eqref{equation:Schouten-k-k} and \eqref{equation:Schouten-H-H} gives that $\wt g$ is (K\"ahler--)Einstein with Ricci curvature $\widetilde\Ric = (2 m + 2) \tilde g$.

\subsubsection{The CR structure on a local leaf space $\wt M_0$}
\label{subsubsection:compatibility-compactifications}
In addition to the assumptions of Section \ref{sec-Kaehler-Einstein}, we assume now furthermore that $M_0 \neq \emptyset$. Again we let $\pi : M \to \wt M$ denote the canonical projection onto the leaf space of $k$, and if necessary we replace $M$ with a sufficiently small open subset of an arbitrary point (in $M_0$). Note that under these assumptions $k$ is automatically a projective symmetry of $(M,\mbp)$, since it is one on the open dense subset $M\setminus M_0$ (note that on $M_\pm$ it is a Killing field of a metric compatible with $\mbp$).

We have seen in Section \ref{section_proj_cproj_tractors} that the normal projective tractor connection of $(M,\mbp, \mathbb{J})$ descends to the normal c-projective tractor connection of the induced c-projective structure $(\tilde J,\tilde\mbp)$ on $\wt M$. Hence, the parallel Hermitian metric $h$ of the tractor bundle of $(M,\mbp)$ induces a parallel Hermitian metric $\tilde{h}$ on the (standard) c-projective tractor bundle
$(\widetilde \mcT, \tilde{\mathbb J})$ of $(\wt M,\tilde J,\tilde\mbp)$. The metric $\tilde h$ determines a curved orbit decomposition $$\wtM=\wtM_+ \cup \wtM_0 \cup \wtM_-$$ of the c-projective manifold $(\wtM,\tilde J,\tilde\mbp)$---analogous to those described for projective structures in Section \ref{subsection:holonomy-reductions-projective}, instead taking $\wtG := \SL(m + 1, \bbC)$ and $\wtP$ as in Section \ref{subsubsection:c-projective-primer}---according to the strict sign of the projecting part of $\tilde h$, which is a section $\tilde\tau \in \Gamma(\mcE(1, 1))$; see \cite[Section 3.3]{CGH}. For $\tau=h(X,X)$ we have by construction $\tau = \pi^* \wt\tau$. Hence $\wtM_{\pm} = \pi(M_{\pm})$, respectively, and $\wtM_0 = \pi(M_0)$, which agrees with the use of those symbols in the introduction and the previous subsection.

By \cite[Theorem 3.3(2)]{CGH}, $\wtM_0$ is a smooth (real) hypersurface separating $\wt M_+$ and $\wt M_-$ in the complex manifold $(\wt M, \tilde J)$. As such it is naturally equipped with a CR structure of hypersurface type, that is, a distribution $\wt H_0\subset T\wtM_0$ of corank $1$ equipped with a complex structure with vanishing Nijenhuis tensor: The distribution is given by $$\wt H_0 := T \wt M_0 \cap \tilde J(T \wt M_0) \subset T \wt M_0$$ and the complex structure by $\tilde J_0 := \tilde J\vert_{\wt H_0}$. Moreover, Theorem 3.3 of \cite{CGH} shows that $\wt H_0$ is in fact a contact distribution and hence $(\wtM_0, \wt H_0, \tilde J_0)$ a so-called nondegenerate CR structure of hypersurface type, which proves the first claim of statement (a) in Theorem D. Statement (b) follows immediately from \cite{CG-cproj} (see Section 3).

\begin{proof}[Proof of Theorem D]
It only remains to prove the second statement of $(a)$.  We already
mentioned in the proof of Theorem $B$ that the restriction of the
tractor bundle $\mcT$ to $M_0$ can be identified with the standard
conformal tractor bundle of $(M_0, \bold c)$. Let us say a bit more
about this identification: Since $\tau=h_{AB}X^AX^B$ is zero on $M_0$,
the canonical tractor $X$ is null on $M_0$ and so $\mcT$ admits a
finer filtration (than \eqref{euler}) on $M_0$ given
by $$\mcT\vert_{M_0} \supset \mcT^0 \supset \mcT^1,$$ where $\mcT^1$
is the restriction of the line bundle $X\mcE(-1)\subset\mcT$ to $M_0$
and $\mcT^0:=(\mcT^1)^\perp$ its orthogonal complement (with respect
to $h$). This is the canonical filtration of the conformal
tractor bundle $\mcT_{\textrm{conf}}=\mcT\vert_{M_0}$ (as in \cite{BEG}) with
$\mcT^1\cong \mcE[-1]$ and $\mcT^0/\mcT^1\cong TM_0[-1]$ (see
\cite{CGHjlms, CGH}).
Moreover, note that $h$ induces a bilinear form on
$\mcT^0/\mcT^1\cong TM_0[-1]$, equivalently a nondegenerate section
$\bold g\in\Gamma(S^2 T^*M_0[2])$, which defines the conformal
structure on $M_0$ induced by $h$. Since $h$ is Hermitian,
$K^A=\mathbb J^{A}{}_BX^B$ is orthogonal to $X^A$, which shows that
$\mathcal U:=K^A\mcE(-1)\vert_{M_0}$ lies in $\mcT^0$. In particular,
the restriction of the projecting part $k^a:=\mathbb J^{A}{}_BX^B
Z_A{}^a=K^AZ_A{}^a$ of $\mathbb J^{A}{}_B$ to $M_0$ is tangent to
$M_0$. Moreover, $k\vert_{M_0}\in\Gamma(TM_0)$ is null with respect to
the conformal structure, since, on $M_0$,
$$\bold g(k,k)=h(K,K)=h(X,X)=\tau$$ vanishes, and by construction,
$k\vert_{M_0}$ is evidently the conformal Killing vector
field arising in the characterisation of conformal Fefferman
structures in \cite{CGFeffchar}. Now $\bold{g}(k,\cdot)$ (which can
also be identified with the restriction to $M_0$ of the projecting
part $k_b \in \Gamma(T^*M(2))$ of $\Omega_{AB}$) defines a nowhere
vanishing weighted $1$-form on $TM_0/\langle k\rangle$. Note that its
kernel $H_0$ can be also identified with $\mathcal U^{\perp}\cap
(\mcT^1)^{\perp}/(\mathcal U\oplus \mcT^1)$, which shows that $H_0$
inherits from the complex structure $\mathbb J\vert_{M_0}$ on
$\mcT_{\textrm{conf}}=\mcT\vert_{M_0}$ a complex structure $J_0$,
since $\mathbb J\vert_{M_0}$ preserves the subbundle $\mathcal U\oplus
\mcT^1$ (cf. also Section \ref{section_proj_cproj_tractors}). By
Section 4.5 of \cite{CapGoverCRTractors}, it is the pair $(H_0, J_0)$
that descends via $\pi$ to the CR structure $(\wt H_0', \tilde J_0')$
on $\wtM_0$ underlying the Fefferman conformal structure on
$M_0$. Note now that being induced by $\mathbb J$ the complex
structure $J_0$ must coincide with the restriction to $H_0$ of the
complex structure $J$ on $TM/\langle k\rangle$ from Remark
\ref{complex_str_on_TM/k}. Hence, $(TM_0 / \langle k \rangle) \cap
J(TM_0 /\langle k \rangle)$ must contain $H_0$ and so by dimensional
reason
$$H_0 = (TM_0 / \langle k \rangle) \cap J(TM_0 /\langle k \rangle).$$ Hence, $(\wt H_0, \tilde J_0)=(\wt H_0', \tilde J_0')$ as claimed.
\end{proof}

\subsubsection{Sasaki--Einstein strucures in dimension $5$ via conformal holonomy reduction}
In dimension $5$ one can also realise Sasaki--Einstein geometry of
signature $(3, 2)$ via an exceptional holonomy reduction of the normal
tractor connection $\nabla^{\mcT}_{\textrm{conf}}$ associated to an
(oriented) \textit{conformal} geometry. For a conformal structure $(M,
\mbc)$ of signature $(2, 3)$ with conformal metric $\mbg \in \Gamma(S^2 T^*M[2])$, the standard tractor bundle
$\mcT_{\textrm{conf}}$ is equipped with a canonical parallel tractor
metric $h_{\textrm{conf}} \in \Gamma(S^2 \mcT_{\textrm{conf}}^*)$ of
signature $(3, 4)$ \cite{BEG}.

As for conformal structures of general dimension $n \geq 3$, a
nonisotropic parallel section $I^A \in \Gamma(\mcT_{\textrm{conf}})$,
with projecting part, say, $\sigma \in \Gamma(\mcE[1])$, determines an
Einstein metric $g := \sigma^{-2} \mbg\vert_{M \setminus M_0}$
on the
complement of the zero locus $M_0$ of $\sigma$ (itself a smooth
hypersurface). If in signature $(2, 3)$ such an $I$ is spacelike, it
determines a holonomy reduction of $\nabla^{\mcT}_{\textrm{conf}}$ to
$\SO(2, 4)$ and a conformal structure $\mbc_0$ of signature $(1, 3)$
on $M_0$ \cite[\S3.5]{CGHjlms}, \cite{Gal}.

On the other hand, exclusive to this signature is an exceptional
canonical inclusion $\G_2 \hookrightarrow \SO(3, 4)$, where $\G_2$ is a split real form of the simple Lie group of type $\G_2$, and a concomitant
exceptional geometry: An oriented conformal structure of signature
$(2, 3)$ admits a holonomy reduction to $\G_2$ (equivalently a
parallel tractor $3$-form $\Phi_{ABC} \in \Gamma(\Wedge^3
{\mcT}_{\textrm{conf}}^*)$ suitably compatible with
$h_{\textrm{conf}}$) if and only if it arises via a construction of
Nurowski (itself a Fefferman-type construction) that canonically
associates to any generic $2$-plane distribution on a $5$-manifold a
conformal structure of that signature \cite[\S5]{HammerlSagerschnig},
\cite[\S5]{Nurowski}.

If a conformal structure of this signature admits both of these types
of parallel tractor objects, it determines a holonomy reduction of
$\nabla^{\mcT}_{\textrm{conf}}$ to $\SO(2, 4) \cap \G_2 = \SU(1,
2)$. The projecting part of the parallel adjoint tractor
$(\bbJ_{\textrm{conf}})^A{}_B := h^{AC} \Phi_{CDB} I^D$ is a vector
field $k \in \Gamma(TM)$. In the conformal version of the curved orbit
decomposition (analogous to the projective development in Section
\ref{subsection:holonomy-reductions-projective}, replacing there the
groups $G$ and $P$ respectively with $\SO(3, 4)$ and the stabiliser of
an isotropic ray therein), this realises the open curved orbits
$M_{\pm} := \{ \pm \sigma > 0 \}$ as Sasaki manifolds $(M_{\pm},
-g\vert_{M_{\pm}}, k\vert_{M_{\pm}})$. The restriction of
$\nabla^{\mcT}_{\textrm{conf}}$ to $M_0$ coincides with the normal
tractor connection of the conformal structure $\mbc_0$, which thus
enjoys a holonomy reduction to $\SU(1, 2)$ and so $\mbc_0$ locally
arises from a $3$-dimensional CR structure via the classical Fefferman
construction; $k\vert_{M_0}$ is the infinitesimal generator of the
$S^1$-action. This joint reduction is discussed in detail in
\cite[\S5]{SagerschnigWillse}.

This again realises $4$-dimensional Fefferman conformal structures as natural compactifying structures for Sasaki--Einstein structures of signature $(3, 2)$. In this setting, however, the compactification is mediated by conformal geometry rather than by projective geometry, and consequently the resulting compactifications are strictly different and have different asymptotics. (One can see this quickly by observing that the conformal structures $[g^{\pm}]$ determined by the metrics $g^{\pm}$ arising from the projective holonomy reduction do not extend across the zero locus $M_0$, whereas $[g]$ obviously extends to $\mbc$.) For more about these compactifications see \cite[\S3.3]{CGHjlms}.

Away from the hypersurface these perspectives are readily interchangeable: Given a Sasaki--Einstein structure $(M, g, k)$ of signature $(3, 2)$, one can form both the projective tractor bundle $\mcT$ associated to the projective structure $[\nabla^g]$ and the conformal tractor bundle $\mcT_{\textrm{conf}}$ associated to the conformal structure $[g]$. Then, $I$ determines a canonical identification $\mcT_{\textrm{conf}} \cong \mcT \oplus \langle I \rangle$; $\nabla^{\mcT}_{\textrm{conf}}$ preserves $\mcT$ and so restricts to a connection thereon, and this coincides with the normal projective tractor connection $\nabla^{\mcT}$ \cite[\S8.2]{GoverMacbeth}. Via this identification $h_{\textrm{conf}}$ and the parallel projective tractor metric $h$ determined by $g$ are related by $h_{\textrm{conf}} = h - I^{\flat} \otimes I^{\flat}$, where $\cdot{}^\flat$ denotes lowering of an index by $h_{\textrm{conf}}$, and the parallel (conformal) adjoint tractor $\bbJ_{\textrm{conf}}$ restricts to the parallel (projective) tractor complex structure $\bbJ$ on $\mcT$.

\subsubsection{Parallel tractor para-Hermitian structures}
We briefly outline what happens in the paracomplex analogue of our setting, or more precisely, when the standard tractor bundle of a projective structure $(M, \mbp)$ of odd dimension $2 m + 1 \geq 3$ (and again restricting to $2 m + 1 \geq 5$ as appropriate) is equipped with a para-Hermitian structure $(h, \bbJ)$ parallel with respect to the normal tractor connection $\nabla^{\mcT}$, specializing the discussion in Section \ref{subsubsection:paracomplex}. Here, a \textit{tractor para-Hermitian structure} is a tractor metric $\smash{h \in \Gamma(S^2 \mcT^*)}$ together with an $h$-skew para-complex structure $\bbJ$; this skewness forces $h$ to have neutral signature $(m + 1, m + 1)$.

Such a structure determines a holonomy reduction of $\nabla^{\mcT}$ to
the common stabiliser of $h$ and $\bbJ$ in $\SL(2 m + 2, \bbR)$,
namely, $\GL(m + 1, \bbR)$, the paracomplex analogue of the unitary
group. The argument in Proposition
\ref{proposition:complex-single-curved-orbit}  applies mutatis mutandis
to this case to establish (on any simply connected open subset) a
further reduction of the holonomy of $\nabla^{\mcT}$ to the stabiliser
of such a form in $\GL(m + 1, \bbR)$, namely, $\SL(m + 1, \bbR)$, the
paracomplex analogue of the special unitary group.

We briefly describe the curved orbit decomposition determined by this stronger holonomy reduction. Section \ref{subsubsection:paracomplex} showed that the parallel paracomplex structure $\bbJ$ alone determines a nontrivial decomposition $M = M^0_+ \cup M^* \cup M^0_-$ of the underlying space into curved orbits, where $M^0_{\pm}$ respectively consist of the points $x$ where $X_x$ is an eigenvector of $\bbJ_x$ of eigenvalue $\pm 1$. Thse turns out to interact in a simple way with the curved orbit decomposition $M = M_+ \cup M_0 \cup M_-$ determined by $h$ as in Theorem \ref{theorem:orthogonal-reduction}: At $x \in M^0_{\pm}$, we have $h(X_x, X_x) = -h(\bbJ X_x, \bbJ X_x) = -h(\pm X_x, \pm X_x) = -h(X_x, X_x)$ (the first equality is a consequence of the $h$-skewness of $\bbJ$) so $h(X_x, X_x) = 0$ and thus the curved orbits $M^0_{\pm}$ are subsets of the zero locus $M_0$ of $\tau := h(X, X)$. By analyzing the action of $\SL(m + 1, \bbR)$ on the model space $\bbP_+(\Bbb R^{2 n + 2}) \cong \bbP_+(\Bbb R^{m + 1} \times \Bbb R^{m + 1})$ we conclude that the considerations above account for all of the curved orbits.

So, the curved orbit decomposition of $(M, \mbp)$ determined by the reduction to that group is $M = M_+ \cup M_0^+ \cup M_0^* \cup M_0^- \cup M_-$ (here, $M_0^* := M_0 \cap M^*$):
\begin{itemize}
	\item On the open curved orbits $M^{\pm}$, the behavior is similar to that in the complex setting: These curved orbits respectively inherit para-Sasaki--Einstein structures $(M_{\pm}, g^{\pm}, k\vert_{M_\pm})$, and these spaces locally fibre along the integral curves of $k$ over a para-K\"ahler--Einstein manifold. (Para-Sasaki structures are defined as Sasaki structures are in Definition \ref{definition:Sasaki-structure} except for a reversal of sign in condition (b), so that $k$ satisfies $\nabla_a \nabla_b k^c = g_{ab} k^c - \delta^c{}_a k_b$.  A para-K\"ahler structure is defined as a K\"ahler structure is, except that one replaces the complex structure in the hypotheses with a paracomplex structure.)
	\item On the hypersurface $M_0^* \subset M$, the restriction $\smash{\mbc\vert_{M_0^*}}$ of the conformal structure $\mbc$ defined on $M_0$ arises (locally) via a paracomplex analogue of the classical Fefferman construction, which canonically assigns to any $(2 m - 1)$-dimensional Lagrangean contact structure, $m \geq 2$, an $\SO(1, 1)$-bundle over its underlying manifold and a conformal structure of neutral signature $(m, m)$ on its $(2 m)$-dimensional total space.
	\item The closed manifolds $M_0^{\pm}$, which again have dimension $m$ (and hence are in the closure $\smash{\overline{M_0^*}}$), do not inherit from the reduction to $\SL(m + 1, \bbR)$ intrinsic structure beyond the projective structure determined by $\bbJ$ alone (see Section \ref{subsubsection:paracomplex}). In particular, $M^{\pm}$ are maximal totally isotropic submanifolds with respect to $\mbc$.
\end{itemize}

\section{Realisation}\label{section:examples}

By Theorem A any Sasaki--Einstein manifold $(M, g, k)$, where $g$ has signature, say, $(2p - 1, 2 q)$, determines (on any simply connected open subset of $M$) a holonomy reduction to $\SU(p, q)$ of the normal tractor connection of the projective structure $[\nabla^g]$ determined by the Levi-Civita connection $\nabla^g$ of $g$. Since Sasaki--Einstein manifolds are abundant in all odd dimensions at least $5$ \cite[Corollaries B, C, E]{BoyerGalicki}, \cite{Sparks}, so are projective structures admitting such holonomy reductions in those dimensions. By construction, for holonomy reductions that arise this way $M = M_{\pm}$ and in particular the hypersurface curved orbit (where the conformal structure arises) is empty.

On the other hand, some of the most interesting behavior of a projective structure $(M, \mbp)$ equipped with a special unitary holonomy reduction of the normal tractor connection occurs on or along the hypersurface curved orbit. We thus here establish the existence of projective structures so equipped for which the hypersurface curved orbit is nonempty (and which are not locally equivalent to the flat model). Among other things, this justifies our investigation of that behavior.

In fact, there are many examples: To any CR structure $(\wtM_0, \widetilde H_0, \tilde J_0)$ of hypersurface type whose Levi form is nondegenerate (say, of signature $(p - 1, q - 1)$) one can associate a scalar CR invariant $\mcO$ called the \textit{CR obstruction}.  We show in Section \ref{subsection:existence} that for any embedded, real-analytic CR structure for which the CR obstruction vanishes---again, there are many such CR structures---one can freely specify an additional datum and produce a projective structure $(M, \mbp)$ and a holonomy reduction of $\mbp$ to $\SU(p, q)$ (both of which depend on the extra datum) for which the hypersurface curved orbit is the total space of the Fefferman bundle $M_0 \to \wtM_0$ for $(\wtM_0, \widetilde H_0, \tilde J_0)$ and the conformal structure induced thereon by the curved orbit decomposition is the canonical Fefferman conformal structure; in particular the CR structure underlying it is the CR structure one started with.

\subsection{Constructing examples with nonempty hypersurface curved orbit}
\label{subsection:existence}
Given an embedded, real-analytic CR structure as above, the existence of a suitable $\nabla^h$ on an appropriate bundle admitting this holonomy reduction is essentially established in \cite[Proposition 4.9]{GrahamWillse}, which we briefly follow. Using \cite[Theorem 6.5]{GPW} we translate that existence result into natively projective language by identifying $\nabla^h$ with the normal projective tractor connection of some projective geometry $(M, p)$. By construction, the CR structure arising via the resulting curved orbit decomposition will be the one that the construction starts with.

Fix a real-analytic real hypersurface $\wtM_0 \subset \Bbb C^m$, $m > 1$, for which the Levi form is nondegenerate, say, of signature $(p - 1, q - 1)$, so that it determines a CR structure $(\widetilde H_0, \tilde J_0)$ of hypersurface type on $\wtM_0$, and suppose that the CR obstruction $\mcO$ of this structure vanishes. We construct (with the choice of an additional datum) a projective structure $(M, \mbp)$ equipped with a special unitary holonomy reduction for which the hypersurface curved orbit and its conformal structure is the total space of the Fefferman bundle $M_0 \to \wtM_0$ equipped with the Fefferman conformal structure canonically associated to $(\wtM_0, \widetilde H_0, \tilde J_0)$. In this setting the Fefferman circle bundle is trivial, and we identify its total space $M_0$ with $S^1 \times \wtM_0 \subset \bbC^* \times \wtM_0 \subset \bbC^* \times \bbC^m$.

Fefferman showed \cite{Fefferman} that (1) there is a smooth defining function $u$ for $\wtM_0$ defined in a neighbourhood of $\wtM_0$ in $\Bbb C^m$ that solves the Dirichlet-type problem
\begin{equation}\label{equation:Dirichlet-problem}
	K(u) = 1 + O(u^{m + 1}), \qquad K(u) := (-1)^{j + 1} \det \begin{pmatrix}u&u_{\bar b}\\u_a&u_{a\bar b}\end{pmatrix} ,
\end{equation}
$1 \leq a, b \leq m$, and (2) such a solution $u$ is determined uniquely modulo $O(u^{m + 2})$---in this embedded setting the CR obstruction $\mcO$ can be identified with a nonzero multiple of $(\partial_u^{m + 1} K) \vert_{u = 0}$. For any solution $u$,
\begin{equation}\label{equation:Monge-Ampere-to-Kahler}
	h := \partial^2_{\alpha \bar\beta} (-|z_0|^2 u) dz^{\alpha} d\bar z^{\beta},
\end{equation}
is a K\"ahler metric on a suitable neighbourhood $A$ of $\bbC^* \times \wtM_0$ in $\bbC^* \times \bbC^m$, and the conformal structure $\mbc$ on $S^1 \times \wtM_0$ is represented by the pullback $g$ of $h$ to that subspace (this conformal structure does not depend on the solution $u$); here, $z_0$ denotes the standard coordinate on $\bbC^*$ and $0 \leq \alpha, \beta \leq m$ . The Ricci curvature of $h$ is
\begin{equation}\label{equation:CR-ambient-Ricci}
	R
		= (\partial^2_{\alpha \bar\beta} \log |{\det h}|) dz^{\alpha} d\bar z^{\beta}
		= (\partial^2_{a \bar b} \log K(u)) dz^a d\bar z^b .
\end{equation}

Now fix a real-analytic reference solution $u$ of \eqref{equation:Dirichlet-problem}. Specializing \cite[Theorem 2.11]{Graham} to the real-analytic setting and to the case that the CR obstruction vanishes gives that for each $a \in C^{\omega}(\wtM_0)$ there is a unique adjusted real-analytic solution $v$ satisfying
\[
	v = u + a u^{m + 2} + O(u^{m + 3})
		\qquad
	\textrm{and}
		\qquad
	K(v) = 1 .
\]
For any such $v$, \eqref{equation:CR-ambient-Ricci} implies that the K\"ahler structure $h$ is Ricci-flat, which shows that we may regard $h$ as an \textit{ambient metric} for $\mbc$ in the sense of \cite{GPW}\footnote{NB that the definition of ambient metric in \cite{GPW} differs from the one used in \cite{FeffermanGraham} and \cite{GrahamWillse}: In the former such metrics are required to be Ricci-flat, whereas in the latter the condition is somewhat weaker, even (for even-dimensional metrics) in the real-analytic case.} Then, Theorem 6.5 of that reference gives that (1) there is a canonically determined projective structure $\mbp$ on the quotient $M := A / \Bbb R_+$ (here, $\Bbb R_+$ acts on the $\bbC^* \times \wtM_0$ by dilation on the first factor), (2) $A$ is the projective Thomas Cone associated to $(M, \mbp)$, (3) the identification in (2) identifies the normal tractor connection $\nabla^{\mcT}$ of $\mbp$ on the standard tractor bundle $\mcT$ with the Levi-Civita connection $\nabla^h$ of $h$, and so (4) we may identify $h$ with a parallel tractor metric on $\mcT$. In particular, $h$ determines a holonomy reduction of $\nabla^{\mcT}$ to $\SO(2 p, 2 q)$ and in turn, via Theorem \ref{theorem:orthogonal-reduction}, a corresponding curved orbit decomposition; by construction the hypersurface curved orbit is the total space $S^1 \times \wtM_0$ of the canonical Fefferman bundle and the conformal structure induced there is $\mbc$.

It remains to show that $(M, \mbp)$ admits a further reduction of holonomy to $\SU(p, q)$, but this is almost immediate: Since $h$ is K\"ahler, the complex structure $\bbJ$ on $A$ given by restricting the standard complex structure on $\bbC^* \times \bbC^m$ is parallel and compatible with $h$. Via the identification in the previous paragraph, we may identify $\bbJ$ with a parallel, $h$-compatible tractor endomorphism, determining a holonomy reduction of $\nabla^{\mcT}$ to $\U(p, q)$. Finally, computing directly gives that the $(m + 1, 0)$-form $\vol_{\bbC} := (z^0)^m dz^0 \wedge dz^1 \wedge \cdots \wedge dz^m$ is parallel, and again identifying it with a parallel tractor object establishes a holonomy reduction of $\nabla^{\mcT}$ to $\SU(p, q)$.

We summarise this discussion in the following proposition:
\begin{proposition}\label{proposition:realisation}
Let $(\wt M_0, \wt H_0, \wt J_0)$ be a real-analytic embedded CR structure of codimension $1$ in $\bbC^m$ with nondegenerate Levi form for which the CR obstruction $\mcO$ vanishes. Then, there is a projective structure $(M, \mbp)$ equipped with a parallel tractor Hermitian structure $(h, \bbJ)$ and parallel tractor complex volume form $\vol_{\bbC}$ for which $(\wtM_0, \wt H_0, \wt J_0)$ is the CR manifold underlying the hypersurface curved orbit.
\end{proposition}

\begin{remark}
In fact, in Proposition \ref{proposition:realisation} we may drop the condition that the CR obstruction $\mcO$ vanishes if we are willing to relax the smoothness required of the resulting projective tractor metric $h$ and projective structure $\mbp$.

If we no longer assume that $\mcO = 0$, specializing \cite[Theorem 2.11]{Graham} to the real-analytic setting gives that, for a real-analytic reference solution $u$ to \eqref{equation:Dirichlet-problem} and function $a \in C^{\omega}(\wtM_0)$ as above, there is a unique function $v$ satisying $K(v) = 1$ and having asymptotic expansion
\[
	v \sim u \sum_{k = 0}^{\infty} \eta_k (u^{m + 1} \log |u|)^k
\]
with $\eta_0 = 1 + a u^{m + 1} + O(u^{m + 2})$ and $\eta_k \in C^{\omega}(\wtM_0)$ for all $k$. We may identify $\mcO$ with $\eta_1 \vert_{\wtM_0}$ (which does not depend on the choice of reference solution $u$), so if $\mcO \neq 0$ then $v$ is $C^{m + 1}$ but not $C^{m + 2}$. In that case, \eqref{equation:Monge-Ampere-to-Kahler} gives that $h$ is $C^{m - 1}$ but not $C^m$, and thus $\nabla^{\mcT}$, and hence $\mbp$, is $C^{m - 2}$ but not $C^{m - 1}$ (here, we say a connection is $C^{\ell}$ if in any [equivalently, every] smooth coordinate chart the Christoffel symbols of any representative connection are all $C^{\ell}$).
\end{remark}

\begin{remark}
The procedure in this subsection is similar to the one given in \cite[\S7]{GPW} for producing examples of projective $6$-manifolds equipped with a holonomy reduction to a split real form $\G_2$ of the complex simple Lie group $\G_2^{\bbC}$, but there are some significant qualitative differences.

First, in place of a CR structure, one starts with an (oriented) $2$-plane distribution on a $5$-manifold $M$ satisfying a genericity condition, to which one can again associate a conformal structure via a construction of Nurowski \cite{Nurowski}; unlike in the special unitary case here, the conformal structure is defined on same underlying manifold as the distribution, rather than the total space of some bundle over $M$ with positive-dimensional fibre. Next, since this conformal structure is defined on an odd-dimensional manifold, there is no analogue of the obstruction, and correspondingly there is no additional datum to specify analogous to the function $a$, and so a choice of real-analytic distribution alone determines an essentially unique projective structure and holonomy reduction to $\G_2$. Finally, the reduction of holonomy from the special orthogonal group to $\G_2$ is established by means of a general theorem guaranteeing for odd-dimensional conformal structures the existence of a parallel projective tractor given a parallel conformal tractor on the hypersurface curved orbit; that theorem does not apply with the same strength in case of even-dimensional conformal structures, including in the special unitary case investigated here. In particular, it is not deducible from that result alone that (and in fact it is not clear whether) Proposition \ref{proposition:realisation} still holds if one replaces in the hypothesis the embedded CR structure of codimension $1$ in $\Bbb C^m$ with an (abstract) integrable CR manifold of hypersurface type.
\end{remark}

\end{document}